\documentclass[11pt]{amsart}

\usepackage{amsmath,amssymb,amsthm,amsfonts,mathrsfs,bm}
\usepackage{amscd,stmaryrd}
\usepackage{mathtools}
\usepackage[usenames,dvipsnames]{color}
\usepackage{hyperref}
\usepackage{graphicx,subfigure}
\definecolor{nicegreen}{RGB}{0,180,0}
\hypersetup{colorlinks=true,citecolor=nicegreen,linkcolor=Red,urlcolor=blue}
\usepackage[divide={2.45cm,*,2.45cm}]{geometry}
\usepackage{enumerate}
\usepackage{color}
\usepackage[all]{xy}
\usepackage{tikz-cd}
\usepackage{wasysym}
\usepackage{adjustbox}
\usepackage{caption}
\usepackage{multirow}

\allowdisplaybreaks

\newtheorem{thm}{Theorem}[section]
\newtheorem*{thm*}{Theorem}
\newtheorem{cor}[thm]{Corollary}

\newtheorem{lemma}[thm]{Lemma}
\newtheorem{sublemma}[thm]{Sublemma}
\newtheorem{propn}[thm]{Proposition}

\newtheorem*{propn*}{Proposition}

\theoremstyle{definition}
\newtheorem{defn}[thm]{Definition}

\theoremstyle{definition}
\newtheorem{examplex}[thm]{Remark}
\newenvironment{rmk}
  {\pushQED{\qed}\examplex}
  {\popQED\endexamplex}

\newcommand{\qp}{\mathbb{Q}_p}

\newcommand{\qpb}{\overline{\mathbb{Q}}_p}

\newcommand{\zp}{\mathbb{Z}_p}

\newcommand{\rhobar}{\overline{\rho}}
\newcommand{\varrhobar}{\overline{\varrho}}
\newcommand{\rbar}{{\overline{r}}}
\newcommand{\taubar}{\overline{\tau}}
\newcommand{\BC}{\textnormal{BC}}

\newcommand{\JH}{\textnormal{JH}}
\newcommand{\sym}{\textnormal{sym}}
\newcommand{\poly}{\textnormal{poly}}
\newcommand{\der}{\textnormal{der}}

\newcommand{\soc}{\textnormal{soc}}
\newcommand{\Hom}{\textnormal{Hom}}

\newcommand{\Ext}{\textnormal{Ext}}

\newcommand{\aff}{\textnormal{aff}}

\newcommand{\inj}{\textnormal{Inj}}
\newcommand{\proj}{\textnormal{Proj}}

\newcommand{\rad}{\textnormal{rad}}
\newcommand{\pol}{\textnormal{pol}}
\newcommand{\dd}{\textnormal{dd}}
\newcommand{\Gal}{\textnormal{Gal}}
\newcommand{\rig}{\textnormal{rig}}
\newcommand{\HT}{\textnormal{HT}}

\newcommand*{\longhookrightarrow}{\ensuremath{\lhook\joinrel\relbar\joinrel\rightarrow}}
\newcommand*{\longtwoheadrightarrow}{\ensuremath{\relbar\joinrel\twoheadrightarrow}}

\newcommand{\sm}[4]{\left(\begin{smallmatrix} #1 & #2 \\ #3 & #4 \end{smallmatrix}\right)}

\newcommand{\un}[1]{\underline{#1}}

\newcommand{\cC}{{\mathcal{C}}}

\newcommand{\cE}{{\mathcal{E}}}

\newcommand{\cG}{{\mathcal{G}}}
\newcommand{\cH}{{\mathcal{H}}}

\newcommand{\cJ}{{\mathcal{J}}}

\newcommand{\cM}{{\mathcal{M}}}

\newcommand{\cO}{{\mathcal{O}}}

\newcommand{\cS}{{\mathcal{S}}}

\newcommand{\cZ}{{\mathcal{Z}}}

\newcommand{\bG}{{\mathbf{G}}}

\newcommand{\bbA}{{\mathbb{A}}}

\newcommand{\bbC}{{\mathbb{C}}}

\newcommand{\bbF}{{\mathbb{F}}}
\newcommand{\bbG}{{\mathbb{G}}}

\newcommand{\bbQ}{{\mathbb{Q}}}
\newcommand{\bbR}{{\mathbb{R}}}

\newcommand{\bbT}{{\mathbb{T}}}

\newcommand{\bbV}{{\mathbb{V}}}

\newcommand{\bbZ}{{\mathbb{Z}}}

\newcommand{\tP}{\widetilde{P}}

\newcommand{\tR}{\widetilde{R}}

\newcommand{\tW}{\widetilde{W}}

\newcommand{\tv}{\widetilde{v}}
\newcommand{\tw}{\widetilde{w}}

\newcommand{\fM}{{\mathfrak{M}}}

\newcommand{\fS}{{\mathfrak{S}}}

\newcommand{\fm}{{\mathfrak{m}}}

\newcommand{\fp}{{\mathfrak{p}}}
\newcommand{\fq}{{\mathfrak{q}}}

\newcommand{\ft}{{\mathfrak{t}}}

\newcommand{\fw}{{\mathfrak{w}}}

\newcommand{\sfK}{{\mathsf{K}}}

\newcommand{\sfa}{{\mathsf{a}}}
\newcommand{\sfb}{{\mathsf{b}}}
\newcommand{\sfc}{{\mathsf{c}}}
\newcommand{\sfd}{{\mathsf{d}}}

\newcommand{\sfx}{{\mathsf{x}}}
\newcommand{\sfy}{{\mathsf{y}}}

\newcommand{\nG}{{\textnormal{G}}}

\newcommand{\nJ}{{\textnormal{J}}}

\newcommand{\nL}{{\textnormal{L}}}

\newcommand{\nN}{{\textnormal{N}}}

\newcommand{\nP}{{\textnormal{P}}}

\newcommand{\nS}{{\textnormal{S}}}

\newcommand{\nU}{{\textnormal{U}}}

\newcommand{\nW}{{\textnormal{W}}}

\newcommand{\bbra}[1]{\llbracket #1\rrbracket}

\newcommand{\wt}[1]{\widetilde{#1}}

\newcommand{\F}{\mathbb{F}}

\newcommand{\congto}{\xrightarrow{\,\sim\,}}

\DeclareMathOperator{\Spec}{Spec}

\newcommand{\stacks}[1]{\href{https://stacks.math.columbia.edu/tag/#1}{#1}}

\begin{document}
\nocite{}

\title[Gelfand--Kirillov dimension for $p$-adic unitary groups of rank 2]{Gelfand--Kirillov dimension for mod $p$ representations of $p$-adic unitary groups of rank 2}
\date{}

\author{Karol Kozio\l}
\address{Department of Mathematics, CUNY Baruch College, 55 Lexington Ave, New York, NY 10010 USA} 
\email{karol.koziol@baruch.cuny.edu}

\author{Stefano Morra}
\address{Universit\'e Paris 8, Laboratoire d'Analyse, G\'eom\'etrie et Applications, Universit\'e Sorbonne Paris Nord, CNRS, UMR 7539,  F-93430, Villetaneuse, France}
\email{morra@math.univ-paris13.fr}


\begin{abstract}
  Let $p$ be a prime number and $F/F^+$ a CM extension of a totally real field such that every place of $F^+$ above $p$ is unramified and inert in $F$.  We fix a finite place $v$ of $F^+$ above $p$, and let $\overline{r}: \textnormal{Gal}(\overline{F^+}/F^+) \longrightarrow {}^C\textnormal{U}_{1,1}(\overline{\mathbb{F}}_p)$ be a modular $L$-parameter valued in the $C$-group of a rank 2 unitary group associated to $F/F^+$.  We assume $\overline{r}$ is semisimple and sufficiently generic at $v$.  Using recent results of Breuil--Herzig--Hu--Morra--Schraen along with our previous work, we prove that certain admissible smooth $\overline{\mathbb{F}}_p$-representations of the $p$-adic unitary group $\textnormal{U}_{1,1}(F^+_v)$ associated to $\overline{r}$ in spaces of mod $p$ automorphic forms have Gelfand--Kirillov dimension $[F^+_v:\mathbb{Q}_p]$.   
  \end{abstract}

\maketitle
\tableofcontents

\section{Introduction}

\subsection{Context and main result}

Let $p$ be a prime number.  In recent years there has been substantial progress in formulating versions of local and global Langlands reciprocity laws with $p$-torsion coefficients.  The most spectacular progress in this direction has been the $p$-adic local Langlands correspondence for $\nG\nL_2(\bbQ_p)$, which established a bijection between $p$-adic Galois representations of $\mathrm{Gal}(\overline{\bbQ}_p/\bbQ_p)$ and representations of $\nG\nL_2(\bbQ_p)$ on $p$-adic Banach spaces (see \cite{berger:bourbaki,breuil:bourbaki, breuil:ICM} for an overview).  The starting point in constructing such a correspondence is the so-called semisimple mod $p$ local Langlands correspondence for $\nG\nL_2(\bbQ_p)$: in \cite{breuil:GL2QpI}, Breuil constructed a bijection between continuous irreducible $2$-dimensional $\mathrm{Gal}(\overline{\bbQ}_p/\bbQ_p)$-representations on $\bbF$-vector spaces ($\bbF$ being a finite ``sufficiently large'' extension of $\bbF_p$, which we fix through this introduction) and supersingular smooth $\nG\nL_2(\bbQ_p)$-representations over $\bbF$-vector spaces, which was moreover compatible with the weight part of Serre's conjecture.

The situation when considering $p$-adic groups larger than $\nG\nL_2(\bbQ_p)$ has been suprisingly more difficult.  This is due in part to the overabundance of irreducible smooth representations of $p$-adic groups on $\bbF$-vector spaces, which are far more numerous than the objects on the ``Galois side'' (see \cite{BP, hu:supersingular, schraen:presentation, wu} for an account of some of the difficulties which arise).  Furthermore, even for representations coming from global constructions (which are expected to be relevant for a $p$-adic Langlands correspondence via local/global compatibility), it has been extremely hard to obtain structural results.  Beyond information at tame level obtained in \cite{EGS,LMS,huwang:K1, le:nonsplit}, the difficult results of \cite{BHHMS,huwang} have been the first progress towards non-tame properties of the supersingular representations appearing in Hecke eigenspaces of the cohomology of Shimura curves.

Despite these hard-won achievements, the overwhelming majority of the advances in the $p$-adic local Langlands correspondence have involved the general linear group.  As a result, the possibility of functoriality for $p$-torsion Langlands parameters has been largely unexplored so far.

This paper aims at making progress on understanding cohomological representations appearing in a putative mod $p$ local Langlands correspondence in the case of $p$-adic unitary groups in two variables.  The reason for considering such groups is twofold.  On the one hand, the most definitive results towards a mod $p$ local Langlands correspondence so far are for the group $\nG\nL_2$, and the unitary groups we consider are outer forms thereof, which allows us to adapt techniques appearing in the literature to our setup.  On the other hand, initial evidence (in the form of the weight part of Serre's conjecture) obtained in \cite{koziolmorra} leads us to expect that the results below will have interesting and fruitful applications and will reveal new structure.  Hence, in this article we continue the work begun in \cite{koziolmorra} and establish analogs of the main results of \cite{EGS,LMS,BHHMS} for $p$-adic unitary groups in two variables defined over an unramified extension of $\bbQ_p$.

In order to state the main result we introduce some notation.  Let $K_2/K$ be unramified extensions of $\qp$ with $K_2/K$ quadratic, and write $\nU_{1,1}$ for the unramified unitary group in two variables defined over the ring of integers $\cO_K$ of $K$.  We let ${}^C \nU_{1,1}$ denote the $C$-group of $\nU_{1,1}$, i.e., the usual Langlands $L$-group of a canonical central extension of $\nU_{1,1}$.  A local $L$-parameter is a continuous homomorphism $\Gal(\qpb/K) \longrightarrow {}^C{\nU_{1,1}}(\bbF)$ compatible with the projection ${}^C{\nU_{1,1}}(\bbF) \longtwoheadrightarrow \Gal(K_2/K)$, and whose multiplier is equal to the mod $p$ cyclotomic character.  We shall be interested in mod $p$ representations of $\nU_{1,1}(K)$ associated to such $L$-parameters.

As mentioned above, due to the preponderance of smooth mod $p$ representations of $p$-adic groups, it is natural to first restrict our attention to representations having global origin, namely those appearing in the isotypic components in the cohomology of arithmetic groups with mod $p$ coefficients.  To present this setup, let $F/F^+$ be a CM extension unramified at all finite places, where $F^+/\bbQ$ is totally real.  We assume that $p$ is unramified in $F$ and any place of $F^+$ above $p$ is \textbf{inert} in $F$.  Moreover, for simplicity of exposition, we assume that $p$ is the unique $p$-adic place of $F^+$ and that $F^+_p \cong K$.  (This assumption is not required for our main results in the body of the paper.)  Let $\bbG$ be a unitary group defined over $\cO_{F^+}$, which is isomorphic to the compact unitary group $\nU_2(\bbR)$ at all infinite places, isomorphic to $\nU_{1,1}(K)$ at $p$, and isomorphic to $\nG\nL_2(F^+_v)$ at all places $v$ which split in $F$.

Given the above setup, we may now construct the representations of $\nU_{1,1}(K)$ on which we will focus.  We denote by $S_{\bbG}(\sfK^p,\bbF)$ the space of algebraic automorphic forms on $\bbG$ with tame level $\sfK^p \subset \bbG(\bbA^{p,\infty}_{F^+})$ and infinite level at $p$.  Letting $\fm'_{\rbar}$ denote the maximal ideal associated to a global $L$-parameter $\rbar:\Gal(\overline{\bbQ}/F^+) \longrightarrow {}^C{\nU_{1,1}}(\bbF)$ in the unramified Hecke algebra acting on $S_{\bbG}(\sfK^p,\bbF)$ (see Subsection \ref{subsub:Hecke} for its precise definition), we set
$$\pi(\rbar) := S_{\bbG}(\sfK^p,\bbF)[\fm'_{\rbar}],$$
which is a smooth mod $p$ representation of $\bbG(F^+_p) \cong \nU_{1,1}(K)$ associated to $\rbar$ (and which, a priori, depends on $\rbar$ and the global setup, and not only $\rbar|_{\Gal(\qpb/K)}$).

The $\nU_{1,1}(\cO_K)$-socle of the representation $\pi(\rbar)$ was investigated in the paper \cite{koziolmorra} in relation to Serre weight conjectures.  Going deeper into the structure of $\pi(\rbar)$, and inspired by the techniques of \cite{BHHMS}, the main result of the present paper is as follows:

\begin{thm}
  \label{mainthm-intro}
We maintain the above global setup: $F/F^+$ is a CM field extension of $F^+$ which is unramified at all finite places, and such that $p$ is unramified and inert in $F$.  Let $\bbG$ be the unitary group defined above, and let $\rbar: \textnormal{Gal}(\overline{\bbQ}/F^+) \longrightarrow {}^C\nU_{1,1}(\bbF)$ be an $L$-parameter with cyclotomic multiplier. We furthermore assume that:
\begin{itemize}
\item $\pi(\rbar)\neq 0$ (in other words, $\rbar$ is modular);
\item $\rbar^{-1}({}^C{\nU^{\circ}_{1,1}}(\bbF)) = \Gal(\overline{\bbQ}/F)$;
  \item $\rbar(\Gal(\overline{\bbQ}/F)) \supset \nG\nL_2(\bbF_p)$;
  \item $\overline{\bbQ}^{\ker(\textnormal{ad}^0(\rbar))}$ does not contain $F(\zeta_p)$;
  \item $\rbar$ is unramified outside $p$;
  \item $\rbar$ is tamely ramified and $12$-generic above $p$. 
\end{itemize}
Then $\dim_{\nU_{1,1}(K)}(\pi(\rbar)) = [K:\bbQ_p]$, where $\dim_{\nU_{1,1}(K)}$ denotes the Gelfand--Kirillov dimension.
\end{thm}

\noindent (We reiterate that the assumption on $p$ being inert in $F^+$ is only for simplicity of exposition, and is not required for our main result below.  In the body of the paper, our assumption on splitting behavior is that $F^+/\bbQ$ is unramified at $p$, that $F/F^+$ is unramified at all finite places, and that every place of $F^+$ above $p$ is inert in $F$.  The representation $\pi(\rbar)$ is then defined with infinite level at a fixed place $v$ of $F^+$ above $p$; see equation \eqref{defofpi} below.)

\subsection{Strategy}

The proof of the $\nG\nL_2$-analog of Theorem \ref{mainthm-intro} is based on two main ingredients: 
\begin{enumerate}
  \item the behavior of certain torsion $\bbF[\![I_\nG]\!]$-modules with minimal multiplicities (which turn out to have the expected Gelfand--Kirillov dimension), where $I_\nG$ denotes the upper-triangular Iwahori subgroup of $\nG\nL_2(K)$; and
  \item \label{ingredient2} the fact that the smooth representations appearing in cohomology satisfy these minimal multiplicity assumptions.
\end{enumerate}
The first component is purely of a group-theoretic nature, while the second comes from favorable alignment between locally algebraic representation theory and the intersection theory of deformation rings with certain $p$-adic Hodge theoretic conditions.

The main observations we use for proving Theorem \ref{mainthm-intro} is that the two ingredients above, which pertain to the general linear group, can be adapted to the setting of unitary groups.  We spell out more carefully what we mean by this.  Firstly, we note that the Iwahori subgroup $I_\nG$ of $\nG\nL_2(K)$ and the Iwahori subgroup $I_\nU$ of $\nU_{1,1}(K)$ ``agree up to a central subgroup:''  both are subgroups of $I_{\nG\nU}$, the Iwahori subgroup of $\nG\nU_{1,1}(K)$, and $I_{\nG\nU}$ is generated by its center and either of these two subgroups.  Thus, we may deduce results about modules over $\bbF[\![I_\nU]\!]$ by first extending the action to $\bbF[\![I_{\nG\nU}]\!]$, and then restricting to $\bbF[\![I_\nG]\!]$, or vice versa (we make this procedure more precise in Section \ref{sec:transfer}).  Moreover, this transfer procedure works more generally for various subgroups of $\nG\nL_2(\cO_K)$ and $\nU_{1,1}(\cO_K)$, and as a consequence we may directly import many results regarding Ext groups and socle filtrations from \cite{BHHMS}, see Section \ref{appendix:EGC}.  This addresses the first ingredient above.

In order to have an analog for unitary groups of the second ingredient above, we prove results which are endemic to unitary groups (i.e., results which we cannot obtain simply by applying the transfer procedure of the previous paragraph).  The main results in this direction are the precise descriptions of various rings parametrizing deformations of local $L$-parameters $\Gal(\overline{\bbQ}_p/K) \longrightarrow {}^C\nU_{1,1}(\bbF)$ and satisfying the $p$-adic Hodge theoretic conditions which are relevant for the local/global compatibility statements we need (namely, Hodge--Tate weights $(2,-1)$ and ``multi-type'' potentially crystalline descent data).  This is done by following the outline of \cite{BHHMS}, but now imposing the extra symmetries coming from the polarizations on Kisin modules introduced in \cite{koziolmorra}.  This yields explicit descriptions of the geometry and intersection theory of deformation spaces of ${}^C\nU_{1,1}$-valued $L$-parameters, for which we refer to Section \ref{sec:galdefs}.

In order to make use of the structure of deformation rings obtained above to prove minimal multiplicity results about $\pi(\rbar)$ (as in the ingredient \eqref{ingredient2} above), we must relate local deformation rings to spaces of global automorphic forms.  This is done following the strategy initiated in \cite{EGS}, via the mechanism of Taylor--Wiles--Kisin patching of \cite{CEGGPS}.  We employ a slight variation on our previous work \cite{koziolmorra} to obtain an exact functor $M_\infty$ from finitely generated $W(\bbF)$-modules with an action of $\nU_{1,1}(\cO_K)$ to finitely generated modules over a ring $R_\infty$, which is a power series ring over a tensor product of local deformation rings.  
(Here we write $W(\bbF)$ to denote the ring of Witt vectors of $\bbF$.)
As discussed in \cite{CEGGPS}, this functor satisfies favorable geometric and local/global compatibility properties, so that the support of $M_\infty$ computed on suitable representations is ``as large as possible'' over the potentially crystalline deformation spaces considered above.
By applying the functor to various types of representations (irreducible representations of $\nU_{1,1}(\cO_K)$ over $\bbF$, projective envelopes of such representations over $\bbF$ and $W(\bbF)$, lattices in locally algebraic representations, etc.), using the local/global compatibility properties of $M_\infty$, and following the strategy of \cite{BHHMS}, we obtain several freeness results for patched modules.  These parallel the analogous results for $\nG\nL_2$, and suffice to establish ingredient \eqref{ingredient2} for $\nU_{1,1}$.

In the course of the patching arguments in the proof of Theorem \ref{mainthm-intro}, we also establish results about invariants of $\pi(\rbar)$ under the principal congruence subgroup:

\begin{thm}
  Maintain the setting of Theorem \ref{mainthm-intro}.  We then have an isomorphism of $\nU_{1,1}(\cO_K/(p))$-representations
  $$\pi(\rbar)^{\nU_{1,1}(\cO_K)_1} \cong D_0(\rbar|_{\Gal(\overline{\bbQ}_p/K)}),$$
  where $\nU_{1,1}(\cO_K)_1 := \ker(\nU_{1,1}(\cO_K) \longtwoheadrightarrow \nU_{1,1}(\cO_K/(p)))$ is the principal congruence subgroup, and $D_0(\rbar|_{\Gal(\overline{\bbQ}_p/K)})$ denotes the $\nU_{1,1}$-analog of the Diamond diagrams for $\nG\nL_2(\cO_K)$-representations introduced in \cite{BP}.  
\end{thm}

We conclude with some speculation.  Many of our results are direct analogs of those appearing in \cite{BHHMS}.  However, we expect the representations $\pi(\rbar)$ to differ from their $\nG\nL_2$-analogs in a very specific way.  Namely, \cite[Cor. 3.3.5.6]{BHHMS2} shows (under some global assumptions) that when $\rhobar$ is an irreducible $\nG\nL_2(\bbF)$-valued Galois representation, then the $\nG\nL_2(K)$-representation $\pi_{\nG\nL_2}(\rhobar)$ is irreducible and supersingular.  On the other hand, the Serre weight results of \cite{koziolmorra} along with the classification result of \cite{koziol:U11} show that, at least when $K = \bbQ_p$, the representation $\pi(\rbar)$ defined above cannot be irreducible.  We expect that, when $\rbar|_{\Gal(\overline{\bbQ}_p/K)}$ is of ``niveau 2,'' the representation $\pi(\rbar)$ is a direct sum of two irreducible, supersingular representations, which are moreover conjugate by the element $\sm{1}{0}{0}{p} \in \nG\nU_{1,1}(K)$.  We hope to address this in future work.

\subsection{Structure of the paper}  We now give a brief outline of the contents of each section.

After collecting basic notation in Section \ref{sec:notation}, we proceed in Section \ref{sec:gptheory} to recall the $p$-adic unitary groups we will be working with.  We also spell out how these unitary groups are related to general linear groups via the process of base change, and discuss the relevant Langlands dual groups.  The notions of $L$-parameters and their interaction with representation theory of $p$-adic unitary groups is treated in Section \ref{sec:Lparams-Serrewts}, where we also recall the form of the inertial local Langlands correspondence we will be needing in this paper.

Section \ref{sec:transfer} contains the key principles to transfer representation theoretic results between different groups.  We mainly deal with modules over Artinian algebras, but we also give a few results which pertain to representations on $W(\bbF)$-lattices (Subsection \ref{sec:transfer:char0}).

The combinatorics of Serre weights is handled in Section \ref{appendix:EGC}.  In particular, we interpret the base change map as a map between two extension graphs (introduced for the group $\nG\nL_2$ in \cite{LMS}), and deduce results about predicted sets of Serre weights and constituents of mod $p$ reductions of Deligne--Lusztig representations.

Sections \ref{sec:Kisinmods} and \ref{sec:galdefs} are devoted to computations of the potentially crystalline deformations which are relevant for global applications.  We first give the necessary background on polarized Kisin modules with monodromy.  This notion is necessary to access potentially crystalline deformations which are not potentially Barsotti--Tate, the latter having already been treated in \cite[\S 5]{koziolmorra}.  We subsequently compute the relevant deformation rings in Section \ref{sec:galdefs}, following the proofs of \cite[\S 4]{BHHMS}.  In order to make the arguments of \textit{op. cit.} work in our setting, we analyze the effect of the polarization by including certain symmetry ideals (described in Tables \ref{Table1} and \ref{Table2}).

In Section \ref{sec:BP} we define Diamond diagrams (introduced in \cite{BP} for $\nG\nL_2$) in the setting of unitary groups.  This is necessary to give explicit representation theoretic criteria for an upper bound on the Gelfand--Kirillov dimension of a smooth $\bbF$-valued representation of $\nU_{1,1}(K)$.  In particular, the transfer procedure of Section \ref{sec:transfer} comes into play to import the representation theoretic results of \cite{BHHMS} to our unitary group context.

In Section \ref{sec:global} we finally apply the results of Sections \ref{sec:galdefs} and \ref{sec:BP} to arrive at our main result.  After constructing and establishing the fundamental properties of the relevant patching functors, we deduce several results about freeness of patched modules.  Finally, in Subsection \ref{freenesslattices} we perform the ``gluing argument'' of \cite{EGS} to obtain the multiplicity one statements necessary to apply the upper bound criterion from Subsection \ref{subsec:UB}.  (The lower bound follows in a manner analogous to \cite{BHHMS}.)

We include two appendices.  In Appendix \ref{appA}, we show that if $\sigma$ is a predicted Serre weight of a tamely ramified $L$-parameter $\rhobar:\Gal(\overline{\bbQ}_p/K) \longrightarrow {}^C\nU_{1,1}(\bbF)$ (of the kind considered in the introduction), then $\rhobar$ admits a crystalline lift with Hodge--Tate weights determined by $\sigma$.  In Appendix \ref{appB}, we follow the argument of \cite[App. A]{emertongee} to show that such a local $\rhobar$ can be globalized to a global $\rbar:\Gal(\overline{\bbQ}/F^+) \longrightarrow {}^C\nU_{1,1}(\bbF)$, which is moreover modular.

\subsection*{Acknowledgements}  The initial ideas for this project grew out of a SQuaRE meeting at the American Institute of Mathematics in August 2019, and subsequently developed into the manuscript \cite{BHHMS}.  We would like to heartily thank Christophe Breuil, Florian Herzig, Yongquan Hu, Benjamin Schraen, and Sug Woo Shin for their input and enlightening conversations.  We also thank AIM for providing excellent working conditions.  Parts of this work were carried out at the University of Michigan and the University of Toronto, and we would like to thank these institutions for their support.

During the preparation of this article, KK was supported by NSF grant DMS-2101836, and SM was supported by the Institut Universitaire de France and the A.N.R. project CLap-CLap ANR-18-CE40-0026.

\section{Notation}
\label{sec:notation}

\subsection{Fields}
\label{notation:fields}
Let $p$ denote an odd prime number, and fix an algebraic closure $\qpb$ of $\qp$.  We denote by $\overline{\bbF}_p$ the residue field of $\qpb$, and we assume that all field extensions of $\qp$ are contained in $\qpb$.  Throughout we will work with a finite extension $E$ of $\qp$ which will serve as our field of coefficients. We let $\cO$ denote the ring of integers of $E$, $\varpi$ its uniformizer, and $\bbF$ its residue field.  We will assume $E$ and $\bbF$ are sufficiently large as necessary.

Let $f\geq 1$, and let $K$ denote the unramified extension of $\qp$ of degree $f$.  We let $\cO_K$ denote its ring of integers, with canonical uniformizer $p$, and denote its residue field by $k_K$.  We also let $K_2$ denote the unique unramified quadratic extension of $K$, and $\cO_{K_2}$ its ring of integers.  The group $\nU_1(K)\subset \cO_{K_2}^\times$ is defined as the kernel of the norm map $K_2^\times \longrightarrow K^\times$.

We fix an embedding $\sigma_0:k_K \longhookrightarrow \bbF$, and define $\sigma_j := \sigma_0 \circ \varphi^j$, where $\varphi:x \longmapsto x^p$ is the arithmetic Frobenius on $k_K$.  We let $\cJ := \Hom(k_K, \bbF)$ denote the set of field embeddings, which is identified with $\{0, 1, \ldots, f - 1\}$, considered modulo $f$.  We let $\sigma_0':k_{K_2} \longhookrightarrow \bbF$ denote an embedding extending $\sigma_0$, and identify $\cJ' := \Hom(k_{K_2},\bbF)$ with $\{0, 1, \ldots, 2f - 1\}$ modulo $2f$ as above.

\subsection{Representations}
All representations will live on vector spaces over $E$ or $\bbF$, or on $\cO$-modules, unless otherwise indicated.  By abuse of notation, we will generally not distinguish between a representation and its isomorphism class.  If $G$ is a group, $H \trianglelefteq G$ a normal subgroup, $V$ an $H$-representation and $g\in G$, we write $V^g$ to denote the $H$-representation obtained by keeping the same underlying vector space (or $\cO$-module) $V$, and letting $h$ act by $ghg^{-1}$.

Given a finite length representation $V$ of some group, we let $\textnormal{JH}(V)$ denote its set of Jordan--H\"older factors.  If $V$ denotes a representation of a (pro)finite group $G$ on a finite-dimensional $E$-vector space, then we may choose a $G$-stable $\cO$-lattice $V^\circ$ inside $V$, and we write $\overline{V^\circ}$ for its reduction mod $\varpi$.  The set of Jordan--H\"older factors of $\overline{V^\circ}$ is independent of the choice of lattice $V^\circ$.  We therefore write $\textnormal{JH}(\overline{V})$ for $\textnormal{JH}(\overline{V^\circ})$.  
We denote by $V\longmapsto V^\vee$ the duality functor defined on the category of finit- dimensional $E$-vector spaces (resp.~finite-dimensional $\bbF$-vector spaces).

\subsection{Galois theory}
For any field $F$, we let $\Gamma_F := \textnormal{Gal}(\overline{F}/F)$ denote the absolute Galois group of $F$, where $\overline{F}$ is a fixed separable closure of $F$.  If $F$ is a number field and $v$ is a place of $F$, we let $F_v$ denote the completion of $F$ at $v$, and use the notation $\textnormal{Frob}_v$ to denote a geometric Frobenius element of $\Gamma_{F_v}$.  If $F$ is a $p$-adic field, we let $I_F$ denote the inertia subgroup of $\Gamma_F$.  In addition, we let $\textnormal{Art}_F:F^\times \longrightarrow \Gamma_F^{\textnormal{ab}}$ denote the local Artin map, which sends uniformizers to geometric Frobenius elements.

For $F$ either a number field or a $p$-adic field, we let $\varepsilon:\Gamma_F\longrightarrow \zp^\times$ denote the $p$-adic cyclotomic character, and let $\overline{\varepsilon}$ or $\omega$ denote its reduction mod $p$.

If $F$ is a $p$-adic field, $V$ a de Rham representation of $\Gamma_F$ over $E$, and $\kappa:F \longhookrightarrow E$ an embedding, then we define $\textnormal{HT}_\kappa(V)$ to be the multiset of Hodge--Tate weights with respect to $\kappa$.  Thus, $\textnormal{HT}_\kappa(V)$ contains $i$ with multiplicity $\dim_E(V\otimes_{F,\kappa}\widehat{\overline{F}}(-i))^{\Gamma_F}$.  In particular, $\textnormal{HT}_\kappa(\varepsilon) = \{1\}$.

By abuse of notation, we also use $\varphi$ to denote an element of $\Gamma_{\bbQ_p}$ which is a fixed lift of $\textnormal{Art}_{\bbQ_p}(p^{-1})\in \Gamma_{\bbQ_p}^{\textnormal{ab}}$; in particular, we have $\varepsilon(\varphi) = 1$.  Consequently, by extending the embedding $\sigma_0$ (resp., $\sigma_0'$) to a field embedding $K \longhookrightarrow E$ (resp., $K_2 \longhookrightarrow E$), we may use $\varphi$ to identify $\cJ$ with $\Hom(K,E)$ (resp., $\cJ'$ with $\Hom(K_2,E)$).  (Note that this is opposite to our notation in \cite{koziolmorra}; there, $\varphi$ was used to denote a geometric Frobenius element.)

\subsection{Miscellany}
We write matrix transposes \emph{on the right}, so that $A^\top$ denotes the transpose of a matrix $A$.  Given an (anti)automorphism $\theta$ of $\nG\nL_n(R)$ which commutes with the transpose, we write $A^{\theta\top}$ for $(A^\theta)^{\top}$; in particular, we write $A^{-\top}$ for $(A^{-1})^\top$.

\section{Group theory}
\label{sec:gptheory}

We recall some of the main notation from \cite[\S 2]{koziolmorra}.  We diverge slightly from the notation in \textit{op. cit.}, in order to be consistent with \cite{BHHMS}.

\subsection{Unitary groups over $\cO_K$}
\label{unitarygps:ok}
\subsubsection{}

Let $\nU_{1,1}$ denote the algebraic group over $\cO_K$ given by
$$\nU_{1,1}(R) = \left\{g\in \nG\nL_2(\cO_{K_2}\otimes_{\cO_K}R): g^{(\varphi^f\otimes 1)\top}\Phi_2g = \Phi_2\right\},$$
where $R$ is an $\cO_K$-algebra, and $\Phi_2 := \sm{0}{1}{-1}{0}$.  We also define $\nG\nU_{1,1}$ by
$$\nG\nU_{1,1}(R) = \left\{g\in \nG\nL_2(\cO_{K_2}\otimes_{\cO_K}R): g^{(\varphi^f\otimes 1)\top}\Phi_2g = \kappa\Phi_2~\textnormal{for some $\kappa \in R^\times$} \right\}.$$

We shall crucially use several group decompositions.  First, we have an isomorphism of $\cO_K$ group schemes:
\begin{equation}
  \label{GL2decomp}
\nG\nU_{1,1}  \cong  \nG\nL_{2/\cO_K} \times^{\bG_{m/\cO_K}} \textnormal{Res}_{\cO_{K_2}/\cO_K}(\bG_{m/\cO_{K_2}}).
\end{equation}
We have an analogous decomposition relative to $\nU_{1,1}$, at the level of $\cO_K$-points.  Precisely, we have an isomorphism of topological groups
\begin{equation}
  \label{U11decomp}  
  \nG\nU_{1,1}(\cO_K)  \cong  \nU_{1,1}(\cO_K) \times^{\nU_1(\cO_K)} \cO_{K_2}^\times.
\end{equation}
Note that the above decomposition does not hold at the level of $K$-points.  Using equation \eqref{U11decomp}, we may construct other such decompositions for quotient groups or closed subgroups.  For example, we deduce
\begin{equation}
  \label{U11decomp:finite}  
  \nG\nU_{1,1}(k_K)  \cong  \nU_{1,1}(k_K) \times^{\nU_1(k_K)} k_{K_2}^\times.
\end{equation}

\subsubsection{}

We let $H := \widetilde{\nU}_{1,1}$ denote the central $\bG_m$-extension constructed in \cite{buzzardgee}, considered as a group scheme over $\cO_K$.  Concretely, $H = \nU_{1,1}\times_{\nP\nG\nL_2}\nG\nL_2$ is the set of all pairs $(h,h')$, with $h\in \nU_{1,1}, h'\in \nG\nL_2$, subject to the condition that $h$ and $h'$ have the same image in $\nP\nG\nL_2$ (recall that $\nU_{1,1}$ and $\nG\nL_2$ become isomorphic over an algebraic closure of $K$, so that we have a map from $\nU_{1,1}$ to $\nP\nG\nL_2$).  The maps $H \longrightarrow \nU_{1,1}$ and $H\longrightarrow \nG\nL_2$ are the projections onto the corresponding factors, and the map $\imath:\bG_m\longrightarrow H$ is $\lambda\longmapsto (1, \sm{\lambda}{0}{0}{\lambda})$.  The action of $\Gamma_K$ on the first factor of $H(\qpb)$ is the one induced from $\nU_{1,1}(\qpb)$, while the action on the second factor is the standard one.

\subsubsection{}
\label{tori}

We let $T_{\nU}, T_\nG, T_H$ denote the diagonal maximal tori of $\nU_{1,1}, \nG\nL_2$, and $H$, respectively, over $\cO_K$.  The character group of $T_H$ (over $\overline{\bbQ}_p$) is given by
$$X^*(T_{H}) = \left\{ (a,b,c,d)\in X^*(T_\nU)\oplus X^*(T_\nG)\cong \bbZ^4 \right\}\big/\sim$$
where
$$(a,b,c,d)\sim (a + z, b - z, c - z, d + z)$$
for $z\in \bbZ$.  
The maps $X^*(T_\nU)\longrightarrow X^*(T_{H}), X^*(T_\nG)\longrightarrow X^*(T_{H})$ are the inclusions into the corresponding factors, and the projection $X^*(T_{H})\longrightarrow X^*(\bG_m)\cong \bbZ$ is 
$$(a,b,c,d)\longmapsto c + d.$$  
We will identify $X^*(T_\nU)$ with its image in $X^*(T_H)$ without comment.

Similarly, the cocharacter group is given by
$$X_*(T_{H}) = \left\{(a', b', c', d')\in X_*(T_\nU)\oplus X_*(T_\nG)\cong \bbZ^4 :a' - b' = c'- d'\right\}.$$
The maps $X_*(T_{H})\longrightarrow X_*(T_\nU), X_*(T_{H})\longrightarrow X_*(T_\nG)$ are the projections onto the corresponding factors, and the map $X_*(\bG_m)\cong \bbZ \longrightarrow X_*(T_{H})$ is 
$$a' \longmapsto (0,0,a',a').$$

\subsubsection{}

The actions of $\Gamma_K$ on $X^*(T_{H})$ and $X_*(T_{H})$ are the ones induced from $X^*(T_\nU)$ and $X_*(T_\nU)$: they are both unramified, and we have
$$\varphi^f\cdot (a,b,c,d) =  (-b,-a,c,d)$$
on both lattices.

\subsubsection{}

The roots $\Phi_H \subset X^*(T_{H})$ are given by $\{\pm\alpha_H\}$, where 
$$\alpha_H := (1,-1,0,0).$$ 
Likewise, the coroots $\Phi_H^\vee \subset X_*(T_{H})$ are given by $\{\pm\alpha_{H}^\vee\}$ where 
$$\alpha_H^\vee := (1,-1,1,-1).$$ 
The upper triangular Borel subgroup of $H$ determines the subset $\Delta_H = \Phi_H^+ = \{\alpha_H\}$ of positive (and simple) roots (and similarly for the set $\Delta_H^\vee$ of simple coroots).

We define 
$$\rho_H := (1,0,0,0) \in X^*(T_H),$$
so that $\langle\rho_H,\alpha_H^\vee\rangle = 1$.  Note, however, that $\rho_H$ is \emph{not} Galois-invariant.  To remedy this, we define 
$$\eta_H := (0,0,1,0)\in X^*(T_H).$$
This element is Galois-invariant, and $\langle\eta_H,\alpha_H^\vee\rangle = 1$.  (The character $\eta_H$ is a twisting element for $H$, in the sense of \cite{buzzardgee}.)

The Weyl group of $(H, T_{H})$ is denoted $W$.  It is a cyclic group of order 2, whose generator we denote $\fw$.

\subsection{Unitary groups over $\bbZ_p$}
\label{unitarygps:qp}

\subsubsection{}
\label{sub:sub:unitary:loc}
We now consider unitary groups over $\qp$.  We set 
$$(\underline{\nU}_0,~\underline{T}_{\nU,0},~\underline{G}_0,~\underline{T}_0)  :=  \textnormal{Res}_{\cO_K/\zp}(\nU_{1,1},~T_\nU,~H,~ T_H),$$
and define 
$$(\underline{\nU},~\underline{T}_\nU,~\un{G},~ \un{T}) := (\underline{\nU}_0,~\underline{T}_{\nU,0},~\un{G}_0,~ \un{T}_0) \times_{\bbZ_p}\cO.$$
Thus, we have a surjections 
$$\underline{G} \longtwoheadrightarrow \underline{\nU},\qquad \underline{T} \longtwoheadrightarrow \underline{T}_{\nU},$$

\subsubsection{}
Recall that we assumed $E$ to be sufficiently large (in particular, we assume $E$ contains the images of all embeddings $K_2 \longhookrightarrow \overline{E}$).  In particular, the isomorphism $\cO_K\otimes_{\bbZ_p}\cO \stackrel{\sim}{\longrightarrow} \cO^{\cJ}$ given by $x\otimes y \longmapsto (\sigma_j(x)y)_{j \in \cJ}$ gives an isomorphism $\un{G} \cong \prod_{j \in \cJ} H_{/\cO}$.  We have a similar isomorphisms for $\un{T}$, $\un{\nU}, \un{T}_\nU$, $X^*(\un{T})$, $X^*(\un{T}_\nU)$, $\un{\Phi}$, and $\un{\Phi}^+$ (where the latter denote the sets of (positive) roots of $(\un{G},\un{T})$ relative to the upper-triangular Borel subgroup in each embedding).  In particular, any $\mu \in X^*(\un{T})$ may be written as $\mu = (\mu_j)_{j \in \cJ}$ according to this decomposition.

There is an automorphism $\pi$ of $X^*(\un{T})$ corresponding to the arithmetic Frobenius, given by 
$$\pi \big((a_j,b_j,c_j,d_j)_{j \in \cJ}\big)_{j'} = \begin{cases} (a_{j' - 1}, b_{j' - 1}, c_{j' - 1}, d_{j' - 1}) & \textnormal{if}~ 1 \leq j' \leq f - 1 \\ (-b_{f - 1}, -a_{f - 1},  c_{f - 1}, d_{f - 1}) & \textnormal{if}~ j' = 0.\end{cases}$$
An analogous action (i.e., with a ``shift right'') holds for $X_*(\un{T})$, $X^*(\un{T}_\nU)$ and $X_*(\un{T}_\nU)$.

\subsubsection{}
The set of positive roots $\un{\Phi}^+$ is given by $\{\alpha_j\}_{j \in \cJ}$, where $\alpha_j \in X^*(\un{T})$ is equal to $\alpha_H$ in embedding $j$ and 0 elsewhere.  Analogously, we have the set $\un{\Phi}^{\vee,+}$ of positive coroots, given by $\{\alpha_j^\vee\}_{j \in \cJ}$.  We let $\rho_j \in X^*(\un{T})$ denote the character which is equal to $\rho_H$ in embedding $j$ and 0 elsewhere, and set $\rho := \sum_{j \in \cJ} \rho_j$.  We make analogous definitions of $\eta_j$ and $\eta$, and note that $\eta$ gives a twisting element of $\un{G}$.  As with the analogous groups over $\cO_K$, we identify $X^*(\un{T}_\nU)$ with its image in $X^*(\un{T})$ (as those characters whose last two entries are 0 in each embedding).  Thus, $\alpha_j, \rho_j \in X^*(\un{T}_\nU)$ for each $j \in \cJ$.

The Weyl group $\un{W}$ of $(\un{G},\un{T})$ is identified with $W^{f}$.  We shall write elements of $\un{W}$ as $s = (s_0,s_1,\ldots, s_{f - 1})$, and define $\un{\fw} := (\fw, \fw, \ldots, \fw)$.  The group $\un{W}$ is also canonically identified with the Weyl group of $(\un{\nU},\un{T}_\nU)$, and we make this identification without further comment.

\subsection{Base change}
\label{sec:basechange}

It will be convenient to relate the constructions above to those of the group $\nG\nL_{2/\cO_{K_2}}$.  We will do this using the procedure of base change.  

\subsubsection{}
We define
$$(\un{G}',~\un{T}') := \textnormal{Res}_{\cO_{K_2}/\bbZ_p}(\nG\nL_{2/\cO_{K_2}},~T_{\nG/\cO_{K_2}})\times_{\bbZ_p}\cO.$$
Our convention will be that given some object (homomorphism, character, etc.) for the group $\un{G}$, we will denote with a prime the analogously defined object for $\un{G}'$.

\subsubsection{}
As in the previous subsection, the assumption that $E$ is sufficiently large implies that we have an isomorphism $\un{G}' \cong \prod_{j' \in \cJ'} \nG\nL_{2/\cO}$.  We have similar isomorphisms for $\un{T}'$, $X^*(\un{T}')$, $\un{\Phi}'$, and $\un{\Phi}'^+$ (the latter defined relative to the upper-triangular Borel subgroup in each embedding $j' \in \cJ'$).

As in \cite[\S 3B1]{koziolmorra}, we identify $X^*(\underline{T}')$ with two copies of $X^*(\underline{T}_\nU)$.  In particular, given $\mu \in X^*(\underline{T}_\nU)$, we define
\begin{equation}
  \label{def-of-BC}
  \BC(\mu) := (\mu,-\underline{\fw}(\mu)) \in X^*(\underline{T}').
\end{equation}
Analogously, the Weyl group $\un{W}'$ of $(\un{G}', \un{T}')$ is identified with two copies of $\un{W}$.

We let $\pi'$ denote the automorphism of $X^*(\un{T}')$ corresponding to the arithmetic Frobenius.  Given $\mu' \in X^*(\un{T}')$, the automorphism $\pi'$ is given by $\pi'(\mu')_{j'} = \mu'_{j' - 1}$ for all $j' \in \cJ'$.  In particular, if $\mu \in X^*(\un{T}_\nU)$, then we have
$$-(\un{\fw},\un{\fw})\pi'^f\big(\BC(\mu)\big) = \BC(\mu).$$

Finally, we define $\eta'_{j'}$ to be the character which is equal to $(1,0)$ in embedding $j'$ and 0 elsewhere, and set $\eta' := \sum_{j' \in \cJ'}\eta'_{j'}$.

\subsection{Weyl modules}
\label{subsec:WM}
Given $\lambda = (\lambda_j)_{j \in \cJ} \in X^*(\un{T}_{\nU})$ we denote by ${V(\lambda)}$ the representation of $\nU_{1,1}(\cO_K)$ given by 
$$V(\lambda) := \bigotimes_{j\in\cJ}(\textnormal{Sym}^{\lambda_{j,1}-\lambda_{j,2}}(\cO^2)\otimes_{\cO} {\det}^{\lambda_{j,2}})^{(j)} := \bigotimes_{j = 0}^{f - 1} (\textnormal{Sym}^{\lambda_{j,1}-\lambda_{j,2}}(\cO_{K_2}^2)\otimes_{\cO_{K_2}} {\det}^{\lambda_{j,2}})\otimes_{\cO_{K_2},\sigma'_j}\cO.$$
(The tensor product on the right-hand side is naturally a representation of $\nG\nL_2(\cO_{K_2})$, and we view it as a representation of $\nU_{1,1}(\cO_K)$ by restriction.)  Thus, $V(\lambda)$ has rank $0$ as soon as $\lambda_{j,1}-\lambda_{j,2}<0$ for some $j\in\cJ$.  We adopt similar notation for $\lambda \in X^*(T_\nU)$.  Given an $\cO$-algebra $A$, we set $V_A(\lambda) := V(\lambda) \otimes_\cO A$.

The representation $V(\lambda)$ is ``well-defined up to isomorphism'' in the following sense: for $0 \leq j \leq f - 1$, define $\lambda_{j + f} := -\fw(\lambda_j)$.  Then, we have an isomorphism of $\nU_{1,1}(\cO_K)$-representations
$$(\textnormal{Sym}^{\lambda_{j,1}-\lambda_{j,2}}(\cO_{K_2}^2)\otimes_{\cO_{K_2}} {\det}^{\lambda_{j,2}})\otimes_{\cO_{K_2},\sigma'_j}\cO \cong (\textnormal{Sym}^{\lambda_{j + f,1}-\lambda_{j + f,2}}(\cO_{K_2}^2)\otimes_{\cO_{K_2}} {\det}^{\lambda_{j + f,2}})\otimes_{\cO_{K_2},\sigma'_{j + f}}\cO.$$

Finally, we note that $V(\lambda)$ is the representation denoted by $W_\lambda$ in \cite[\S 6B5]{koziolmorra} when $|\Sigma_p| = 1$ in the notation of \emph{op.~cit}.

\subsection{Dual groups}
\label{dualgps}

We now define the relevant Langlands dual groups.

\subsubsection{}

The based root datum of $\nU_{1,1}$ (with respect to the upper-triangular Borel subgroup) is given by 
$$\left(X^*(T_\nU)\cong \bbZ^2 ,~\{(1,-1)\},~ X_*(T_\nU)\cong \bbZ^2 ,~ \{(1,-1)\}\right).$$
We take $\nU_{1,1}^\vee := \nG\nL_{2/\bbZ_p}$ as the dual group, equipped with its standard based root datum.  We thus obtain an induced action of $\Gamma_K$ on $\nU_{1,1}^\vee$ given by
$$\gamma\cdot \widehat{g} = \begin{cases} \widehat{g} & \textnormal{if}~ \gamma\in \Gamma_{K_2},\\ \Phi_2\widehat{g}^{-\top}\Phi_2^{-1} = \det(\widehat{g})^{-1} \cdot \widehat{g} & \textnormal{if}~ \gamma\in \Gamma_K\smallsetminus\Gamma_{K_2},\end{cases}$$
for $\widehat{g} \in \nG\nL_2$.

\subsubsection{}
\label{subsub:dual:root}
Consider now the group $H = \widetilde{\nU}_{1,1}$.  The based root datum of $H$ is given by 
$$\Psi_H := \left(X^*(T_H),~ \Delta_H,~ X_*(T_H),~ \Delta_H^\vee\right),$$
and therefore the dual based root datum is 
$$\Psi_H^\vee = \left(X_*(T_H),~ \Delta_H^\vee,~ X^*(T_H),~ \Delta_H\right).$$
We let $H^\vee$ denote the dual group of $H$ over $\bbZ_p$, with maximal torus $T_H^\vee$ and upper-triangular Borel subgroup which contains $T_H^\vee$.  As in \cite[\S 2C]{koziolmorra}, we have $H^\vee \cong \nG\nL_{2/\bbZ_p} \times \bG_{m/\bbZ_p}$, with $\Gamma_K$-action given by
$$\gamma\cdot (\widehat{h},~a) = \begin{cases}(\widehat{h},~a) & \textnormal{if}~\gamma\in \Gamma_{K_2},\\ \left( a \cdot \Phi_2\widehat{h}^{-\top}\Phi_2^{-1},~ a\right) = \left( a\det(\widehat{h})^{-1} \cdot \widehat{h},~a\right) & \textnormal{if}~\gamma\in \Gamma_K\smallsetminus\Gamma_{K_2},\end{cases}$$
for $(\widehat{h},a)\in\nG\nL_2 \times \bG_m$.

Thus, we obtain the based root datum for $H^\vee$
\begin{eqnarray*}
\Psi_{H^\vee} & := & \left(X^*(T_H^\vee),~ \widehat{\Delta},~ X_*(T_H^\vee),~ \widehat{\Delta}^{\vee}\right)\\
 & = & \left(\bbZ^3,~ \{(1,-1,0)\},~ \bbZ^3,~ \{(1,-1,0)\}\right),
 \end{eqnarray*}
equipped with an action of $\Gamma_K$.  Moreover, we obtain an isomorphism of based root data $\phi:\Psi_H^\vee\stackrel{\sim}{\longrightarrow}\Psi_{H^\vee}$:
\begin{eqnarray*}
 \phi:X_*(T_H) & \stackrel{\sim}{\longrightarrow} & X^*(T_H^\vee)\\
 (a',b',c',d') & \longmapsto & (a',b',c'-a')\\
 (\phi^\vee)^{-1}:X^*(T_H) & \stackrel{\sim}{\longrightarrow} & X_*(T_H^\vee)\\
 (a,b,c,d) & \longmapsto & (a + c,b + d,c + d)
\end{eqnarray*}
where the last coordinate in the character (resp.~cocharacter) group of $T_H^\vee$ corresponds to the $\bG_m$ factor of $H^\vee$.  Note that this exchanges the roots and coroots.  We use this isomorphism to identify the Weyl group of $(H^\vee,T_H^\vee)$ with $W$.

\subsubsection{}
\label{def-of-Cgroup}

Finally, we define
$${}^C\nU_{1,1} := {}^L H = H^\vee \rtimes \textnormal{Gal}(K_2/K) = (\nG\nL_2 \times \bG_m)\rtimes\textnormal{Gal}(K_2/K),$$
with the Galois group acting on $H^\vee$ as above.  The injection $\imath:\bG_m\longrightarrow H$ induces a dual map $\widehat{\imath}:{}^L H\longrightarrow\bG_m$, which is given by $(\widehat{h},a)\rtimes\gamma\longmapsto a$.  We will occasionally make use of this group in a global context (i.e., replacing $K_2/K$ with a quadratic extension $F/F^+$ of number fields) as in \cite[Rmk. 2.1]{koziolmorra}.

In addition, we shall make use of the group scheme $\cG_2$ over $\bbZ_p$ defined in \cite[\S 2.1]{CHT}, which is equipped with a certain homomorphism $\nu: \cG_2 \longrightarrow \bG_m$.  The group $\cG_2$ is isomorphic to ${}^C\nU_{1,1}$; see \cite[\S 8.3]{buzzardgee} and \cite[\S 2D]{koziolmorra} for the explicit description.  In particular, under this isomorphism, the map $\widehat{\imath}$ corresponds to $(-)^{-1}\circ \nu$.

\section{Tame $L$-parameters and predicted Serre weights}
\label{sec:Lparams-Serrewts}

In this section we recall from \cite[\S\S 3,4]{koziolmorra} the notion of tame $L$-parameters for unitary groups and the set of Serre weights associated to them.

\subsection{Serre weights and Deligne--Lusztig representations}
\label{subsec:SW-and-DL}
\subsubsection{}
Let us define the sets of $p$-restricted, regular, and inner-product-zero characters as
\begin{eqnarray*}
  X_1(\un{T}) & := & \{\lambda \in X^*(\un{T}): 0 \leq \langle\lambda, \alpha^\vee\rangle \leq p - 1~\textnormal{for all}~\alpha \in \un{\Phi}^+\},\\
  X_{\textnormal{reg}}(\un{T}) & := & \{\lambda \in X^*(\un{T}): 0 \leq \langle\lambda, \alpha^\vee\rangle < p - 1~\textnormal{for all}~\alpha \in \un{\Phi}^+\},\\
   X^0(\un{T}) & := & \{\lambda \in X^*(\un{T}): \langle\lambda, \alpha^\vee\rangle = 0 ~\textnormal{for all}~\alpha \in \un{\Phi}^+\}.
\end{eqnarray*}
Recall that a Serre weight of $\un{G}_0(\bbF_p) = \widetilde{\nU}_{1,1}(k_K)$ is an irreducible representation of $\un{G}_0(\bbF_p)$ on an $\overline{\bbF}_p$-vector space.  Given $\lambda\in X_1(\un{T})$, we write $F(\lambda)$ to denote the restriction to $\un{G}_0(\bbF_p)$ of the irreducible algebraic representation of $\un{G}\times_{\cO}\overline{\bbF}_p$ of highest weight $\lambda$.  By \cite[Lem. 9.2.4]{GHS} this assignment gives a well-defined bijection 
\begin{eqnarray*}
\frac{X_1(\un{T})}{(p - \pi)X^0(\un{T})} & \longrightarrow & \left\{\text{Serre weights of $\un{G}_0(\bbF_p)$}\right\}_{/\cong}\\
\lambda & \longmapsto & F(\lambda).
\end{eqnarray*}
We will always assume that $\bbF$ is large enough so that any Serre weight can be realized over $\bbF$.  We have analogous constructions for $\nU_{1,1}(k_K)$ and $\nG\nL_{2}(k_{K_2})$, the former of which is included in the above construction (by taking $\lambda \in X_1(\un{T}_\nU) = X^*(\un{T}_\nU) \cap X_1(\un{T}$)), and the latter of which we denote with primed notation.  

\subsubsection{}
Recall that the fundamental $p$-alcove of  $X^*(\un{T})$ is defined as
\[
\{\lambda\in X^*(\un{T})\otimes_{\bbZ}\bbR\ :\ 0<\langle \lambda+\eta,\alpha^\vee\rangle<p ~\textnormal{for all}~ \alpha\in \un{\Phi}^+\}
\]
and that $\mu \in X^*(\un{T})$ is said to be $N$-deep (in the fundamental $p$-alcove) if $N < \langle \mu + \eta, \alpha^\vee\rangle < p - N$ for all $\alpha \in \un{\Phi}^+$.  We then say that a Serre weight $F$ is $N$-deep if $F\cong F(\lambda)$ for some $\lambda\in X_1(\un{T})$ which is $N$-deep.  (This notion is independent of the choice of $\lambda$.)

\subsubsection{}
Given $(w,\mu)\in \un{W}\times X^*(\un{T})$ such that $\mu-\eta$ is $0$-deep, we have an associated Deligne--Lusztig representation $R_w(\mu)$ as in \cite[\S 9.2]{GHS} (where the representation $R_w(\mu)$ is denoted $R(w,\mu)$).  It is a genuine representation  of $\un{G}_0(\bbF_p)$ over $\overline{\bbQ}_p$ and we will always assume that $E$ is large enough so that $R_w(\mu)$ can be realized over $E$.  Given $N \geq 0$, we say that a Deligne--Lusztig representation $R$ is \textit{$N$-generic} if $R\cong R_w(\mu+\eta)$ with $\mu$ being $N$-deep.

We will often use the following fact from \cite[\S 4.1]{herzig:duke}: if $(w,\mu)\in \un{W}\times X^*(\un{T})$ is such that $\mu-\eta$ is $0$-deep, and if $(\nu,s)\in X^*\rtimes \un{W}$ then 
$$R_{w}(\mu)\cong R_{sw\pi(s)^{-1}}\big(s(\mu)+p\pi^{-1}(\nu)-sw\pi(s)^{-1}(\nu)\big).$$
Conversely, if $R_{w}(\mu)\cong R_{w'}(\mu')$ with $\mu - \eta, \mu' - \eta$ being $0$-deep, then $w'=sw\pi(s)^{-1}$ and $\mu'=s(\mu)+p\pi^{-1}(\nu)-sw\pi(s)^{-1}(\nu)$ for some $(\nu,s)\in X^*\rtimes \un{W}$.

\subsubsection{}
As with Serre weights, we have analogous constructions for $\nU_{1,1}(k_K)$ and $\nG\nL_2(k_{K_2})$.  In particular, we may naturally identify Serre weights and Deligne--Lusztig representations of $\un{G}_0(\bbF_p)$ on which the subgroup $\imath(k_K^\times)$ acts trivially with Serre weights and Deligne--Lusztig representations of $\un{\nU}_0(\bbF_p) = \nU_{1,1}(k_K)$.  Further, we will occasionally view Serre weights and Deligne--Lusztig representations of $\un{G}_0(\bbF_p)$ as representations of $\un{G}_0(\bbZ_p)$ by inflation (similarly for $\nU_{1,1}$ and $\nG\nL_2$).

\subsubsection{}
\label{twist-by-epsilon}
Given $\mu \in X_1(\un{T}_\nU)$, we define a \textit{base change map} from Serre weights of $\nU_{1,1}(k_K)$ to Serre weights of $\nG\nL_2(k_{K_2})$ by
$$\BC\big(F(\mu)\big) := F'(\BC(\mu)) = F'(\mu, -\un{\fw}(\mu)).$$
The map $\BC$ is well-defined and injective on isomorphism classes.  Furthermore, recall from \cite[\S 3C7]{koziolmorra} that we have a map $\epsilon$ on isomorphism classes of representations of $\textnormal{GL}_2(k_{K_2})$ induced by twisting by the automorphism
$$g \longmapsto \left(\begin{pmatrix}0 & 1 \\ -1 & 0 \end{pmatrix}g^{-\top}\begin{pmatrix}0 & 1 \\ -1 & 0 \end{pmatrix}^{-1}\right)^{(q)}.$$
On Serre weights, this automorphism takes the form 
$$\epsilon\left(F'(\mu, \mu')\right) \cong F'(-\underline{\fw}(\mu'), -\underline{\fw}(\mu)),$$
where $\mu,\mu' \in X_1(\underline{T}_\nU)$ (see \cite[\S 3E2]{koziolmorra}).  By \cite[Lem. 3.25]{koziolmorra}, for a Serre weight $F'$ of $\nG\nL_2(k_{K_2})$, we have $\epsilon(F') \cong F'$ if and only if $F'$ is in the image of the base change map.

Similarly, if $w \in \un{W}$ and $\mu \in X^*(\un{T}_\nU)$ with $\mu - \eta$ being $0$-deep, we define a \textit{base change map} from Deligne--Lusztig representations of $\nU_{1,1}(k_K)$ to Deligne--Lusztig representations of $\nG\nL_2(k_{K_2})$ by
$$\BC\big(R_w(\mu)\big) = R'_{(w,w)}(\BC(\mu)) = R'_{(w,w)}(\mu, -\un{\fw}(\mu))$$
(see \cite[Eqs. (3C.1), (3C.2), \S 3C6]{koziolmorra}).  The map $\BC$ is well-defined and injective on isomorphism classes.  On Deligne--Lusztig representations, twisting by $\epsilon$ takes the form
$$\epsilon\big(R'_{(w,w')}(\mu,\mu')\big) \cong R'_{(w',w)}(-\un{\fw}(\mu'),-\un{\fw}(\mu)),$$
where $w,w' \in \un{W}$ and $\mu,\mu' \in X^*(\un{T}_\nU)$ with both $\mu - \eta$ and $\mu' - \eta$ being $0$-deep (see \cite[\S 3C7]{koziolmorra}).  By \cite[Lem. 3.12]{koziolmorra}, for a Deligne--Lusztig representation $R'$ of $\nG\nL_2(k_{K_2})$, we have $\epsilon(R') \cong R'$ if and only if $R'$ is in the image of the base change map.

\subsection{$L$-parameters}

\subsubsection{}
\label{subsub:Lparamdefs}
We begin with a definition.

\begin{defn}
  Let $R$ be a topological $\bbZ_p$-algebra.  
  \begin{enumerate}
    \item An \textit{$L$-parameter (with $R$-coefficients)} is a continuous homomorphism $\Gamma_{K} \longrightarrow {}^C\nU_{1,1}(R)$ which is compatible with the projection to $\Gal(K_2/K)$.
    \item Given an $L$-parameter $\rho:\Gamma_K \longrightarrow {}^C\nU_{1,1}(R)$, the homomorphism $\widehat{\imath}\circ \rho :\Gamma_K \longrightarrow R^\times$ is called the \textit{multiplier of $\rho$}.
    \item An \textit{inertial $L$-parameter} is a continuous homomorphism $I_{K} \longrightarrow H^\vee(R)$ which extends to an $L$-parameter.
    \item Two (inertial) $L$-parameters are \textit{equivalent} if they are $H^\vee(R)$-conjugate.
    \item Given an $L$-parameter $\rho: \Gamma_K \longrightarrow {}^C\nU_{1,1}(R)$, we define its \emph{base change} $\BC(\rho):\Gamma_{K_2} \longrightarrow \nG\nL_2(R)$ as the composition of the restriction $\rho|_{\Gamma_{K_2}}: \Gamma_{K_2} \longrightarrow (\nG\nL_2\times\bG_m)(R)$ with the projection $(\nG\nL_2\times\bG_m)(R) \longtwoheadrightarrow \nG\nL_2(R)$.
    \item Similarly, given a $\cG_2$-valued $L$-parameter $\varrho: \Gamma_K \longrightarrow \cG_2(R)$, we define its \emph{base change} $\BC'(\varrho):\Gamma_{K_2} \longrightarrow {}\nG\nL_2(R)$ as the composition of the restriction $\varrho|_{\Gamma_{K_2}}:\Gamma_{K_2} \longrightarrow (\nG\nL_2\times\bG_m)(R)$ with the projection $(\nG\nL_2\times\bG_m)(R) \longtwoheadrightarrow \nG\nL_2(R)$.  (See \cite[\S 4A3]{koziolmorra} for the relationship between $\BC(\rho)$ and $\BC'(\varrho)$.)
  \end{enumerate}
\end{defn}
(This is \cite[Defs. 4.1, 4.3]{koziolmorra}, employing the equivalences between $H^\vee(R)$-conjugacy classes of parameters with source $\Gamma_K$ and $\un{G}^\vee(R)$-conjugacy classes of parameters with source $\Gamma_{\bbQ_p}$; see \cite[Lems. 9.4.1, 9.4.5]{GHS}.)

We have similar notions when replacing $\Gamma_K$ by $\Gamma_{F^+}$ with $F^+$ a global field with a place $v$ satisfying $F^+_v\cong K$ (in which case we talk about \emph{global} $L$-parameter).

\subsubsection{}

Given $(w,\mu)\in\un{W}\times X^*(\un{T})$, we define a tamely ramified inertial $L$-parameter $\tau(w,\mu):I_K \longrightarrow H^\vee(\cO)$ as in \cite[\S 4A4]{koziolmorra} (using the Teichm\"uller lift of $\omega_{2f}$ in \textit{op. cit.}), and let $\overline{\tau}(w,\mu)$ denote its mod $p$ reduction.  
Using the same argument as in \cite[Lem. 4.4]{koziolmorra}, if $\mu \in X^*(\un{T}_\nU)$ and $\eta_{\underline{c}} := \sum_{j \in \cJ}c_j\eta_j$, then we have an equivalence of tame inertial types $I_{K_2} \longrightarrow \nG\nL_2(\cO)$:
\begin{equation}
  \label{BC-of-inertial-type}
  \BC\big(\tau(w,\mu + \eta_{\underline{c}})\big)  \cong  \tau'\Big((w,w),\BC(\mu) + \eta'_{(\underline{c},\underline{c})}\Big),  
\end{equation}
where $\eta'_{(\underline{c},\underline{c})} := \sum_{0 \leq j \leq f - 1}c_j\eta'_j + \sum_{0 \leq j \leq f - 1}c_j \eta'_{j + f}$, and where the tame inertial type on the right is defined as in \cite[Def. 2.3.1]{BHHMS}.

\subsubsection{}
\label{subsub:SW}

Given a tamely ramified inertial $L$-parameter $\overline{\tau}: I_K \longrightarrow H^\vee(\bbF)$, we define an associated set of Serre weights $\nW^?(\overline{\tau})$, as in \cite[Def.~9.2.5]{GHS} (see also the discussion of this set in \cite[\S 4B]{koziolmorra}).  Given a tamely ramified $L$-parameter $\rhobar:\Gamma_K \longrightarrow {}^C\nU_{1,1}(\bbF)$, we define $\nW^?(\rhobar) := \nW^?(\rhobar|_{I_K})$.

\subsubsection{}

Given $N \geq 0$, we say that a tamely ramified inertial $L$-parameter $\tau:I_K \longrightarrow H^\vee(\cO)$ is \textit{$N$-generic} if $\tau\cong \tau(w,\mu + \eta)$ for some $(w,\mu)\in\un{W}\times X^*(\un{T})$ with $\mu$ being $N$-deep.  When $\tau$ is $0$-generic, we call the pair $(w,\mu)$ a \textit{lowest alcove presentation} of $\tau$.  Analogously, given a tamely ramified $L$-parameter $\rhobar:\Gamma_K \longrightarrow {}^C\nU_{1,1}(\bbF)$, we say that $\rhobar$ is \textit{$N$-generic} if $\rhobar|_{I_K}$ is $N$-generic.

\subsubsection{}

Suppose $R$ is a topological $\bbZ_p$-algebra, and let $\rho: \Gamma_K \longrightarrow {}^C\nU_{1,1}(R)$ denote an $L$-parameter.  For $k \in \bbZ$, we let $\rho(k)$ denote the \textit{$k^{\textnormal{th}}$ Tate twist of $\rho$}, i.e., the unique $L$-parameter $\Gamma_K \longrightarrow {}^C\nU_{1,1}(R)$ which satisfies $\BC(\rho(k)) = \BC(\rho)\otimes_R \varepsilon^k$ and $\widehat{\imath}\circ \rho(k) = (\widehat{\imath}\circ \rho){\varepsilon}^{2k}$.

\subsection{Inertial local Langlands}
\label{sec:ILL}
Recall that a \textit{tame inertial type} is a homomorphism $\tau':I_{K_2}\longrightarrow \nG\nL_2(\cO)$ with open kernel, which factors through the tame quotient of $I_{K_2}$, and which extends to a representation of the Weil group of $K_2$.
If $\tau'$ satisfies the condition $(\tau')^{\varphi^{f}}\cong \tau'^\vee$ (i.e.,~$\tau'$ is \emph{conjugate self-dual}) we can associate to $\tau'$ a smooth irreducible representation $\sigma(\tau')$ of $\nU_{1,1}(\cO_K)$ over $E$ (in a manner compatible with Henniart's inertial local Langlands correspondence via base change, see \cite[Def. 4.10]{koziolmorra}).  The representation $\sigma(\tau')$ satisfies results towards an inertial local Langlands correspondence for $\nU_{1,1}(K)$ (\cite[Thm. 4.11]{koziolmorra}).

Suppose $\mu \in X^*(\un{T}_\nU)$ is such that $\mu - \eta$ is 1-deep.  Then the same is true of $\BC(\mu) \in X^*(\un{T})$, and $\tau' := \tau'((s,s),\BC(\mu))$ is $1$-generic. Thus, by \cite[Cor. 2.3.5]{LLL} and \cite[\S 2.2]{BHHMS}, we have
\begin{equation}
  \label{ILLC'-explicit}
  \sigma'(\tau') \cong R'_{(s,s)}(\BC(\mu)),
\end{equation}
where the left-hand side denotes the representation of $\nG\nL_2(\cO_{K_2})$ associated to $\tau'$ by Henniart's inertial local Langlands correspondence.  By the injectivity of the base change map on Deligne--Lusztig representations (\cite[\S 3C6]{koziolmorra}), we obtain
\begin{equation}
  \label{ILLC-explicit}
  \sigma(\tau') \cong R_s(\mu).
\end{equation}

\section{Some general algebra}
\label{sec:transfer}

The goal of this section is to collect some general results about modules over certain Artinian rings.  This will allow us to deduce results about multiplicities and extensions for the group $\nU_{1,1}$ from the analogous results for $\nG\nL_2$.

\subsection{Setup}  
\label{subsec:setup}
We put ourselves in the following general setting.

Let $\Lambda'$ denote an Artinian algebra over $\bbF$, $\Lambda \subset \Lambda'$ a subalgebra, and $Z$ a finite subgroup of $\cZ(\Lambda')^\times$, the group of units of the center of $\Lambda'$.  We assume these data satisfy the following hypotheses:
\begin{itemize}
\item The field $\bbF$ is sufficiently large (so that, in particular, all simple modules are absolutely simple).
\item The order of $Z$ is prime to $p$.
\item There exists a set-theoretic section $y \longmapsto \widetilde{y}$ to the surjection $Z \longtwoheadrightarrow Z/(\Lambda \cap Z) =: Y$ such that: 
\begin{itemize}
\item $\widetilde{1_Y} = 1_Z = 1_{\Lambda'}$.
\item $\Lambda'$ is free as a $\Lambda$-module, on the left and right, with basis given by $\{\widetilde{y}\}_{y \in Y}$.
\item For $y_1, y_2 \in Y$, we have $\widetilde{y_1}\Lambda = \Lambda\widetilde{y_1}$ and $\widetilde{y_1}\cdot\widetilde{y_2}\Lambda = \widetilde{y_1y_2}\Lambda$.  (Note: it is not necessarily true that $\widetilde{y_1}\cdot \widetilde{y_2} = \widetilde{y_1y_2}$ in $\Lambda'$.)
\end{itemize}
\end{itemize}
The last bullet point means that $\Lambda'$ is equal to a crossed product algebra $\Lambda*Y$, cf. \cite[\S I.5.8]{mcconnellrobson}.  Note also that this definition implies $\Lambda'$ is generated by $\Lambda$ and $Z$.

\subsection{Examples}  
\label{exs}

Before proceeding with general theory, we give some examples of the above setup that we will use later.

We will make use of the decompositions \eqref{U11decomp:finite} and \eqref{GL2decomp}.  

\begin{enumerate}
\item \label{pt1}
We can take 
$$\Lambda'  =  \bbF[\textnormal{GU}_{1,1}(k_K)], \qquad \Lambda  = \bbF[\textnormal{U}_{1,1}(k_K)], \qquad Z  =  k_{K_2}^\times,$$
so that, using the decomposition \eqref{U11decomp:finite}, we get
$$\Lambda' \cong \bbF[\textnormal{U}_{1,1}(k_K)]\otimes_{\bbF[\textnormal{U}_1(k_K)]} \bbF[k_{K_2}^\times] \cong \Lambda*(k_{K_2}^\times/\textnormal{U}_1(k_K)) \cong \Lambda*k_K^\times.$$

\item \label{pt2}
We can take
$$ \Lambda'  =  \bbF[\textnormal{GU}_{1,1}(k_K)], \qquad \Lambda  =  \bbF[\textnormal{GL}_{2}(k_K)], \qquad Z  =  k_{K_2}^\times,$$
so that, using the decomposition \eqref{GL2decomp}, we get 
$$\Lambda' \cong \bbF[\textnormal{GL}_2(k_K)]\otimes_{\bbF[k_K^\times]} \bbF[k_{K_2}^\times] \cong \Lambda*(k_{K_2}^\times/k_K^\times) \cong  \Lambda*\textnormal{U}_1(k_K).$$
\end{enumerate}

In the next two examples, we use the notation $\sfK_{\nJ}$, for $\nJ \in \{\nG, \nU, \nG\nU\}$, to denote the group of $\cO_K$-points of the group scheme $J_{/\cO_K} \in \{\nG\nL_{2/\cO_K}, \nU_{1,1/\cO_K}, \nG\nU_{1,1/\cO_K}\}$.  Also, $\sfK_{\nJ,1}$ denotes the pro-$p$ radical of $\sfK_{\nJ}$, $Z_{\nJ}$ denotes the center of $\sfK_{\nJ}$, and $Z_{\nJ,1}$ denotes the maximal pro-$p$ subgroup of $Z_{\nJ}$.  Note that, since $p > 2$, we have $\sfK_{\nJ, 1} \cong \nS\nL_2(\cO_K)_1 \times Z_{\nJ,1}$, where $\nS\nL_2(\cO_K)_1$ denotes the principal congruence subgroup of $\nS\nL_2(\cO_K)$.  We write $\bbF\llbracket \sfK_{\nJ}\rrbracket$ for the Iwasawa algebra of $\sfK_{\nJ}$ over $\bbF$, i.e.,~the inverse limit $\underset{\sfK'\triangleleft \sfK_{\nJ}}{\varprojlim}\bbF[\sfK_{\nJ}/\sfK']$ where the limit runs over the normal compact open subgroups $\sfK'$ of $\sfK_{\nJ}$.

\begin{enumerate}
\setcounter{enumi}{2}
\item \label{pt3}
Let $\fm_{\nG\nU}^n$ be the two-sided ideal of $\bbF\llbracket \sfK_{\nG\nU}\rrbracket$ generated by $\langle k - 1: k \in \sfK_{\nG\nU,1}\rangle^n$ and $\langle z - 1: z \in Z_{\nG\nU,1}\rangle$, and let $\fm_{\nU}^n$ be the two-sided ideal of $\bbF\llbracket \sfK_{\nU}\rrbracket$ generated by $\langle k - 1: k \in \sfK_{\nU,1}\rangle^n$ and $\langle z - 1: z \in Z_{\nU,1}\rangle$.  Set 
$$\Lambda' = \bbF\llbracket \sfK_{\nG\nU}\rrbracket/\fm_{\nG\nU}^n, \qquad \Lambda = \bbF\llbracket \sfK_{\nU}\rrbracket/\fm_{\nU}^n,$$ 
and let $Z$ denote the image of $\cO_{K_2}^\times$ in $\Lambda'$, which is isomorphic to $k_{K_2}^\times$.  Then decomposition \eqref{U11decomp} implies
$$\Lambda' = (\bbF\llbracket \sfK_{\nU}\rrbracket/\fm_{\nU}^n) \otimes_{\bbF[\textnormal{U}_1(k_K)]} \bbF[k_{K_2}^\times] = \Lambda*(k_{K_2}^\times/\textnormal{U}_1(k_K)) = \Lambda*k_K^\times.$$

\item \label{pt4}
Let $\fm_{\nG}^n$ be the two-sided ideal of $\bbF\llbracket \sfK_{\nG}\rrbracket$ generated by $\langle k - 1: k \in \sfK_{\nG,1}\rangle^n$ and $\langle z - 1: z \in Z_{\nG,1} \rangle$.  Set 
$$\Lambda' = \bbF\llbracket \sfK_{\nG\nU}\rrbracket/\fm_{\nG\nU}^n, \qquad \Lambda = \bbF\llbracket \sfK_{\nG}\rrbracket/\fm_{\nG}^n,$$ 
and let $Z$ denote the image of $\cO_{K_2}^\times$ in $\Lambda'$, which is isomorphic to $k_{K_2}^\times$.  Then decomposition \eqref{GL2decomp} implies
$$\Lambda' = (\bbF\llbracket \sfK_{\nG}\rrbracket/\fm_{\nG}^n) \otimes_{\bbF[k_K^\times]} \bbF[k_{K_2}^\times] = \Lambda*(k_{K_2}^\times/k_K^\times) =  \Lambda*\textnormal{U}_1(k_K).$$
\end{enumerate}

In the next two examples, we let $I_{\nJ}$, for $\nJ \in \{\nG, \nU, \nG\nU\}$, denote the upper triangular Iwahori subgroup of $\sfK_{\nJ}$, and $I_{\nJ,1}$ its pro-$p$-Iwahori subgroup.  Since $p > 2$, we have $I_{\nJ,1} \cong I_{\nS,1}\times Z_{\nJ,1}$, where $I_{\nS,1}$ denotes the upper triangular pro-$p$-Iwahori subgroup of $\nS\nL_2(\cO_K)$.  
As above, we write $\bbF\llbracket I_{\nJ}\rrbracket$ to denote the Iwasawa algebra of $I_{\nJ}$ over $\bbF$.

\begin{enumerate}
\setcounter{enumi}{4}
\item \label{pt5}
Let $\fm_{\nG\nU,\textnormal{Iw}}^n$ denote the two-sided ideal of $\bbF\llbracket I_{\nG\nU} \rrbracket$ generated by $\langle i - 1: i \in I_{\nG\nU,1}\rangle^n$ and $\langle z - 1: z\in Z_{\nG\nU,1}\rangle$, and let $\fm_{\nU,\textnormal{Iw}}^n$ denote the two-sided ideal of $\bbF\llbracket I_{\nU} \rrbracket$ generated by $\langle i - 1: i \in I_{\nU,1}\rangle^n$ and $\langle z - 1: z\in Z_{\nU,1}\rangle$.  Set 
$$\Lambda' = \bbF\llbracket I_{\nG\nU} \rrbracket/\fm_{\nG\nU,\textnormal{Iw}}^n, \qquad \Lambda = \bbF\llbracket I_\nU \rrbracket/\fm_{\nU,\textnormal{Iw}}^n,$$
and let $Z$ denote the image of $\cO_{K_2}^\times$ in $\Lambda'$, which is isomorphic to $k_{K_2}^\times$.  Then the ``Iwahori version'' of decomposition \eqref{U11decomp} implies 
$$\Lambda' = (\bbF\llbracket I_\nU \rrbracket/\fm_{\nU,\textnormal{Iw}}^n)\otimes_{\bbF[\nU_1(k_K)]} \bbF[k_{K_2}^\times] = \Lambda * (k_{K_2}^\times/\nU_1(k_K)) = \Lambda * k_K^\times.$$

\item \label{pt6}
Let $\fm_{\nG,\textnormal{Iw}}^n$ be the two-sided ideal of $\bbF\llbracket I_\nG \rrbracket$ generated by $\langle i - 1: i \in I_{\nG,1}\rangle^n$ and $\langle z - 1: z\in Z_{\nG,1}\rangle$.  Set
$$\Lambda' = \bbF\llbracket I_{\nG\nU} \rrbracket/\fm_{\nG\nU,\textnormal{Iw}}^n, \qquad \Lambda = \bbF\llbracket I_\nG \rrbracket/\fm_{\nG,\textnormal{Iw}}^n,$$
and let $Z$ denote the image of $\cO_{K_2}^\times$ in $\Lambda'$, which is isomorphic to $k_{K_2}^\times$.  Then the ``Iwahori version'' of decomposition \eqref{GL2decomp} implies 
$$\Lambda' = (\bbF\llbracket I_\nG \rrbracket/\fm_{\nG,\textnormal{Iw}}^n)\otimes_{\bbF[k_K^\times]} \bbF[k_{K_2}^\times] = \Lambda * (k_{K_2}^\times/k_K^\times) = \Lambda * \nU_1(k_K).$$
\end{enumerate}

The results below will allow us to pass information back and forth from $\nG\nL_2$ to $\nU_{1,1}$ via $\nG\nU_{1,1}$ in various contexts.  

\vspace{15pt}

\subsection{Module theory}

We return to the general setting of Subsection \ref{subsec:setup}.  Unless indicated otherwise, ``module'' will mean ``left module.''

\begin{lemma}
\label{resprops}
\hfill
\begin{enumerate}
\item Let $V$ be a $\Lambda'$-module.  Then $V$ is simple (resp., projective) if and only if $V|_\Lambda$ is simple (resp., projective). \label{resprops-2}
\item Injective $\Lambda'$-modules remain injective upon restriction to $\Lambda$. \label{resprops-3}
\item Let $V$ be a $\Lambda'$-module.  Then
$$\soc^i_{\Lambda'}(V)|_\Lambda \cong \soc^i_{\Lambda}(V|_{\Lambda}) \qquad \textnormal{and} \qquad \rad^i_{\Lambda'}(V)|_{\Lambda} \cong \rad^i_{\Lambda}(V|_\Lambda),$$ 
where $\soc_\bullet^i$ and $\rad_\bullet^i$ denote the socle and radical filtrations, respectively (see \cite[\S 1]{alperin}). \label{resprops-4}
\item Let $V$ be a $\Lambda'$-module.  Then
$$\inj_{\Lambda'}(V)|_\Lambda \cong \inj_{\Lambda}(V|_\Lambda) \qquad \textnormal{and} \qquad \proj_{\Lambda'}(V)|_{\Lambda} \cong \proj_{\Lambda}(V|_{\Lambda}),$$ 
where $\inj_\bullet$ and $\proj_\bullet$ denote the injective envelope and projective cover, respectively.
\label{resprops-5}
\end{enumerate}
\end{lemma}

\begin{proof}
\begin{enumerate}
\item The claim for simplicity follows from the fact that simple modules have central characters, and that $\Lambda'$ is generated by $\Lambda$ and $Z$.  For projectivity, we use \cite[Thm. 7.5.6(ii)]{mcconnellrobson}.
\item Since $\Lambda'$ is free of finite rank over $\Lambda$, the functor $\Lambda'\otimes_\Lambda -$ is exact.  If $V$ is injective over $\Lambda'$, the composite functor $\Hom_{\Lambda'}(\Lambda'\otimes_\Lambda -, V)$ is also exact.  Since the latter functor is naturally isomorphic to $\Hom_\Lambda(-, V|_\Lambda)$, we conclude that $V|_\Lambda$ is injective.  
\item Since $|Z|$ is invertible in $\bbF$, by \cite[Thm. 1.4]{rr} we have $\mathfrak{J}_{\Lambda'} = \mathfrak{J}_{\Lambda} \Lambda'$, where $\mathfrak{J}_\bullet$ denotes the Jacobson radical.  The claim then follows from the definitions of the socle and radical filtrations given in \cite[\S 1]{alperin}, noting that we can omit the finite-dimensionality assumption on $V$ of \textit{op. cit.} by invoking \cite[\S 4 Exer. 18; Prop. 24.4(2)]{lam}.
\item By point \eqref{resprops-3}, we have that $I := \inj_{\Lambda'}(V)|_{\Lambda}$ is injective, so that $I \cong \inj_{\Lambda}(\soc_{\Lambda}(I))$ (we are using here the fact that for a module $M$ over an Artinian ring $\Lambda$, the extension $\soc_\Lambda(M) \subset M$ is essential; see \cite[\S 19, Exer. 28(2)]{lam2}).  By point \eqref{resprops-4}, we have 
$$\soc_{\Lambda}(I) = \soc_{\Lambda}(\inj_{\Lambda'}(V)|_{\Lambda}) \cong \soc_{\Lambda'}(\inj_{\Lambda'}(V))|_{\Lambda} \cong \soc_{\Lambda'}(V)|_{\Lambda} \cong \soc_{\Lambda}(V|_{\Lambda}),$$ 
so that $I \cong \inj_{\Lambda}(\soc_{\Lambda}(V|_{\Lambda})) \cong \inj_{\Lambda}(V|_{\Lambda})$.  The argument for projective covers is similar.  
\end{enumerate}
\end{proof}

\begin{lemma}
  \label{decomp:centralchar-charp}
Let $V$ be a $\Lambda'$-module.  Then $V$ decomposes as a direct sum of $\Lambda'$-modules
$$V \cong \bigoplus_{\chi:Z \rightarrow \bbF^\times} V^{Z = \chi},$$
where $V^{Z = \chi} := \{v \in V: z\cdot v = \chi(z)v\}$.
\end{lemma}

\begin{proof}
Let $\chi$ be a character of $Z$.  Since $|Z|$ is prime to $p$, we can define
$$e_\chi := \frac{1}{|Z|}\sum_{z\in Z}\chi(z)z^{-1} \in \bbF[Z] \subset \Lambda'.$$
The set $\{e_\chi\}_{\chi:Z \rightarrow \bbF^\times}$ gives a set of mutually orthogonal idempotents whose sum is 1.  Therefore, we have
$$V = \bigoplus_{\chi:Z \rightarrow \bbF^\times} e_\chi V.$$
Since the $e_\chi$ are central, the space $e_\chi V$ is stable by $\Lambda'$.  One easily checks that $e_\chi V = V^{Z = \chi}$.
\end{proof}

Let $V$ be a $\Lambda'$-module, and suppose $Z$ acts on $V$ by a character (in particular, this occurs when $V$ is simple).  We let $\omega_V:Z \longrightarrow \bbF^\times$ denote this character of $Z$.

\begin{cor}\label{extcor}
Suppose $U$ is a $\Lambda$-module.  Then the $\Lambda$-action on $U$ extends to an action of $\Lambda'$, that is, $U$ can (non-uniquely) be given the structure of a $\Lambda'$-module.
\end{cor}

\begin{proof}
Mimicking the above proof with $\Lambda'$ replaced by $\Lambda$ and $Z$ replaced by $\Lambda \cap Z$, we may decompose $U$ as $\bigoplus_{\psi:\Lambda\cap Z \rightarrow \bbF^\times} U^{\Lambda \cap Z = \psi}$.  Therefore it suffices to prove the claim when $\Lambda \cap Z$ acts on $U$ by a character $\omega_U: \Lambda \cap Z \longrightarrow \bbF^\times$.  In this case, we may choose any $\chi: Z \longrightarrow \bbF^\times$ for which $\chi|_{\Lambda \cap Z} = \omega_U$, and let $Z$ act on $U$ via $\chi$.  
\end{proof}

\begin{lemma}\label{resext}
Let $V$ and $W$ be two $\Lambda'$-modules on which $Z$ acts by a character, and suppose $\omega_V = \omega_W$.  Then the restriction map gives an isomorphism
$$\Ext^i_{\Lambda'}(V,W) \stackrel{\sim}{\longrightarrow} \Ext^i_{\Lambda}(V|_{\Lambda}, W|_{\Lambda}).$$
\end{lemma}

\begin{proof}
Suppose $W'$ is a $\Lambda'$-module on which $Z$ acts by a character, such that $\omega_V = \omega_{W'}$.  The space $\Hom_{\Lambda}(V|_{\Lambda}, W'|_{\Lambda})$ is then a $Z$-representation: we have 
$$(z\cdot\varphi)(v) = z\cdot\big(\varphi(z^{-1}\cdot v)\big)$$
and in particular
$$\Hom_{\Lambda'}(V,W')  =  \Hom_{\Lambda}(V|_{\Lambda}, W'|_{\Lambda})^Z  =  \Hom_{\Lambda}(V|_{\Lambda}, W'|_{\Lambda}),$$
where the last equality follows from the fact that $\omega_V = \omega_{W'}$.

Now let $0 \longrightarrow W \longrightarrow I^0 \longrightarrow I^1 \longrightarrow \ldots$ denote an injective resolution of $W$ as a $\Lambda'$-module.  By decomposing each $I^i$ as a direct sum indexed by characters of $Z$, we may assume that $Z$ acts on each $I^i$ by a character equal to $\omega_W$.  Further, by restricting and using Lemma \ref{resprops}\eqref{resprops-3}, we see that $0 \longrightarrow W|_{\Lambda} \longrightarrow I^\bullet|_{\Lambda}$ is an injective resolution of $W|_{\Lambda}$ as an $\Lambda$-module.  Therefore, the previous paragraph implies
$$\Ext_{\Lambda}^i(V|_{\Lambda}, W|_{\Lambda}) = H^i\big(\Hom_{\Lambda}(V|_{\Lambda}, I^\bullet|_{\Lambda})\big) = H^i\big(\Hom_{\Lambda'}(V, I^\bullet)\big) = \Ext^i_{\Lambda'}(V,W).$$
\end{proof}

\begin{cor}\label{multcor}
Suppose $V$ is a finitely generated $\Lambda'$-module on which $Z$ acts by a character, and $\sigma$ is a simple $\Lambda'$-module such that $\omega_\sigma = \omega_V$.  Then the same is true upon restriction to $\Lambda$, and we have
$$[V:\sigma] = [V|_{\Lambda}: \sigma|_\Lambda],$$
where $[V:\sigma]$ denotes the multiplicity with which $\sigma$ appears as a Jordan--H\"older factor of $V$ (and analogously for the restrictions).
\end{cor}

\begin{proof}
We have
\begin{eqnarray*}
[V:\sigma] & = & \dim_{\bbF}\Big(\Hom_{\Lambda'}\big(V,\inj_{\Lambda'}(\sigma)\big)\Big) \\
 & \stackrel{\textnormal{Lem.}~ \ref{resext}}{=} & \dim_{\bbF}\Big(\Hom_{\Lambda}\big(V|_{\Lambda}, \inj_{\Lambda'}(\sigma)|_{\Lambda}\big)\Big) \\
 & \stackrel{\textnormal{Lem.}~ \ref{resprops}\eqref{resprops-5}}{=} & \dim_{\bbF}\Big(\Hom_{\Lambda}(V|_{\Lambda}, \inj_{\Lambda}(\sigma|_{\Lambda})\big)\Big) \\
 & = & [V|_{\Lambda} : \sigma|_{\Lambda}].
\end{eqnarray*}
\end{proof}

\subsection{Analogs in characteristic 0}
\label{sec:transfer:char0}

We now record some analogous results for modules over a discrete valuation ring.  Since the necessary results are only used in a few places, we have not strived for the maximum level of generality.

We suppose we are in the following situation.  Let $\bbF$ and $\cO$ be as in Subsection \ref{notation:fields}.  (The result below holds more generally for any discrete valutation ring with residue field $\bbF$.)  We let $\cG$ denote a finite group, and $\cH \trianglelefteq \cG$ a normal subgroup.  We assume these data satisfy the following hypotheses:
\begin{itemize}
  \item The field $\bbF$ is sufficiently large (i.e., $\bbF$ contains the $m^{\textnormal{th}}$ roots of unity, where $m$ is the least common multiple of the orders of elements of $\cG$; in particular every simple $\bbF[\cG]$-module is absolutely simple).
  \item The order of $Z(\cG)$, the center of $\cG$, is prime to $p$. 
  \item The finite group $\cG$ is generated by $\cH$ and $Z(\cG)$.
\end{itemize}
The last point implies that $\cO[\cG]$ is equal to a crossed product algebra.  More precisely, let us set $Y := Z(\cG)/(\cH \cap Z(\cG))$.  Then we have $\cO[\cG] \cong \cO[\cH]*Y$, the latter defined as in Subsection \ref{subsec:setup}.

\begin{lemma}\hfill
  \label{integralres}
  \begin{enumerate}
    \item Any simple $\cO[\cG]$-module comes via inflation from the surjection $\cO[\cG] \longtwoheadrightarrow \bbF[\cG]$.  Similarly for $\cO[\cH]$.  \label{integralres-1}
    \item Any indecomposable $\cO[\cG]$-module admits a central character.  Similarly for $\cO[\cH]$.  \label{integralres-2}
    \item Let $V$ be a $\cO[\cG]$-module.  Then $V$ is simple (resp., projective) if and only if $V|_{\cO[\cH]}$ is simple (resp., projective).  \label{integralres-3}
  \end{enumerate}
\end{lemma}

\begin{proof}
  \begin{enumerate}
    \item Let $\varpi$ denote a uniformizer of $\cO$.  Any simple $\cO[\cG]$-module $V$ is finitely generated over $\cO$, which implies that $\varpi V$ is a proper submodule of $V$, and therefore equal to 0.
    \item Let $P$ denote an indecomposable $\cO[\cG]$-module.  As in the proof of Lemma \ref{decomp:centralchar-charp}, we have an $\cO[\cG]$-equivariant decomposition
    $$P \cong \bigoplus_{\chi:Z(\cG) \rightarrow \cO^\times} P^{Z(\cG) = \chi},$$
    the direct sum running over characters of $Z(\cG)$.  The claim now follows from indecomposability of $P$.
    \item The claim about simplicity follows from item \eqref{integralres-1} and Lemma \ref{resprops}\eqref{resprops-2}.  For projectivity, we use \cite[Thm. 7.5.6(ii)]{mcconnellrobson}
  \end{enumerate}
\end{proof}

\section{Extension graph combinatorics}
\label{appendix:EGC}

We discuss combinatorics related to the extension graph, as in \cite[\S 2.4]{BHHMS}.  We follow the conventions of \textit{op. cit.,} which differ by an $\eta$-shift from those of \cite{LMS} and \cite{koziolmorra}.

We will consider the extension graphs for both $\underline{G} = \textnormal{Res}_{\cO_K/\bbZ_p}(H)\times_{\bbZ_p}\cO$ and $\underline{G}' = \textnormal{Res}_{\cO_{K_2}/\bbZ_p}(\nG\nL_2)\times_{\bbZ_p}\cO$.  Let
$$\Lambda_{\textnormal{wt}} := X^*(\underline{T})/X^0(\underline{T}),\qquad \Lambda'_{\textnormal{wt}} := X^*(\underline{T}')/X^0(\underline{T}')$$
denote the weight lattices for $\underline{G}^{\der}$ and $\underline{G}'^{\der}$, respectively, and identify them with $\bbZ^{\cJ}$ and $\bbZ^{\cJ'}$.  Given $\mu \in X^*(\underline{T})$ (resp., $\mu' \in X^*(\underline{T}')$), we define recentered extension graphs by
\begin{eqnarray*}
  \Lambda^{\mu}_{\textnormal{wt}} & := & \left\{\omega \in \Lambda_{\textnormal{wt}}: 0 \leq \langle \overline{\mu} + \omega , \alpha^\vee \rangle < p - 1~ \textnormal{for all}~\alpha \in \underline{\Phi}^+\right\}, \\
  \Lambda'^{\mu'}_{\textnormal{wt}} & := & \left\{\omega' \in \Lambda'_{\textnormal{wt}}: 0 \leq \langle \overline{\mu'} + \omega' , \alpha^\vee \rangle < p - 1~ \textnormal{for all}~\alpha \in \underline{\Phi}'^+\right\}.
\end{eqnarray*}
We also define affine and extended affine Weyl groups by 
$$\widetilde{\underline{W}}  :=  X^*(\un{T}) \rtimes \un{W}, \qquad  \underline{W}_\aff  :=  \Lambda_{\textnormal{rt}} \rtimes \un{W}, $$
$$\widetilde{\underline{W}}'  :=  X^*(\un{T}') \rtimes \un{W}', \qquad \underline{W}_\aff'  :=  \Lambda_{\textnormal{rt}}' \rtimes \un{W}',$$
where $\Lambda_{\textnormal{rt}}$ (resp., $\Lambda_{\textnormal{rt}}'$) denotes the root lattice of $\un{G}$ (resp., $\un{G}'$).  The inclusion $X^*(\un{T}) \longhookrightarrow \un{\tW}$ is denoted by $\lambda \longmapsto t_\lambda$ to emphasize the translation action of $X^*(\un{T})$ on $\Lambda_{\textnormal{wt}}$ (with analogous notation for $\un{G}')$.

\subsection{The $\nG\nL_2$ case}

Recall that that $\bullet_p$ denotes the $p$-dilated dot action: if $\mu' \in X^*(\underline{T}')$ and $w't_{\lambda'} \in \underline{\widetilde{W}}'$, then 
$$w't_{\lambda'}\bullet_p \mu' = w'(\mu' + p\lambda' + \eta') - \eta'$$
(and analogously for $\underline{G}$).  Given $\mu' \in X^*(\underline{T}')$, we define 
\begin{eqnarray*}
  \ft'_{\mu'}: X^*(\underline{T}') & \longrightarrow & X^*(\underline{T}')/(p - \pi')X^0(\underline{T}') \\
  \omega' & \longmapsto & \widetilde{w}' \bullet_p(\mu' + \omega') + (p - \pi')X^0(\underline{T}')
\end{eqnarray*}
where $\widetilde{w}'$ is the unique element of $\Omega' \cap t_{-\pi'^{-1}(\omega')}\underline{W}'_{\textnormal{aff}}$.  (Here $\Omega'$ denotes the stabilizer of the fundamental $p$-alcove under the $p$-dilated dot action of $\underline{\tW}'$.)  The map $\ft'_{\mu'}$ factors through the projection $X^*(\underline{T}') \longtwoheadrightarrow \Lambda'_{\textnormal{wt}}$, and we let $\ft_{\mu'}$ denote the restriction to $\Lambda'^{\mu'}_{\textnormal{wt}}$.  One then checks that $\ft_{\mu'}$ has image contained in the regular weights.  We thus obtain an injective map
$$\ft_{\mu'}: \Lambda'^{\mu'}_{\textnormal{wt}} \longhookrightarrow X_{\textnormal{reg}}(\underline{T}')/(p - \pi')X^0(\underline{T}'),$$
whose image consists of exactly those weights $\lambda'$ for which $\lambda'|_{\underline{Z}'} = \mu'|_{\underline{Z}'}~ \textnormal{mod}~(p - \pi')X^0(\underline{T}')$ (where $\un{Z}'$ denotes the center of $\un{G}'$).  In particular, by \cite[Lem. 9.2.4]{GHS}, the map $\omega' \longmapsto F'(\ft_{\mu'}(\omega'))$ gives a bijection between $\Lambda'^{\mu'}_{\textnormal{wt}}$ and the set of regular Serre weights with central character $\mu'|_{\underline{Z}'}$.

Given a subset $J \subset \cJ'$, we define $\eta'_J := \sum_{j' \in J}\eta'_{j'} \in X^*(\underline{T}')$, and denote by $\overline{\eta}'_J$ its image in $\Lambda'_{\textnormal{wt}}$.  Define $\Sigma'\subset \Lambda'_{\textnormal{wt}}$ to be the set $\{\overline{\eta}'_J\}_{J \subset \cJ'}$.

The following are Propositions 2.4.2 and 2.4.3 of \cite{BHHMS}.

\begin{propn}
\label{prop:SW:extgr:GL2}
  Suppose that $\rhobar': \Gamma_{K_2} \longrightarrow \nG\nL_2(\bbF)$ is a tame Galois representation such that $\rhobar'|_{I_{K_2}} \cong \overline{\tau}'(s', \mu' + \eta')$ for some $(s', \mu') \in \underline{W}' \times X^*(\underline{T}')$ with $\mu'$ lying $1$-deep in the fundamental $p$-alcove.  Then
  $$\nW^?(\rhobar') = \left\{F'(\ft_{\mu'}(s'\omega')): \omega' \in \Sigma'\right\}.$$
\end{propn}

\begin{propn}
\label{prop:JH:inertial}
  Suppose $\tau' := \tau'(s'w'^{-1}, \mu' - s'w'^{-1}(\nu'))$ for some $(s',\mu'), (w',\nu') \in \underline{W}'\times X^*(\underline{T}')$ such that $\mu' - s'w'^{-1}(\nu') - \eta'$ lies $1$-deep in the fundamental $p$-alcove.  If $\nu' \in \eta' + \Lambda'_{\textnormal{rt}}$, then 
  $$\JH\left(\overline{\sigma'(\tau')}\right) = \left\{F'\left(\ft_{\mu' - \eta'}(s'w'^{-1}(\omega' - \overline{\nu}'))\right): \omega' \in \Sigma'\right\}.$$
\end{propn}

\subsection{The unitary case}

We now give the analogs of the above results for the group $\widetilde{\nU}_{1,1}(k_K)$.  Recall that $\nU_{1,1}, \underline{\nU}_0,$ and $ \underline{\nU}$ denote the algebraic groups over $\cO_K$, $\bbZ_p$, and $\cO$, respectively, defined in Subsections \ref{unitarygps:ok} and \ref{unitarygps:qp}.

The surjection $\underline{T} \longtwoheadrightarrow \underline{T}_{\nU}$
gives a Galois-equivariant injection
$$X^*(\underline{T}_{\nU}) \longhookrightarrow X^*(\underline{T}).$$
We note that the above injection induces an isomorphism
\begin{equation}
  \label{induced}
  X^*(\underline{T}_\nU)/X^0(\underline{T}_\nU) \stackrel{\sim}{\longrightarrow} X^*(\underline{T})/X^0(\underline{T}) = \Lambda_{\textnormal{wt}}.  
\end{equation}

Recall that we identify $X^*(\underline{T}')$ with two copies of $X^*(\underline{T}_\nU)$.  It follows from the definition \eqref{def-of-BC} of $\BC$ that $\BC(X^0(\underline{T}_\nU)) \subset X^0(\underline{T}')$, which implies that $\BC$ descends to a well-defined map 
$$X^*(\underline{T}_\nU)/X^0(\underline{T}_\nU) \stackrel{\BC}{\longrightarrow} X^*(\underline{T}')/X^0(\underline{T}') = \Lambda'_{\textnormal{wt}}.$$
We denote by $\overline{\BC}: \Lambda_{\textnormal{wt}} \longrightarrow \Lambda'_{\textnormal{wt}}$ the composition of the inverse of \eqref{induced} with the map $\BC$ just above.

Let us make $\overline{\BC}$ a bit more explicit.  Suppose $\omega \in X^*(\underline{T})$, and define 
$$\omega_\nU := \sum_{j \in \cJ} \langle\omega, \alpha_j^\vee\rangle\rho_j \in X^*(\underline{T}_\nU)$$
(recall that $\rho_j$ denotes the character which is equal to $(1,0,0,0)$ in embedding $j$ and 0 elsewhere).  Thus, we have $\omega - \omega_\nU \in X^0(\underline{T})$ by construction, and any choice of $\omega'\in X^*(\underline{T}_\nU)$ for which $\omega - \omega' \in X^0(\underline{T})$ differs from $\omega_\nU$ by an element of $X^0(\underline{T}_\nU)$.  We thus obtain
$$\overline{\BC}(\omega) = \BC(\omega_\nU),$$
which gives a well-defined map $\Lambda_{\textnormal{wt}} \longrightarrow \Lambda'_{\textnormal{wt}}$.

As with the $\nG\nL_2$ case, we let $\bullet_p$ denote the $p$-dilated dot action of $\underline{\widetilde{W}}$ on $X^*(\underline{T})$: if $\mu \in X^*(\underline{T})$ and $wt_\lambda \in \underline{\widetilde{W}}$, we then have 
$$wt_\lambda \bullet_p \mu = w(\mu + p\lambda + \eta) - \eta.$$
We note that if $\mu,\lambda \in X^*(\underline{T}_\nU)$, then $wt_\lambda \bullet_p \mu \in X^*(\underline{T}_\nU)$ as well (this follows from the fact that, using the equivalence relation on $X^*(T_H)$, the character $w(\eta) - \eta$ is equivalent to one with last two entries equal to 0 in each embedding).  Given $\mu \in X^*(\underline{T}_\nU)$, we define 
\begin{eqnarray*}
  \ft'_\mu: X^*(\underline{T}) & \longrightarrow & X^*(\underline{T})/(p - \pi)X^0(\underline{T}) \\
  \omega & \longmapsto & \widetilde{w}\bullet_p (\mu + \omega) + (p - \pi)X^0(\underline{T})
\end{eqnarray*}
where $\widetilde{w}$ denotes the unique element of $\Omega \cap t_{-\pi^{-1}(\omega)}\underline{W}_{\textnormal{aff}}$.  As above, the map $\ft'_\mu$ factors through the projection $X^*(\underline{T}) \longtwoheadrightarrow \Lambda_{\textnormal{wt}}$, and we let $\ft_\mu$ denote the restiction to $\Lambda_{\textnormal{wt}}^\mu$.  This gives an injective map
$$\ft_\mu: \Lambda^\mu_{\textnormal{wt}} \longhookrightarrow X_{\textnormal{reg}}(\underline{T})/(p - \pi)X^0(\underline{T}),$$
whose image consists of exactly those regular weights $\lambda$ for which $\lambda|_{\underline{Z}} = \mu|_{\underline{Z}}~ \textnormal{mod}~ (p - \pi)X^0(\underline{T})$.

We note the following useful compatibility:

\begin{lemma}
\label{lem:comp:ext:gr}
  Suppose $\mu \in X^*(\underline{T}_\nU)$, $\omega \in X^*(\underline{T})$, and suppose that the image of $\omega$ in $\Lambda_{\textnormal{wt}}$ lies in $\Lambda_{\textnormal{wt}}^\mu$.  We then have
  $$\BC\left(\ft_\mu(\omega)\right) = \ft_{\BC(\mu)}(\overline{\BC}(\omega)).$$
\end{lemma}

Implicit in the statement of this lemma is that the base change map $\BC$ is well-defined on the relevant quotients.

\begin{proof}
  There are several places where we must check well-definedness.
\begin{enumerate}
  \item First, suppose that the image of $\omega$ in $\Lambda_{\textnormal{wt}}$ lies in $\Lambda_{\textnormal{wt}}^\mu$.  Since the map $-\fw$ fixes each coroot, we see that $\overline{\BC}(\omega)$ lies in $\Lambda'^{\BC(\mu)}_{\textnormal{wt}}$.  Thus, $\overline{\BC}$ restricts to a map
  $$\Lambda_{\textnormal{wt}}^\mu \stackrel{\overline{\BC}}{\longrightarrow} \Lambda'^{\BC(\mu)}_{\textnormal{wt}}.$$
  Consequently, the expression $\ft_{\BC(\mu)}(\overline{\BC}(\omega))$ is well-defined.
  \item Next, we claim that, assuming $\mu \in X^*(\underline{T}_\nU)$, the image of $\ft_\mu$ lies in $X_{\textnormal{reg}}(\underline{T}_\nU)/(p - \pi)X^0(\underline{T}_\nU)$.  It suffices to show that for $\omega \in X^*(\underline{T})$ with image contained in $\Lambda_{\textnormal{wt}}^\mu$, the character $\widetilde{w}\bullet_p (\mu + \omega)$ is equivalent modulo $(p - \pi)X^0(\underline{T})$ to one with its last two entries equal to 0.  Write $\widetilde{w} = wt_\lambda$, so that 
  \begin{eqnarray*}
    \widetilde{w}\bullet_p (\mu + \omega) & = & w(\mu + \omega + p\lambda + \eta) - \eta  \\
    & = & \left(w(\mu) + w(\eta) - \eta\right) + w(\omega + p\lambda).
  \end{eqnarray*}
  The first parenthesized expression lies in $X^*(\underline{T}_\nU)$, so it suffices to verify the same is true of $\omega + p\lambda$, up to an element of $(p - \pi)X^0(\underline{T})$.  This follows from a straightforward calculation using the construction of the element $\widetilde{w}$.

  \item Finally, if $\lambda \in X^*(\underline{T}_\nU)$, then a straightforward check verifies that $\BC((p - \pi)\lambda) = (p - \pi')\BC(\lambda)$.  Therefore, $\BC$ descends to a well-defined map 
  $$X_{\textnormal{reg}}(\underline{T}_\nU)/(p - \pi)X^0(\underline{T}_\nU) \stackrel{\BC}{\longrightarrow} X_{\textnormal{reg}}(\underline{T}')/(p - \pi')X^0(\underline{T}'),$$
  which implies that the expression $\BC(\ft_\mu(\omega))$ is also well-defined.
\end{enumerate}

Now that all relevant maps have been checked to be well-defined, verifying that $\BC(\ft_\mu(\omega)) = \ft_{\BC(\mu)}(\overline{\BC}(\omega))$ comes down to a straightforward calculation which we leave to the reader.
\end{proof}

Given a subset $J \subset \cJ$, we define $\eta_J := \sum_{j \in J}\eta_{j} \in X^*(\underline{T})$, and denote by $\overline{\eta}_J$ its image in $\Lambda_{\textnormal{wt}}$.  We will make use of the fact that $\overline{\rho}_J := \overline{\sum_{j \in J}\rho_j}$ is equal to $\overline{\eta}_J$ in $\Lambda_{\textnormal{wt}}$.  Define $\Sigma\subset \Lambda_{\textnormal{wt}}$ to be the set $\{\overline{\eta}_J = \overline{\rho}_J\}_{J \subset \cJ}$.

\begin{propn}
  \label{prop:SW:extgr:U2}
  Suppose $\rhobar: \Gamma_K \longrightarrow {}^C\nU_{1,1}(\bbF)$ is a tame $L$-parameter satisfying $\widehat{\imath} \circ \rhobar = \omega$.  Assume furthermore that $\rhobar$ is $1$-generic, so that we may write $\rhobar|_{I_K} \cong \overline{\tau}(s,\mu + \eta)$ for some $(s,\mu) \in \underline{W} \times X^*(\underline{T}_\nU)$ with $\mu$ lying $1$-deep in the fundamental $p$-alcove.  We then have
  $$\nW^?(\rhobar) = \left\{F(\ft_{\mu}(s\omega)): \omega \in \Sigma \right\}.$$
\end{propn}

\begin{proof}
  Given such a $\rhobar$, by \cite[Lem. 4.4]{koziolmorra} we have 
  $$\BC(\rhobar)|_{I_{K_2}} \cong \overline{\tau}'\left((s,s),\BC(\mu) + \eta'\right).$$
  Suppose $F$ is a Serre weight of $\un{G}_0(\bbF_p)$ on which $\imath(k_K^\times)$ acts trivially.  By \cite[Thm. 4.9]{koziolmorra} and Proposition \ref{prop:SW:extgr:GL2}, we have 
  \begin{eqnarray*}
    F \in \nW^?(\rhobar) & \Longleftrightarrow & \BC(F) \in \nW^?(\BC(\rhobar)) \\
    & \Longleftrightarrow & \BC(F) \cong F'\left(\ft_{\BC(\mu)}((s,s)\omega')\right) \quad \textnormal{for some}~\omega' \in \Sigma'.
  \end{eqnarray*}

  Let us write $\omega' = (\omega_1,\omega_2)$, with each of $\omega_1,\omega_2 \in X^*(\underline{T}_\nU)$ and image in $\Sigma$.  From the discussion in Subsection \ref{twist-by-epsilon}, we see that we must have
  $$F'\left(\ft_{\BC(\mu)}((s,s)(\omega_1,\omega_2))\right) \cong F'\left(\ft_{\BC(\mu)}((s,s)(-\underline{\fw}(\omega_2),-\underline{\fw}(\omega_1)))\right).$$
  Since the map $\Sigma' \ni \omega' \longmapsto F'(\ft_{\BC(\mu)}((s,s)\omega'))$ is injective, we conclude that $\omega_2 = -\underline{\fw}(\omega_1)$ as elements of $\Sigma$.  Consequently, we have
  \begin{eqnarray*}
    \BC(F) & \cong & F'\left(\ft_{\BC(\mu)}(s\omega_1,-s\underline{\fw}(\omega_1))\right) \\
    & \cong & F'\left(\ft_{\BC(\mu)}(\overline{\BC}(s\omega_1))\right) \\
    & \stackrel{\textnormal{Lem. \ref{lem:comp:ext:gr}}}{\cong} & F'\left(\BC(\ft_{\mu}(s\omega_1))\right) \\
    & \cong & \BC\left(F(\ft_{\mu}(s\omega_1))\right).
  \end{eqnarray*}
  Finally, we note that the map $F \longmapsto \BC(F)$ is injective on isomorphism classes of Serre weights, and therefore we get
  $$\BC(F) \cong \BC\left(F(\ft_{\mu}(s\omega_1))\right) \Longleftrightarrow F \cong F(\ft_{\mu}(s\omega_1)).$$  
\end{proof}

\begin{propn}
\label{prop:JH:fct:extgr:U2}
  Suppose $\tau' := \tau'((sw^{-1},sw^{-1}), \BC(\mu + \rho - sw^{-1}(\nu)))$ for some $(s,\mu), (w,\nu) \in \underline{W}\times X^*(\underline{T}_\nU)$ such that $\mu - sw^{-1}(\nu)$ lies $1$-deep in the fundamental $p$-alcove.  Then we have $\tau'^{\varphi^f} \cong \tau'^\vee$, and if $\nu \in \rho + \Lambda_{\textnormal{rt}}$, then 
  $$\JH\left(\overline{\sigma(\tau')}\right) = \left\{F\left(\ft_{\mu}(sw^{-1}(\omega - \overline{\nu}))\right): \omega \in \Sigma\right\}.$$
\end{propn}

\begin{proof}
  The fact that $\tau'^{\varphi^f} \cong \tau'^\vee$ follows from a calculation using the definition of $\tau'(s,\mu)$.  See also the proof of Lemma \ref{symiff}.

  We note that by equation \eqref{ILLC-explicit}, we have
  $$\sigma(\tau') \cong R_{sw^{-1}}(\mu + \rho - sw^{-1}(\nu)).$$
  Using equation \eqref{ILLC'-explicit} and \cite[Lem. 3.26]{koziolmorra}, we then obtain 
  $$F \in \JH\left(\overline{\sigma(\tau')}\right) \Longleftrightarrow \BC(F) \in \JH\left(\overline{\sigma'(\tau')}\right).$$

  Let us write $\nu = \rho + \lambda_{\textnormal{rt}}$ with $\lambda_{\textnormal{rt}} \in \Lambda_{\textnormal{rt}}$.  We then have
  \begin{eqnarray*}
    \BC(\mu + \rho - sw^{-1}(\nu)) & = & (\mu, -\underline{\fw}(\mu)) + (\rho,-\underline{\fw}(\rho)) - (sw^{-1}, sw^{-1})(\nu, -\underline{\fw}(\nu)) \\
    & = & (\mu, -\underline{\fw}(\mu)) + (\rho,\rho) - (0,\rho + \underline{\fw}(\rho))\\
    & & - (sw^{-1}, sw^{-1})(\rho + \lambda_{\textnormal{rt}}, \rho + \lambda_{\textnormal{rt}}) + (0, \rho + \underline{\fw}(\rho))\\
    & = & (\mu, -\underline{\fw}(\mu)) + \eta'  - (sw^{-1}, sw^{-1})(\rho + \lambda_{\textnormal{rt}}, \rho + \lambda_{\textnormal{rt}})
  \end{eqnarray*}
  Taking $\mu' = (\mu, -\underline{\fw}(\mu)) + \eta', s'w'^{-1} = (sw^{-1},sw^{-1})$ and $\nu' = (\rho + \lambda_{\textnormal{rt}}, \rho + \lambda_{\textnormal{rt}}) = \eta' + (\lambda_{\textnormal{rt}}, \lambda_{\textnormal{rt}})$, Proposition \ref{prop:JH:inertial} implies that the Serre weight $\BC(F)$ from the previous paragraph has the form 
  $$\BC(F) \cong  F'\left(\ft_{\mu' - \eta'}(s'w'^{-1}(\omega' - \overline{\nu}'))\right) \cong F'\left(\ft_{(\mu, -\underline{\fw}(\mu))}((sw^{-1},sw^{-1})(\omega' - \overline{\nu}'))\right),$$
  for some $\omega' \in \Sigma'$.  Further, the argument in the proof of Proposition \ref{prop:SW:extgr:U2} shows that $\omega'$ is of the form $(\omega,\omega)$ for some $\omega \in \Sigma$, and thus $\BC(F)$ has the form
  \begin{eqnarray*}
    \BC(F) & \cong & F'\left(\ft_{\BC(\mu)}(\overline{\BC}(sw^{-1}(\omega - \overline{\nu})))\right) \\ 
     & \stackrel{\textnormal{Lem. \ref{lem:comp:ext:gr}}}{\cong} & F'\left(\BC(\ft_{\mu}(sw^{-1}(\omega - \overline{\nu})))\right) \\
     & \cong & \BC\left(F(\ft_{\mu}(sw^{-1}(\omega - \overline{\nu})))\right)
  \end{eqnarray*}
  (we also use the fact that $\eta'$ and $\BC(\rho)$ have the same image in $\Lambda_{\textnormal{wt}}'$).  Thus, by injectivity of the base change map on Serre weights, we conclude that 
  $$F \cong F\left(\ft_{\mu}(sw^{-1}(\omega - \overline{\nu}))\right).$$
\end{proof}

We conclude with a result reciprocal to Proposition \ref{prop:JH:fct:extgr:U2}.

\begin{cor}
  \label{cor:types-cont-weight}
  Suppose $\sigma = F(\mu)$ for some $\mu \in X^*(\underline{T}_\nU)$ which lies $2$-deep in the fundamental $p$-alcove.  Then any tame inertial type $\tau'$ which satisfies $\sigma \in \JH(\overline{\sigma(\tau')})$ is of the form $\tau' \cong \tau'((s,s),\BC(\mu + \rho - s(\rho)))$ for some $s \in \underline{W}$.  Furthermore, these inertial types are pairwise non-isomorphic.
\end{cor}

\begin{proof}
  Let us write $\sigma = F(\mu) = F(\ft_{\mu}(0))$, and let $s \in \underline{W}$.  Using Proposition \ref{prop:JH:fct:extgr:U2} with $(w,\nu) = (1,\rho)$, we see that 
  $$F(\ft_\mu(0)) = F\big(\ft_\mu(s(\rhobar - \rhobar))\big) \in \JH\left(\overline{\sigma(\tau')}\right),$$
  where $\tau' := \tau'((s,s),~ \BC(\mu + \rho - s(\rho)))$.  Furthermore, an argument analogous to the second half of the proof of Lemma \ref{symiff} shows that these inertial types are pairwise non-isomorphic.

  We claim that these are the only representations of $\nU_{1,1}(k_K)$ over $E$ which contain $\sigma$ in their mod $p$ reduction.  By the adjunction properties of the $cde$ triangle (see \cite[\S 15.4]{serre:reps}), it suffices to show that the semisimple $E[\nU_{1,1}(k_K)]$-module $\tP_\sigma[1/p]$ has length $2^f$, where $\tP_\sigma$ denotes the projective cover of $\sigma$ in the category of $\cO[\nU_{1,1}(k_K)]$-modules.  The previous paragraph implies that $\tP_\sigma[1/p]$ has length at least $2^f$, so in order to verify the claim we assume by contradiction that the length is at least $2^f + 1$.  We now proceed by dimension-counting, as in the last paragraph of the proof of \cite[Lem. 4.1.2]{BHHMS}.  Precisely, the genericity assumptions on $\mu$ imply that $\dim_{\bbF}(\sigma) \geq 2$.  Therefore, we get $\dim_E(\tP_\sigma[1/p]) = \dim_\bbF(\tP_\sigma\otimes_{\cO}\bbF) = (2p)^f$, where the last equality follows from extending the central character to $\nG\nU_{1,1}(k_K)$, using Lemma \ref{resprops}\eqref{resprops-5}, and invoking the analogous result for $\nG\nL_2(k_K)$ (see \cite[\S 3]{BP}).  If $V$ is an irreducible constituent of $\tP_\sigma[1/p]$, then $[\overline{V}:\sigma] \neq 0$ (again using \cite[\S 15.4]{serre:reps}), which implies $\dim_E(V) \geq 2$.  By \cite[\S 6]{ennola}, this implies $\dim_E(V) \geq p^f - 1$.  Using that $p > 3$, we thus obtain
  $$\dim_{E}(\tP_\sigma[1/p]) = \sum_{V \in \JH(\tP_\sigma[1/p])}\dim_E(V) \geq (2^f + 1)(p^f - 1) > (2p)^f = \dim_E(\tP_\sigma[1/p]),$$
  giving the desired contradiction.  
\end{proof}

\begin{cor}
  \label{cor:types-cont-weight:2}
  Suppose $\sigma = F(\mu)$ for some $\mu \in X^*(\underline{T}_\nU)$ which lies $2$-deep in the fundamental $p$-alcove. Given $s_1,s_2 \in \underline{W}$, define the tame inertial types $\tau'_1 := \tau'((s_1,s_1), \BC(\mu + \rho - s_1(\rho))), \tau'_2 := \tau'((s_2,s_2), \BC(\mu + \rho - s_2(\rho)))$.  If $s_1 \neq s_2$, then $\JH(\overline{\sigma(\tau'_1)}) \neq \JH(\overline{\sigma(\tau'_2)})$.  
\end{cor}

\begin{proof}
  By Proposition \ref{prop:JH:fct:extgr:U2}, we have
  \begin{eqnarray*}
    \JH\left(\overline{\sigma(\tau_1')}\right) & = & \left\{F\big(\ft_\mu(s_1(\omega - \rhobar))\big): \omega \in \Sigma\right\}, \\    
    \JH\left(\overline{\sigma(\tau_2')}\right) & = & \left\{F\big(\ft_\mu(s_2(\omega - \rhobar))\big): \omega \in \Sigma\right\}.
  \end{eqnarray*}
  Fix an index $j \in \cJ$ for which $s_{1,j} \neq s_{2,j}$.  Without loss of generality, we may assume $s_{1,j} = 1, s_{2,j} = \fw$.  Consider the element $\rho_{\cJ\smallsetminus\{j\}} = \rho - \rho_j$, so that $\rhobar_{\cJ\smallsetminus\{j\}}\in \Sigma$.  The first equality above implies 
  $$F\big(\ft_\mu(-\rhobar_j)\big) = F\big(\ft_\mu(s_1(\rhobar_{\cJ\smallsetminus \{j\}} - \rhobar))\big) \in \JH\left(\overline{\sigma(\tau_1')}\right).$$
  We claim that $F(\ft_\mu(-\rhobar_j)) \not\in \JH(\overline{\sigma(\tau_2')})$.  Indeed, if we did have this containment, by the second displayed equality we would have
  $$F\big(\ft_\mu(-\rhobar_j)\big) = F\big(\ft_\mu(s_2(\omega - \rhobar))\big)$$
  for some $\omega \in \Sigma$.  By the injectivity of the map $\Sigma \ni \omega' \longmapsto F(\ft_\mu(\omega'))$, we obtain 
  $$-\rhobar_j = s_2(\omega - \rhobar),$$
  which implies 
  $$\rhobar_j + \rhobar = \omega \in \Sigma,$$
  a contradiction.  Therefore, $F(\ft_\mu(-\rhobar_j)) \not\in \JH(\overline{\sigma(\tau_2')})$, and the result follows.
\end{proof}

\begin{cor}
 \label{cor:types-cont-weight:3}
  Suppose $R_{s_1'}(\mu_1), R_{s_2'}(\mu_2)$ are two Deligne--Lusztig representations of $\nU_{1,1}(k_K)$, at least one of which is $3$-generic, and which satisfy
  $$\JH\left(\overline{R_{s_1'}(\mu_1)}\right) = \JH\left(\overline{R_{s_2'}(\mu_2)}\right).$$
  Then $R_{s_1'}(\mu_1) \cong R_{s_2'}(\mu_2)$.

  Consequently, if $\rhobar_1, \rhobar_2: \Gamma_K \longrightarrow {}^C\nU_{1,1}(\bbF)$ are two tamely ramified $L$-parameters, at least one of which is $3$-generic, and satisfying $\widehat{\imath} \circ \rhobar_1 = \widehat{\imath} \circ \rhobar_2 = \overline{\varepsilon}$ and $\nW^?(\rhobar_1) = \nW^?(\rhobar_2)$, then we have $\rhobar_1|_{I_K} \cong \rhobar_2|_{I_K}$.  
\end{cor}

\begin{proof}
  Define tame inertial types $\tau'_1$ and $\tau'_2$ by
  \begin{eqnarray*}
    \tau'_1 & := & \tau'\big((s_1', s_1'),~\BC(\mu_1)\big), \\    
    \tau'_2 & := & \tau'\big((s_2', s_2'),~\BC(\mu_2)\big).
  \end{eqnarray*}
  As in the second paragraph of the proof of Proposition \ref{prop:JH:fct:extgr:U2}, we have $\sigma(\tau'_i) \cong R_{s_i'}(\mu_i)$.  Now fix $F(\mu_0) \in \JH(\overline{\sigma(\tau'_1)}) = \JH(\overline{\sigma(\tau'_2)})$.  By Corollary \ref{cor:types-cont-weight}, we have
  \begin{eqnarray*}
    \tau'_1 & \cong & \tau'\big((s_1,s_1),~\BC(\mu_0 + \rho - s_1(\rho))\big), \\
    \tau'_2 & \cong & \tau'\big((s_2,s_2),~\BC(\mu_0 + \rho - s_2(\rho))\big),
  \end{eqnarray*}
  for some $s_1,s_2 \in \underline{W}$.  By Corollary \ref{cor:types-cont-weight:2} and the given assumptions, we obtain $s_1 = s_2$, which implies the result.  The second result follows from the first using \cite[Prop. 4.6]{koziolmorra}.
\end{proof}

\subsection{Extensions of Serre weights}

We now calculate $\textnormal{Ext}^1$ groups between Serre weights of $\nU_{1,1}(k_K)$.  We note that the base change technique from the previous section no longer applies, so we will make use of the formalism of Section \ref{sec:transfer}.

Recall the group $\nG\nU_{1,1}$ defined in Section \ref{unitarygps:ok}.  The group $\nG\nU_{1,1}$ contains $\nU_{1,1}$ and $\nG\nL_2$ as closed normal subgroups (the first of these realized as the kernel of the similitude character $\textnormal{sim}:\nG\nU_{1,1} \longrightarrow \bG_m$).  We let $T_{\nG\nU}$ denote the maximal torus of $\nG\nU_{1,1}$, which can be realized as a pushout $T_\nU \times^{T_\nS} T_\nG$ (we write $T_\nS$ for the diagonal maximal torus in $\nS\nL_2$, see also Subsection \ref{tori}).  
The various tori fit into the following diagram:
\begin{center}
  \begin{tikzcd}
    & 1 \ar[d] & 1 \ar[d] & & \\
    1 \ar[r] & T_\nS \ar[r] \ar[d] \ar[dr,phantom,"\square"]& T_\nG \ar[r, "\det"] \ar[d] & \bbG_m \ar[r] \ar[d, equals] & 1 \\
    1 \ar[r] & T_\nU \ar[r] \ar[d, "\det"'] & T_{\nG\nU} \ar[r, "\textnormal{sim}"] \ar[d] & \bbG_m \ar[r] & 1 \\
    & \nU_1 \ar[r, equals] \ar[d] & \nU_1 \ar[d] & & \\
    & 1 & 1 & & \\
  \end{tikzcd}
\end{center}

The above diagram induces a diagram of character groups:
\begin{center}
  \begin{tikzcd}[column sep = 3em,row sep = 3em]
    & 0  & 0  & & \\
    0 & X^*(T_\nS) \cong \bbZ \ar[l] \ar[u] & X^*(T_\nG) \cong \bbZ^2 \ar[l, "a - b \mapsfrom (a{,}b)"'] \ar[u] & X^*(\bbG_m) \cong \bbZ \ar[l,"(a{,}a)\mapsfrom a"']  & 0 \ar[l]\\
    0  & X^*(T_\nU) \cong \bbZ^2 \ar[l] \ar[u, "(a{,}b) \mapsto a - b"] & X^*(T_{\nG\nU}) \ar[l, "\textnormal{res}_{\nU}"'] \ar[u, "\textnormal{res}_{\nG}"] \ar[ul, phantom, "\square"] & X^*(\bbG_m) \cong \bbZ \ar[l, "\textnormal{sim}^*"'] \ar[u,equals] & 0 \ar[l]\\
    & X^*(\nU_1) \cong \bbZ \ar[u, "a \mapsto (a{,}a)"] & X^*(\nU_1) \cong \bbZ \ar[u] \ar[l,equals] & & \\
    & 0 \ar[u] & 0 \ar[u] & & \\
  \end{tikzcd}
\end{center}
The isomorphisms appearing are the canonical ones.

Thus, we see that we see that can describe $X^*(T_{\nG\nU})$ as a pullback, which we write as
$$X^*(T_{\nG\nU}) = \left\{(a,b,c,d) \in X^*(T_\nU)\oplus X^*(T_\nG) \cong \bbZ^4: a - b = c - d\right\}.$$
The maps $\textnormal{res}_\nU:X^*(T_{\nG\nU}) \longrightarrow X^*(T_\nU), \textnormal{res}_\nG : X^*(T_{\nG\nU}) \longrightarrow X^*(T_{\nG})$ are the projections onto the corresponding factors, and the map $\textnormal{sim}^*:X^*(\bbG_m) \cong \bbZ \longrightarrow X^*(T_{\nG\nU})$ is given by $a \longmapsto (0,0,a,a)$.

Analogously to the other groups already considered, we define 
$$(\underline{\nG\nU},~\underline{T}_{\nG\nU}) := \textnormal{Res}_{\cO_K/\bbZ_p}(\nG\nU_{1,1},~T_{\nG\nU})\times_{\bbZ_p}\cO.$$
We use similar notation $\underline{T}_\nS$ and $\underline{T}_\nG$ for the analogously defined tori for $\nS\nL_2$ and $\nG\nL_2$, respectively.  In particular, the character group $X^*(\underline{T}_{\nG\nU})$ is a pullback of $X^*(\underline{T}_\nU)$ and $X^*(\underline{T}_\nG)$ along the restriction maps to $X^*(\underline{T}_\nS)$.  This means that we have maps 
\begin{equation}
  \label{diag-for-ext-compn}
  \begin{tikzcd}[row sep = 0.3em]
    X^*(\underline{T}_\nU) & X^*(\underline{T}_{\nG\nU}) \ar[l, twoheadrightarrow, "\underline{\textnormal{res}}_{\nU}"'] \ar[r, twoheadrightarrow, "\underline{\textnormal{res}}_{\nG}"] & X^*(\underline{T}_\nG) \\
    (a_j, b_j)_{j \in \cJ} & (a_j, b_j, c_j, d_j)_{j \in \cJ} \ar[r,mapsto] \ar[l, mapsto] & (c_j, d_j)_{j \in \cJ}
  \end{tikzcd}
\end{equation}
We define a section $\underline{\textnormal{sec}}_\nU: X^*(\underline{T}_\nU) \longrightarrow X^*(\underline{T}_{\nG\nU})$ of $\underline{\textnormal{res}}_\nU$ by 
$$\underline{\textnormal{sec}}_\nU\left((a_j,b_j)_{j \in \cJ}\right) = (a_j,b_j, a_j, b_j)_{j \in \cJ}.$$
We define $\underline{\textnormal{sec}}_\nG$ analogously.  We caution that neither $\underline{\textnormal{sec}}_\nU$ nor $\underline{\textnormal{sec}}_\nG$ are Galois-equivariant.  In particular, the map $\underline{\textnormal{res}}_\nG\circ\underline{\sec}_{\nU}:X^*(\underline{T}_\nU) \longrightarrow X^*(\underline{T}_\nG)$ is the identity map, which does not respect the Galois action.

We now compare Serre weights and Deligne--Lusztig representations for each of the groups above.  We use a subscript to indicate which group is acting on the given representation (e.g., $F_{\nG\nU}(\lambda)$ is a Serre weight for $\nG\nU_{1,1}(k_K)$, $R_{\nU,w}(\mu)$ is a Deligne--Lusztig representation of $\nU_{1,1}(k_K)$, etc.).  We record the following properties, which follow readily from the definitions (see also \cite[Thm. 4.2]{delignelusztig} for the fifth item):
\begin{lemma}\label{extreslemma}
  \begin{enumerate}
    \item If $\lambda \in X_1(\underline{T}_{\nG\nU})$ is a $p$-restricted weight, then so are $\underline{\textnormal{res}}_\nU(\lambda)$ and $\underline{\textnormal{res}}_\nG(\lambda)$.
    \item If $\mu \in X_1(\underline{T}_\nU)$ is a $p$-restricted weight, then any $\lambda \in \underline{\textnormal{res}}_{\nU}^{-1}(\mu)$ is $p$-restricted as well.  \textit{Mutatis mutandis} for $\nG\nL_2$, with subscripts $\nG$.
    \item\label{extreslemma:item3} If $\lambda \in X_1(\underline{T}_{\nG\nU})$, then we have
    $$F_{\nG\nU}(\lambda)|_{\nU_{1,1}(k_K)} \cong F_\nU(\underline{\textnormal{res}}_\nU(\lambda)),$$
    $$F_{\nG\nU}(\lambda)|_{\nG\nL_2(k_K)} \cong F_\nG(\underline{\textnormal{res}}_\nG(\lambda)).$$
    \item\label{extreslemma:item4} If $\mu, \mu' \in X_1(\underline{T}_\nU)$ are two $p$-restricted weights such that $F_\nU(\mu)$ and $F_\nU(\mu')$ have the same central character, then the same is true of $F_{\nG\nU}(\underline{\textnormal{sec}}_\nU(\mu))$ and $F_{\nG\nU}(\underline{\textnormal{sec}}_\nU(\mu'))$.  \textit{Mutatis mutandis} for $\nG\nL_2$, with subscripts $\nG$.
    \item\label{extreslemma:item5} If $\mu \in X^*(\underline{T}_{\nG\nU})$ is such that $\mu - \eta_{\nG\nU}$ is $0$-deep, then
    $$R_{\nG\nU, w}(\mu)|_{\nU_{1,1}(k_K)} \cong R_{\nU,w}(\underline{\textnormal{res}}_{\nU}(\mu)),$$
    $$R_{\nG\nU, w}(\mu)|_{\nG\nL_2(k_K)} \cong R_{\nG,w}(\underline{\textnormal{res}}_{\nG}(\mu)),$$
    where we have used the canonical identification between the Weyl groups of $\nG\nU_{1,1}$, $\nG\nL_2$ and $\nU_{1,1}$.  
    \item\label{extreslemma:item6} If $\mu \in X^*(\underline{T}_\nU)$ is such that $\mu - \rho$ is $0$-deep, then 
    $$\JH\left(\overline{R_{\nG\nU,s}(\underline{\textnormal{sec}}_{\nU}(\mu))}\right) = \left\{F_{\nG\nU}(\underline{\textnormal{sec}}_\nU(\lambda)):~ F_\nU(\lambda) \in \JH\left(\overline{R_{\nU,s}(\mu)}\right)\right\}.$$
    \textit{Mutatis mutandis} for $\nG\nL_2$, with subscripts $\nG$.
  \end{enumerate}
\end{lemma}

  We can now deduce the main result on extensions of Serre weights.  Recall that two elements $\omega,\omega' \in X^*(\underline{T}_\nU)$ with images contained in $\Lambda_{\textnormal{wt}}^\mu$ are said to be adjacent if $\omega - \omega' \equiv \rho_j ~\textnormal{mod}~X^0(\underline{T}_\nU)$ for some $j \in \cJ$.  

  \begin{propn}
    \label{serre-wt-extns}
    Suppose $\omega,\omega'$ are two elements of $X^*(\underline{T}_\nU)$ whose images are contained in $\Lambda_{\textnormal{wt}}^\mu$. Then
    $$\dim_{\bbF}\left(\Ext_{\nU_{1,1}(k_K)}^1\left(F_\nU(\ft_\mu(\omega)),~F_\nU(\ft_\mu(\omega'))\right)\right) = \begin{cases}
      1 & \textnormal{if $\omega, \omega'$ are adjacent,}\\
      0 & \textnormal{otherwise.}
    \end{cases}$$
  \end{propn}

  \begin{proof}
   We have the following sequence of isomorphisms:
    \begin{flushleft}
      $\displaystyle{\Ext_{\nU_{1,1}(k_K)}^1\left(F_\nU(\ft_\mu(\omega)),~F_\nU(\ft_\mu(\omega'))\right)}$
    \end{flushleft}
    \begin{eqnarray*}
       & \stackrel{\textnormal{Lem. \ref{extreslemma}\eqref{extreslemma:item3}}}{\cong} & \Ext_{\nU_{1,1}(k_K)}^1\left(F_{\nG\nU}(\underline{\textnormal{sec}}_\nU(\ft_\mu(\omega)))|_{\nU_{1,1}(k_K)},~F_{\nG\nU}(\underline{\textnormal{sec}}_\nU(\ft_\mu(\omega')))|_{\nU_{1,1}(k_K)}\right) \\
      & \stackrel{\textnormal{Lems. \ref{extreslemma}\eqref{extreslemma:item4}, \ref{resext}}}{\cong} & \Ext_{\nG\nU_{1,1}(k_K)}^1\left(F_{\nG\nU}(\underline{\textnormal{sec}}_\nU(\ft_\mu(\omega))),~F_{\nG\nU}(\underline{\textnormal{sec}}_\nU(\ft_\mu(\omega')))\right) \\
      & \stackrel{\textnormal{Lems. \ref{extreslemma}\eqref{extreslemma:item4}, \ref{resext}}}{\cong} & \Ext_{\nG\nL_{2}(k_K)}^1\left(F_{\nG\nU}(\underline{\textnormal{sec}}_\nU(\ft_\mu(\omega)))|_{\nG\nL_{2}(k_K)},~F_{\nG\nU}(\underline{\textnormal{sec}}_\nU(\ft_\mu(\omega')))|_{\nG\nL_{2}(k_K)}\right) \\
      & \stackrel{\textnormal{Lem. \ref{extreslemma}\eqref{extreslemma:item3}}}{\cong} & \Ext_{\nG\nL_{2}(k_K)}^1\left(F_{\nG}(\underline{\textnormal{res}}_\nG\circ\underline{\textnormal{sec}}_\nU(\ft_\mu(\omega))),~F_{\nG}(\underline{\textnormal{res}}_\nG\circ\underline{\textnormal{sec}}_\nU(\ft_\mu(\omega')))\right) \\
      & \cong & \Ext_{\nG\nL_{2}(k_K)}^1\left(F_{\nG}(\ft_\mu(\omega)),~F_{\nG}(\ft_\mu(\omega'))\right).
    \end{eqnarray*}
    The result now follows form \cite[Lem. 2.4.6]{BHHMS}.
  \end{proof}

\section{Kisin modules}
\label{sec:Kisinmods}

\subsection{Polarized Kisin modules and polarized \'etale $\varphi$-modules}

Following \cite[\S 5A, 5B]{koziolmorra}, we recall here the theory of polarized Kisin modules, which will be relevant for the computations of deformation rings.

\subsubsection{}
Given a complete local Noetherian $\cO$-algebra $R$ with residue field $\bbF$, we let 
$$\fS_{R} := (\cO_{K_2}\otimes_{\bbZ_p}R)[\![u']\!],$$
equipped with Frobenius map $\varphi$ which acts as the arithmetic Frobenius on $\cO_{K_2}$, acts trivially on $R$, and sends $u'$ to $(u')^p$.  (Note that the Frobenius map $\varphi$ and the variable $u'$ are denoted by $\overline{\varphi}$ and $u$ in \cite{koziolmorra}, respectively.)

We fix a choice $\sqrt[p^{2f} - 1]{-p}$ of a $(p^{2f} - 1)^{\textnormal{st}}$ root of $-p$ and set 
$$L' := K_2(\sqrt[p^{2f} - 1]{-p}).$$  
The rings $\fS_R$ above are endowed with a $\cO_{K_2}\otimes_{\bbZ_p}R$-linear action of $\Gal(L'/K_2)$ by continuous ring automorphisms: if $g\in\Gal(L'/K_2)$, then the associated automorphism is defined by the condition 
$$u' \longmapsto \Big(\frac{g\big(\sqrt[p^{2f} - 1]{-p}\big)}{\sqrt[p^{2f} - 1]{-p}}\otimes 1\Big) u'$$
(note that $\frac{g(\sqrt[p^{2f} - 1]{-p})}{\sqrt[p^{2f} - 1]{-p}}\in \cO_{K_2}^\times$).

Suppose we are given a conjugate self dual tame inertial type $\tau'=\tau'((s,s),\BC(\mu))$ as in \ref{sec:ILL}. Analogously to \cite[Defs. 5.1, 5.2, 5.3]{koziolmorra}, we define $Y^{[0,3],\tau'}(R)$, the groupoid of Kisin modules over $R$ of rank 2 with height in $[0,3]$ and descent datum of type $\tau'$, except that the condition $E(u)\fM\subset \phi_{\fM}(\overline{\varphi}^*\fM)\subset \fM$ of \textit{op. cit.} is replaced by $E(u')^3\fM\subset \phi_{\fM}(\varphi^*\fM)\subset \fM$.  (Here, $E(u') := (u')^{p^{2f} - 1} + p$ denotes the Eisenstein polynomial.)  Furthermore, we define $Y^{\leq(3,0),\tau'}(R)$ to be the full subgroupoid of $Y^{[0,3],\tau'}(R')$ consisting of Kisin modules satisfying the condition
$$(E(u'))^3\det \fM = \phi_{\fM}(\varphi^*(\det\fM)).$$

\subsubsection{}

Given $\fM \in Y^{\leq(3,0),\tau'}(R)$, we define its dual Kisin module $\fM^\vee$ by 
$$\fM^\vee := \Hom_{\fS_R}(\fM, \fS_R),$$
with Frobenius given by
$$1\otimes f \longmapsto \varphi\circ (1 \otimes f) \circ \phi_{\fM}^{-1} \circ E(u')^3,$$
where $1\otimes f \in \varphi^*\Hom_{\fS_R}(\fM,\fS_R) \cong \Hom_{\fS_R}(\varphi^*\fM, \varphi^*\fS_R)$ and descent datum as in \cite[Def. 5.12]{koziolmorra}.

We define the notion of a \emph{polarization $\iota$} on $\fM\in Y^{\leq(3,0),\tau'}(R)$ exactly as in \cite[Def. 5.15]{koziolmorra}, and hence the groupoid $Y^{\leq(3,0),\tau'}_{\textnormal{pol}}(R)$ consisting of polarized Kisin modules, i.e.,~pairs $(\fM,\iota)$ where $\fM\in Y^{\leq(3,0),\tau'}(R)$ and $\iota: (\sigma^f)^*\fM \stackrel{\sim}{\longrightarrow} \fM^\vee$ is a polarization on $\fM$.  (For the definition of the Frobenius pullback $(\sigma^f)^*$, see \cite[\S 5B3]{koziolmorra}.)

\subsubsection{}
We now discuss analogous notions for \'etale $\varphi$-modules.  We let $\cO_{\cE,L'} := \left(\cO_{K_2}[\![u']\!][1/u']\right)^{\wedge_p}$ and $\cO_{\cE,K_2} := \left(\cO_{K_2}[\![v]\!][1/v]\right)^{\wedge_p}$, endowed with a Frobenius map $\varphi$ acting as the arithmetic Frobenius on $\cO_{K_2}$ and sending $u'$ (resp.,~$v$) to $(u')^p$ (resp.,~$v^p$).
Given a $p$-adically complete, Noetherian $\cO$-algebra $R$, we define as in \cite[\S 5.4.1]{LLLM3} the groupoid of rank $2$ \'etale $\varphi$-modules over $\cO_{\cE,L'}\widehat{\otimes}_{\bbZ_p}R$ with descent data, and the groupoid of rank $2$ \'etale $\varphi$-modules over $\cO_{\cE,K_2}\widehat{\otimes}_{\bbZ_p}R$.

Given $\bullet\in\{L',K_2\}$ and an \'etale $\varphi$-module $\cM$ over $\cO_{\cE,\bullet}\widehat{\otimes}_{\bbZ_p}R$, we define its dual as
$$\cM^\vee :=\Hom_{\cO_{\cE,\bullet}\widehat{\otimes}_{\bbZ_p}R}(\cM,\cO_{\cE,\bullet}\widehat{\otimes}_{\bbZ_p}R).$$
We endow it with the semilinear endomorphism induced by $1\otimes f\mapsto \varphi\circ(1\otimes f)\circ \phi_{\cM}^{-1}$, and, if $\bullet = L'$, with the descent data defined as in \cite[Def. 5.12]{koziolmorra}.  This makes $\cM^\vee$ into a rank $2$ \'etale $\varphi$-module (with descent data if $\bullet = L'$) over  $\cO_{\cE,\bullet}\widehat{\otimes}_{\bbZ_p}R$.

By fixing a compatible system $\{\sqrt[p^n]{-p}\}_{n\geq 1}$ of $(p^n)^{\textnormal{th}}$ roots of $-p$ and setting
$$K_{\infty} := \bigcup_{n \geq 1} K(\sqrt[p^n]{-p}),\quad K_{2,\infty} := \bigcup_{n \geq 1} K_2(\sqrt[p^n]{-p}),$$
by \cite[Thm. 2.1.27]{dee} we have an exact equivalence of categories between rank $2$ \'etale $\varphi$-modules over $\cO_{\cE,L'}\widehat{\otimes}_{\bbZ_p}R$ with descent data (resp.,~rank $2$ \'etale $\varphi$-modules over $\cO_{\cE,K_2}\widehat{\otimes}_{\bbZ_p}R$) and continuous representations of $\Gamma_{K_{2,\infty}}$ over rank $2$ projective $R$-modules.

\subsubsection{}
A polarization for a rank $2$ \'etale $\varphi$-module over $\cO_{\cE,K_2}\widehat{\otimes}_{\bbZ_p}R$ is defined exactly as for a Kisin module.
If $(\fM,\iota)\in Y^{\leq(3,0),\tau'}_{\textnormal{pol}}(R)$, then $\cM:= (\fM\otimes_{\fS_{\bbZ_p}}\cO_{\cE,L'})^{\Gal(L'/K_2) = 1}$ is an \'etale $\varphi$-module over $\cO_{\cE,K_2}\widehat{\otimes}_{\bbZ_p}R$ which is naturally endowed with a polarization.  
Consequently the $\Gamma_{K_{2,\infty}}$-representation $T^*_{\textnormal{dd}}(\fM) = \bbV_{K_2}^*(\cM)$, where $\bbV_{K_2}^*$ denotes the usual anti-equivalence of Fontaine between \'etale $\varphi$-modules and continuous representations of $K_{2,\infty}$, also acquires a polarization. (Here we are using the fact that $\fM^\vee\in Y^{\leq(0,3),\tau^{\prime\,\vee}}(R)$ and $T^*_{\textnormal{dd}}(\fM^\vee)\cong T^*_{\textnormal{dd}}(\fM)^\vee\otimes\varepsilon^3$ as $\Gamma_{K_{2,\infty}}$-representations, along with the results of \cite{dee} to remove the condition on Artinian coefficients.)
Therefore $T^*_{\textnormal{dd}}(\fM)$ descends to a $^{C}\nU_{1,1}$-valued representation of $\Gamma_{K_{\infty}}$.
Conversely, if $\cM$ is an \'etale $\varphi$-module over $\cO_{\cE,K_2}\widehat{\otimes}_{\bbZ_p}R$ and $\bbV_{K_2}^*(\cM)$ descends to a $^{C}\nU_{1,1}$-valued representation of $\Gamma_{K_{\infty}}$, then $\cM$ is naturally equipped with a polarization.

\subsection{Finite height condition}
\label{subsec:FH}

\subsubsection{}
Given $\fM\in Y^{\leq(3,0),\tau'}(R)$ we have the notion of \emph{eigenbasis} for $\fM$ (\cite[Def. 5.8]{koziolmorra}).  If $\beta$ is an eigenbasis of $\fM \in Y^{[0,3],\tau'}(R)$ and $\iota$ denotes a polarization on $\fM$, then we say $\beta$ is an eigenbasis of $(\fM, \iota)$ if $\iota((\sigma^f)^*\beta) = (\underline{1}, -\underline{1})\beta^\vee$.  (This is the same as \cite[Def. 5.16]{koziolmorra}, with ``gauge basis'' replaced by ``eigenbasis.'')

Given a Kisin module $\fM$ we let $\fM^{(j')}$ denote the $R[\![u']\!]$-submodule formed by the elements $m\in\fM$ such that $(x\otimes1)m=(1\otimes \sigma_0'\circ\varphi^{j'}(x))m$ for any $x\in \cO_{K_2}$, and write $\phi_{\fM}^{(j')}$ to denote the restriction of $\phi_{\fM}$ to $\varphi^*(\fM^{(j')})$.
For $j' \in \cJ'$, we define matrices $C_{\fM, \beta}^{(j')}$ and $A_{\fM, \beta}^{(j')}$ as follows.  We set $C^{(j')}_{\fM,\beta} := \textnormal{Mat}_\beta(\phi_{\fM}^{(j')})$, i.e.,
$$\phi_{\fM}^{(j')}(\varphi^*\beta^{(j')}) = \beta^{(j' + 1)}\cdot C_{\fM,\beta}^{(j')},$$
and define
$$A_{\fM, \beta}^{(j')} := \textnormal{Ad}\left((\dot{s}'_{\textnormal{or}, j' + 1})^{-1} (u')^{-\textbf{a}_{((s,s),\BC(\mu))}^{(j' + 1)}}\right)(C_{\fM, \beta}^{(j')})$$
(where $s'_{\textnormal{or},j' + 1}$ and $\mathbf{a}^{(j' + 1)}_{((s,s),\BC(\mu))}$ are defined in \cite[\S 2.3]{BHHMS}, and $\dot{s}'_{\textnormal{or},j' + 1}$ is the permuation matrix associated to $s'_{\textnormal{or},j' + 1}$).  Given $\fM\in Y^{\leq(3,0),\tau'}(\F)$, we define the \emph{shape of $\fM$ with respect to $\tau'$} as in \cite[Def. 5.9(i)]{koziolmorra}.

Suppose that $\beta$ is a gauge basis for $(\fM,\iota)\in Y^{\leq (3,0), \tau'}_{\pol}(R)$.  The finite height conditions \cite[Prop. 4.18]{LLLM} give
\begin{equation}
\label{finheight}
\det(A^{(j')}_{\fM, \beta}) = \det(C^{(j')}_{\fM, \beta}) = (a^{(j')})^{-1}E(u')^3
\end{equation}
for some $a^{(j')} \in R^\times$.  Further, by \cite[Lem. 5.18]{koziolmorra} we have
\begin{eqnarray}
A^{(j' - f)}_{\fM, \beta} & = & \begin{cases} E(u')^3 \textnormal{Ad}(\dot{\fw})\big((A^{(j')}_{\fM, \beta})^{-\top}\big) & \textnormal{if}~ j' \neq f - 1, 2f - 1, \\ -E(u')^3 \textnormal{Ad}(\dot{\fw})\big((A^{(j')}_{\fM, \beta})^{-\top}\big) & \textnormal{if}~ j' = f - 1, 2f - 1,\end{cases} \notag \\
 & = &  \begin{cases} a^{(j')} \textnormal{Ad}\big(\textnormal{diag}(-1,1)\big)(A^{(j')}_{\fM,\beta}) & \textnormal{if}~ j' \neq f - 1, 2f - 1, \\ -a^{(j')} \textnormal{Ad}\big(\textnormal{diag}(-1,1)\big)(A^{(j')}_{\fM,\beta}) & \textnormal{if}~ j' = f - 1, 2f - 1.\end{cases} \label{polarizationeqs}
\end{eqnarray}
and furthermore any deformation of a polarized Kisin module is determined by the deformations of the matrices $(A^{(j')}_{\fM,\beta})_{0 \leq j' \leq f - 1}$.  
(Here $\dot{\fw}$ denotes the permutation matrix associated to $\fw$.)

\subsection{Monodromy condition}
We now study the monodromy condition on polarized Kisin modules.  We follow closely the notations and conventions of \cite[\S 3.1.2]{BHHMS}, to which we refer the reader for any undefined notation and further details.

\subsubsection{}
In what follows $R$ is a $p$-adically complete, flat $\cO$-algebra that is topologically of finite type. 
As in \cite[\S 3.1.2]{BHHMS}, we define the ring $\cO^{\textnormal{rig}}_R\subset R[1/p][\![u']\!]$, endowed with a $\Gal(L'/K_2)$-action extending the $\Gal(L'/K_2)$-action on $R[\![u']\!]$, and, given $(\fM,\iota)\in Y^{[0,3],\tau'}_{\textnormal{pol}}(R)$, we let $\fM^{\textnormal{rig}} := \fM \otimes_{R[\![u']\!]} \cO^{\textnormal{rig}}_R = \bigoplus_{j'\in\cJ'}\fM^{\textnormal{rig},(j')}$.  By \cite[Prop. 3.1.7]{BHHMS}, the module $\fM^{\textnormal{rig}}[1/\lambda]$ may be equipped with a unique derivation $N_{\fM^{\textnormal{rig}}}$ over $N_\nabla$ satisfying $N_{\fM^{\textnormal{rig}}}\phi_{\fM^{\textnormal{rig}}} = E(u')\phi_{\fM^{\textnormal{rig}}}N_{\fM^{\textnormal{rig}}}$ and $N_{\fM^{\textnormal{rig}}}~(\textnormal{mod}~u') = 0$. (Here $\lambda \in \cO_{\cO}^{\textnormal{rig}} \subset \cO_{R}^{\textnormal{rig}}$ is a certain transcendental element and $N_\nabla$ is the derivation relative to $\lambda$, both recalled in \cite[\S 3.1.2]{BHHMS}.)  For $j' \in \cJ'$ we let $N_{\fM^{\textnormal{rig}}}^{(j')}$ denote the restriction of $N_{\fM^{\textnormal{rig}}}$ to $\fM^{\textnormal{rig}, (j')}$ and given an eigenbasis $\beta$ for $(\fM,\iota)$, we define $N_{\fM^{\rig},\beta}^{(j')} := \textnormal{Mat}_\beta(N_{\fM^{\textnormal{rig}}}^{(j')})$, i.e., 
$$N_{\fM^{\rig}}^{(j')}(\beta^{(j')}) = \beta^{(j')}\cdot N_{\fM^{\rig},\beta}^{(j')}.$$

\subsubsection{}
We say that $(\fM,\iota)\in Y^{[0,3],\tau'}_{\textnormal{pol}}(R)$ satisfies the \emph{monodromy condition} if $\lambda^2N_{\fM^{\rig},\beta}^{(j')}$ vanishes to order $2$ at $u'=(-p)^{1/(p^{2f} - 1)}$ for all $j'\in\cJ'$ (\cite[Def. 3.1.8]{BHHMS}).  It is independent of the choice of the eigenbasis, and it is equivalent to the condition that $T_{\textnormal{dd}}^*(\fM)[1/p]$ is the restriction to $\Gamma_{K_{2,\infty}}$ of a potentially crystalline representation of $\Gamma_{K_2}$ over $R[1/p]$, of inertial type $\tau'$ and Hodge--Tate weights in $[0,3]$.  As $K_\infty \cdot K_2 = K$, $K_\infty \cap K_2 = K_{2,\infty}$, and $T_{\textnormal{dd}}^*(\fM)$ descends to a representation of $\Gamma_{K_\infty}$, we conclude that the monodromy condition is equivalent to the condition that $T_{\textnormal{dd}}^*(\fM)[1/p]$ is the restriction to $\Gamma_{K_{\infty}}$ of a potentially crystalline representation of $\Gamma_{K}$ over $R[1/p]$.

\begin{lemma}
Let $(\fM,\iota) \in Y^{\leq (3,0),\tau'}_{\pol}(R)$ and $\beta$ be a gauge basis for it.
For any $j' \in \cJ'$ we have
$$N_{\fM^{\rig},\beta}^{(j' + f)} = \textnormal{Ad}(\fw')(N_{\fM^{\rig},\beta}^{(j')})$$
where we set $\fw' = \sm{0}{1}{-1}{0}$.
\end{lemma}

\begin{proof}
Using the polarization $\iota$ we have 
$$C^{(j' + f)}_{\fM, \beta} = E(u')^3(C^{(j')}_{\fM, \beta})^{-\top}$$
(cf.~\cite[Prop. 5.13]{koziolmorra}).
Moreover the equation \eqref{finheight} above gives
\begin{equation}
\label{Cmatrix-cocycle}
C^{(j' + f)}_{\fM, \beta} = a^{(j')}\textnormal{Ad}(\fw')(C^{(j')}_{\fM, \beta}),
\end{equation}
We now recall from \cite[Thm. 5.6]{LLLM}, \cite[Prop. 3.4.12]{LLL} that we can construct $N_{\fM^{\rig},\beta}^{(j')}$ as 
$$N_{\fM^{\rig},\beta}^{(j')} = N_1^{(j')} + \sum_{i = 1}^\infty \left(\prod_{k = 0}^{i - 1}\varphi^k \Big(C^{(j' - k - 1)}_{\fM,\beta}\Big) \right)\varphi^i(N_1^{(j' - i)})\left(\prod_{k = i - 1}^0 \varphi^k \Big(E(u')(C^{(j' - k - 1)}_{\fM,\beta})^{-1}\Big) \right),$$
where $N_1^{(j')} := \lambda u' \frac{d}{du'}(C^{(j' - 1)}_{\fM,\beta})(C^{(j' - 1)}_{\fM,\beta})^{-1}$. The claim follows from equation \eqref{Cmatrix-cocycle}, along with the relations $\varphi(\fw') = \fw'$ and $\frac{d}{du'}(a^{(j')}) = 0$.
\end{proof}

\subsubsection{}

Given $(\fM,\iota) \in Y^{\leq (3,0),\tau'}_{\pol}(R)$ and $\beta$ a gauge basis for it, we define $P_N(A^{(j'-1)}_{\fM,\beta})$ exactly as in \cite[Eq. (16)]{BHHMS} for any $j'\in\cJ'$.  The following is \cite[Prop. 3.1.9]{BHHMS}.

\begin{propn}
\label{prop:mon:cond}
Let $(\fM,\iota) \in Y^{\leq (3,0),\tau'}_{\pol}(R)$ and let $\beta$ be a gauge basis for it.
Assume further that $\tau'$ is $N$-generic with $N\geq 4$. Then the monodromy condition is equivalent to the conditions
\begin{eqnarray*}
  \left(P_N(A^{(j'-1)}_{\fM,\beta})\right)\big|_{v=-p}  +  O(p^{N-2}) & =& 0\\
  \Big(\frac{d}{dv}\Big)\big|_{v=-p} \left(P_N(A^{(j'-1)}_{\fM,\beta})\right) +  O(p^{N - 3}) & = & 0
\end{eqnarray*}
for all $j' = 0,\dots,f-1$, where the terms $O(p^{N-2})$ and $O(p^{N-3})$ denote specific but inexplicit elements of $p^{N-2-t}\textnormal{Mat}_2(R)$.
\end{propn}

\section{Galois deformations}
\label{sec:galdefs}

\subsection{Preliminaries}
\label{galdefs:prelim}

Fix a tamely ramified $L$-parameter $\rhobar:\Gamma_K \longrightarrow {}^C\nU_{1,1}(\bbF)$ which satisfies $\widehat{\imath}\circ \rhobar = \omega$.  Recall from \cite[Lem. 2.1.1]{CHT} (see also \cite[\S 5C2]{koziolmorra}) that this data is equivalent to the data of the representation $\BC(\rhobar): \Gamma_{K_2} \longrightarrow \nG\nL_2(\bbF)$ and a compatible polarization $\overline{\alpha}: \BC(\rhobar)^{\varphi^f} \stackrel{\sim}{\longrightarrow} \BC(\rhobar)^\vee \otimes \omega$.

We let $(s,\mu) \in \un{W} \times X^*(\un{T}_\nU)$ be such that 
$$\rhobar|_{I_K} \cong \taubar(s,\mu + \eta).$$  
We may choose this data so that 
$$s \in \left\{(1,1, \ldots, 1),~ (\fw,1,\ldots, 1)\right\}$$
(see e.g.~\cite[\S 4]{BHHMS}).
We assume moreover that $\mu$ lies $N$-deep in the fundamental $p$-alcove, for some $N \geq 12$.  By equation \eqref{BC-of-inertial-type}, we have
$$\BC(\rhobar)|_{I_{K_2}} \cong \taubar'\big((s,s),~\BC(\mu) + \eta'\big).$$

For $\lambda' \in X^*(\un{T}')$, we define the $\lambda'$-admissible set by
$$\textnormal{Adm}'^\vee(t_{\lambda'}) := \left\{\tw' \in \un{\tW}'^\vee: \tw' \leq t_{w'(\lambda')} ~\textnormal{for some}~w'\in \un{W}'\right\},$$
relative to the Bruhat order on $\un{\tW}'^\vee$.  (Here $\un{\tW}'^\vee$ denotes the group $\un{\tW}'$ equipped with the Bruhat order relative to the antidominant fundamental $p$-alcove; see \cite[\S 2.1]{BHHMS}.)  In particular, we have
$$\textnormal{Adm}'^\vee(t_{\eta'}) = \{t_{(1,0)},~ \fw t_{(1,0)},~ t_{(0,1)}\}^{2f},$$
and we furthermore define
$$\textnormal{Adm}'^\vee(t_{\eta'})^{\sym} := \left\{\tw' \in \textnormal{Adm}'^\vee(t_{\eta'}): \tw'_{j' + f} = \tw'_{j'}~\textnormal{for all $j' \in \cJ'$}\right\}.$$

Now choose $\tw' \in \textnormal{Adm}'^\vee(t_{\eta'})$ and write $\tw'^* = t_{\nu'} w' \in X^*(\un{T}') \rtimes \un{W}'$ (see \cite[\S 2.1]{BHHMS} for the definition of the anti-automorphism $\tw' \longmapsto \tw'^*$).  Explicitly, if
$$\tw'_{2f - 1 - j'} = t_{(1,0)},\quad \textnormal{resp.},\quad \fw t_{(1,0)},\quad \textnormal{resp.},\quad t_{(0,1)},$$
then 
$$(\tw'^*)_{j'} = t_{\nu'_{j'}}w'_{j'} = t_{(1,0)},\quad \textnormal{resp.},\quad t_{(1,0)}\fw,\quad \textnormal{resp.},\quad t_{(0,1)}.$$
We define
\begin{equation}
  \label{defoftauw'}
  \tau'_{\tw'} := \tau'\big((s,s)w'^{-1},~ \BC(\mu) + \eta' - (s,s)w'^{-1}(\nu')\big).
\end{equation}
(For the sake of comparison with \cite[\S 4.1]{BHHMS}, we note that the inertial type $\tau'_{\tw'}$ above relative to $\BC(\rhobar)$ is equal to the inertial type $\tau'_{\tw'\cdot t_{(\underline{1},\underline{1})}}$ of \cite{BHHMS} relative to $\BC(\rhobar)\otimes \omega$.  Note that $\tw' \in \textnormal{Adm}'^\vee(t_{\eta'})$ if and only if $\tw'\cdot t_{(\underline{1},\underline{1})} \in \textnormal{Adm}'^\vee(t_{(\underline{2},\underline{1})})$.)

Suppose we are given a tame inertial type $\tau' = \tau'((s_1,s_2),~(\mu_1,\mu_2))$, where $s_i \in \underline{W}, \lambda_i \in X^*(\underline{T}_\nU)$.  One easily checks that the tame inertial types $\tau'^{\vee}$ and $(\tau')^{\varphi^f}$ admit the following presentations:
\begin{eqnarray*}
  \tau'^\vee & \cong & \tau'\big((s_1,s_2),~(-\underline{\fw}(\mu_1), -\underline{\fw}(\mu_2))\big), \\
  (\tau')^{\varphi^f} & \cong & \tau'\big((s_2,s_1), (\mu_2,\mu_1)\big).
\end{eqnarray*}

\begin{lemma}
  \label{symiff}
Suppose $\tw' \in \textnormal{Adm}'^\vee(t_{\eta'})$.  We have ${\tau'_{\tw'}}^\vee \cong (\tau'_{\tw'})^{\varphi^f}$ if and only if $\tw' \in \textnormal{Adm}'^{\vee}(t_{\eta'})^{\sym}$, if and only if $(\tw'^*)_{j' + f} = (\tw'^*)_{j'}$ for all $j' \in \cJ'$.  
\end{lemma}

\begin{proof}
The second equivalence follows directly from the definition of $\tw'^*$.  We check the first equivalence.

If $\tw' \in \textnormal{Adm}'^{\vee}(t_{\eta'})^{\sym}$, then we may write $\tw'^* = t_{\nu'}w' = t_{(\nu,\nu)}(w,w)$, where $(\nu,w) \in X^*(\underline{T}_\nU) \times \underline{W}$.  This gives
\begin{eqnarray*}
  \tau'_{\tw'} & = & \tau'\big((sw^{-1},sw^{-1}),~ \BC(\mu) + \eta' - (sw^{-1}, sw^{-1})(\nu,\nu)\big)\\
   & = & \tau'\big((sw^{-1},sw^{-1}),~ \BC(\mu + \rho - sw^{-1}(\nu)) + (0,\rho + \underline{\fw}(\rho)) - (0,sw^{-1}(\nu) + \underline{\fw}sw^{-1}(\nu))\big)\\
   & = & \tau'\big((sw^{-1},sw^{-1}),~ \BC(\mu + \rho - sw^{-1}(\nu))\big).
\end{eqnarray*}
Thus, by the paragraph just before the statement of the lemma, we see that $({\tau'_{\tw'}})^{\varphi^f}$ and ${\tau'_{\tw'}}^\vee$ are isomorphic.

Conversely, suppose $({\tau'_{\tw'}})^{\varphi^f,\vee} \cong {\tau'_{\tw'}}$.  Since $\mu$ is $N$-deep in the fundamental $p$-alcove, the characters 
$$\BC(\mu) - (s,s)w'^{-1}(\nu') \qquad \textnormal{and} \qquad -(\un{\fw},\un{\fw})\pi'^f\left(\BC(\mu) + \eta' - (s,s)w'^{-1}(\nu')\right) - \eta'$$
(i.e., those appearing in $\tau'_{\tw'}$ and $({\tau'_{\tw'}})^{\varphi^f,\vee}$ after subtracting $\eta'$) are both $(N - 1)$-deep in the fundamental $p$-alcove for $\un{G}'$.  Hence, according to \cite[Lem. 2.2.3]{LLL}, the elements 
\begin{eqnarray*}
  & \big((s,s)w'^{-1},~ \BC(\mu) + \eta' - (s,s)w'^{-1}(\nu')\big) & \\
  & \textnormal{and}  & \\
 & \big(\pi'^f((s,s)w'^{-1}),~-(\un{\fw},\un{\fw})\pi'^f\left(\BC(\mu) + \eta' - (s,s)w'^{-1}(\nu')\right)\big) & 
\end{eqnarray*}
of $\underline{W}'\times X^*(\underline{T}')$ are good, and \cite[Prop. 9.2.3]{GHS} then implies that we have an isomorphism
$$\epsilon\left(R'_{(s,s)w'^{-1}}\big(\BC(\mu) + \eta' - (s,s)w'^{-1}(\nu')\big)\right) \cong R'_{(s,s)w'^{-1}}\big(\BC(\mu) + \eta' - (s,s)w'^{-1}(\nu')\big)$$
(where $\epsilon$ is as in the proof of Proposition \ref{prop:SW:extgr:U2} and we have written $\tw^{\prime\ast}=t_{\nu'}w'$).  Applying \cite[Lem. 3.12]{koziolmorra}, we see that the Deligne--Lusztig representation on the right-hand side is in the image of the base change map, and by the injectivity claim of \cite[Prop. 9.2.1]{GHS} we get
$$\tau'_{\tw'} \cong \tau'\big((v,v),~ \BC(\lambda)\big),$$
for some $v \in \underline{W}$ and some $\lambda\in X^*(\underline{T}_\nU)$ with $\lambda - \rho$ lying in the fundamental $p$-alcove.  By \cite[Prop. 2.2.15]{LLL}, there exists $z't_{\xi'} \in \underline{\tW}'$ satisfying:
\begin{itemize}
  \item if $z'_{j'} = 1$, then $\xi'_{j'} \in X^0(T_\nG)$;
  \item if $z'_{j'} = \fw$, then $\xi'_{j'} \in \eta'_{j'} + X^0(T_\nG)$;
  \item we have $(v,v) = z'(s,s)w'^{-1}\pi'(z')^{-1}$;
  \item we have
  \begin{equation}
    \label{sym-omega-conj}
    \BC(\lambda) = z'(\BC(\mu)) + z'(\eta') - z'(s,s)w'^{-1}(\nu') + p\xi' - (v,v)(\pi'(\xi')).
  \end{equation}
\end{itemize}
In particular, the pairings of the left-hand side of equation \eqref{sym-omega-conj} with $\alpha_{j'}^\vee$ and $\alpha_{j' + f}^\vee$ are equal, and thus the same is true of the right-hand side.  Reducing this equality of pairings modulo 2, we get
\begin{flushleft}
  $\langle\BC(\mu),\alpha_{j'}^\vee\rangle + 1 + \langle\nu',\alpha_{j'}^\vee\rangle + \langle\xi', \alpha_{j'}^\vee\rangle + \langle \pi'(\xi'), \alpha_{j'}^\vee\rangle$
\end{flushleft} 
\begin{flushright}
  $\equiv \langle\BC(\mu),\alpha_{j' + f}^\vee\rangle + 1 + \langle\nu',\alpha_{j' + f}^\vee\rangle + \langle\xi', \alpha_{j' + f}^\vee\rangle + \langle \pi'(\xi'), \alpha_{j' + f}^\vee\rangle~(\textnormal{mod}~2).$
\end{flushright}
Using that $\langle\nu', \alpha_{j'}^\vee\rangle = \pm 1$ and $\langle \BC(\mu),\alpha_{j'}^\vee\rangle = \langle \BC(\mu),\alpha_{j' + f}^\vee\rangle$ for all $j'$, the above equivalence simplifies to
$$\langle\xi', \alpha_{j'}^\vee\rangle + \langle \pi'(\xi'), \alpha_{j'}^\vee\rangle \equiv \langle\xi', \alpha_{j' + f}^\vee\rangle + \langle \pi'(\xi'), \alpha_{j' + f}^\vee\rangle~(\textnormal{mod}~2).$$
The preceding equivalence, along with the first two bullet points above, shows that $z'_{j'} = z'_{j' + f}$ if and only if $z'_{j' + 1} = z'_{j' + f + 1}$.  Thus, $z'$ must be of the form
$$z' = (z,z) \qquad \textnormal{or} \qquad  (z,z)(\underline{1},\underline{\fw})$$
for some $z\in \underline{W}$.  In either case, the third bullet point above implies that $w'_{j'} = w'_{j' + f}$ for all $j'$, and accordingly we write $w' = (w,w)$ for $w\in \underline{W}$.

Suppose now, by contradiction, that $z' = (z,z)(\underline{1},\underline{\fw})$.  This implies in particular that $\pi'^f(z') = z'(\underline{\fw},\underline{\fw})$.  Since the left-hand side of \eqref{sym-omega-conj} is fixed by the map $\omega' \longmapsto -(\underline{\fw},\underline{\fw})\pi'^f(\omega')$, the same is true of the right-hand side; rearranging the resulting equation (and examining the possibilities for the components of $\nu'$) gives 
\begin{equation}
  \label{sym-omega-conj-2}
  z'\left(\sum_{j'\in \cJ'} \big(\langle \BC(\mu), \alpha_{j'}^\vee \rangle + \delta_{j'}\big)\alpha_{j'}\right) = (p - (v,v)\pi')\left(-(\underline{\fw},\underline{\fw})\pi'^f(\xi') - \xi'\right),
\end{equation}
where $\delta_{j'}\in \{0,1,2\}$.  Thus, the left-hand side lies in $\Lambda'_{\textnormal{rt}}$, while the right-hand side lies in $(p - (v,v)\pi')X^*(\underline{T}')$, and consequently both sides lie in $\Lambda'_{\textnormal{rt}} \cap (p - (v,v)\pi')X^*(\underline{T}') = (p - (v,v)\pi')\Lambda'_{\textnormal{rt}}$.  Further, by pairing equation \eqref{sym-omega-conj-2} with $\alpha_{j'}^\vee$, we get that 
$$\langle \BC(\mu), \alpha_{j'}^\vee \rangle + \delta_{j'} = \frac{p \pm 1}{2}.$$
However, plugging this into equation \eqref{sym-omega-conj-2} gives an element which does not lie in $(p - (v,v)\pi')\Lambda'_{\textnormal{rt}}$, and we arrive at a contradiction.

We may therefore assume $z' = (z,z)$.  Since the left-hand side of \eqref{sym-omega-conj} is fixed by the map $\omega' \longmapsto -(\underline{\fw},\underline{\fw})\pi'^f(\omega')$, the same is true of the right-hand side; rearranging the resulting equation gives
$$\big(\eta' + (\underline{\fw},\underline{\fw})(\eta')\big) + (p - (v,v)\pi')\big(\xi' + (\underline{\fw},\underline{\fw})\pi'^f(\xi')\big) = (zsw^{-1},zsw^{-1})\big(\nu' + (\underline{\fw},\underline{\fw})\pi'^f(\nu')\big).$$
The left-hand side lies in $X^0(\underline{T}')$, and therefore pairing the right-hand side with $\alpha_{j'}^\vee$ gives
$$0 = \langle \nu', \alpha_{j'}^\vee \rangle - \langle \nu', \alpha_{j' + f}^\vee\rangle.$$
By the description of the components of $\nu'$, we conclude that $\nu'_{j'} = \nu'_{j' + f}$, which finishes the proof.
\end{proof}

\begin{rmk}
  The above lemma and proof hold verbatim under the weaker assumption that $\mu$ is $1$-deep in the fundamental $p$-alcove.
\end{rmk}

\begin{lemma}
  \label{kisin-mod-existence}
Suppose $\tw' \in \textnormal{Adm}'^\vee(t_{\eta'})^{\sym}$.  Up to isomorphism, there exists a unique semisimple polarized Kisin module $(\overline{\fM},\overline{\iota}) \in Y^{\leq (3,0), \tau'_{\tw'}}_{\pol}(\bbF)$ of shape $\tw'$ such that $T_{\dd}^*(\overline{\fM}) \cong (\BC(\rhobar)\otimes\omega)|_{\Gamma_{K_{2,\infty}}}$, and such that $\overline{\iota}$ is compatible with the polarization $\overline{\alpha}\otimes 1$ on $(\BC(\rhobar)\otimes \omega)|_{\Gamma_{K_{2,\infty}}}$ via $T_{\dd}^*$.
\end{lemma}

\begin{proof}
The existence of a unique such $\overline{\fM}$ in $Y^{\leq (3,0), \tau'_{\tw'}}(\bbF)$ follows from \cite[Lem. 4.1.1]{BHHMS}, and the existence and uniqueness of $\overline{\iota}$ follows as in \cite[Lem. 5.20]{koziolmorra}.  (See also the parenthetical comment after the definition \eqref{defoftauw'} of $\tau'_{\tw'}$ above for a comparison between \cite{BHHMS} and \cite{koziolmorra}.)
\end{proof}

\begin{lemma}
  \label{serre-wt-intn-JHtype}
  There is a unique bijection $\theta: \nW^?(\rhobar) \longrightarrow \{t_{(1,0)},~ t_{(0,1)}\}^f$ such that for any $\sigma \in \nW^?(\rhobar)$ and $\tw' \in \textnormal{Adm}'^\vee(t_{\eta'})^\sym$, we have $\sigma \in \JH(\overline{\sigma(\tau'_{\tw'})})$ if and only if $\tw'_{j'} \neq \theta(\sigma)_{j'~ (\textnormal{mod}~f)}$ for all $j' \in \cJ'$.  Moreover, if $\tau'$ is any inertial type satisfying $(\tau')^{\varphi^f} \cong \tau'^\vee$, then $\JH(\overline{\sigma(\tau')}) \cap \nW^?(\rhobar) \neq \emptyset$ if and only if $\tau' \cong \tau'_{\tw'}$ for some $\tw' \in \textnormal{Adm}'^\vee(t_{\eta'})^\sym$.  
\end{lemma}

\begin{proof}
  The proof follows in the exact same way as the proof of \cite[Lem. 4.1.2]{BHHMS}, using Propositions \ref{prop:SW:extgr:U2} and \ref{prop:JH:fct:extgr:U2}.
\end{proof}

\subsection{Single-type deformation rings for essentially conjugate self-dual representations}
\label{subsec:single-type}

Recall that $\rhobar:\Gamma_K \longrightarrow {}^C\nU_{1,1}(\bbF)$ is a tamely ramified $L$-parameter satisfying $\widehat{\imath}\circ \rhobar = \omega$, and such that $\rhobar|_{I_K} \cong \overline{\tau}(s,\mu + \eta)$ with $\mu$ lying $N$-deep in the fundamental $p$-alcove.  Fix $\tw' \in \textnormal{Adm}'^\vee(t_{\eta'})^\sym$.  We let $R^{\leq (2,-1), \tau'_{\tw'}}_{\rhobar}$ denote the maximal reduced, $\cO$-flat quotient of $R^\Box_{\rhobar}$ which parametrizes potentially crystalline lifts $\rho$ of $\rhobar$ which satisfy the following:
\begin{itemize}
  \item $\HT_{j'}(\BC(\rho)) \in \{(2,-1),~(1,0)\}$ for all $j' \in \cJ'$;
  \item $\HT_{j'}(\BC(\rho)) = \HT_{j' + f}(\BC(\rho))$ for all $j' \in \cJ'$;
  \item $\BC(\rho)$ has inertial type $\tau'_{\tw'}$;
  \item $\widehat{\imath}\circ \rho = \varepsilon$.
\end{itemize}

Given a lift $\rho$ of $\rhobar$ as above, the representation $\rho(1)$ is a lift of $\rhobar(1)$ having Hodge-Tate weights $(3,0)$ or $(2,1)$ in each embedding and fixed by an $f$-fold shift, tame inertial type $\tau'_{\tw'}$, and satisfying $\widehat{\imath}\circ \rho(1) = \varepsilon^3$.  We denote by $R^{\leq (3,0), \tau'_{\tw'}}_{\rhobar(1)}$ the maximal reduced, $\cO$-flat quotient of $R_{\rhobar(1)}^{\Box}$ parametrizing such lifts.  We therefore see that the map $\rho \longmapsto \rho(1)$ induces an isomorphism on deformation rings
\begin{equation}
  \label{isom-def-rings}
  R^{\leq (2,-1), \tau'_{\tw'}}_{\rhobar} \stackrel{\sim}{\longrightarrow} R^{\leq (3,0), \tau'_{\tw'}}_{\rhobar(1)}.
\end{equation}

\begin{thm}
\label{thm:height+monodromy}
Let $\rhobar$ and $\tw'$ be as above.  We have an isomorphism
$$R^{\leq (2,-1), \tau'_{\tw}}_{\rhobar}[\![ X_1, \ldots, X_{2f}]\!] \cong \left(R'\Big/ \Big(\sum_{j' = 0}^{f - 1} (I^{(j')}+I^{(j'+f)}+I_{\textnormal{sym},\infty}^{(j',j'+f)})\Big)\right)[\![ Y_1, \ldots, Y_{4}]\!]$$
where $R' := \widehat{\bigotimes}_{\cO, j' \in \cJ'} R^{(j')},$ $R^{(j')}$ and $I^{(j')}$ are described in \cite[Tables 1--3]{BHHMS}, and $I^{(j',j'+f)}_{\textnormal{sym},\infty}$ is described in Table \ref{Table1}.  Moreover, for any $\lambda = (\lambda_{j'}) \in \{(2,-1), (1,0)\}^{2f}$ satisfying $\lambda_{j' + f} = \lambda_{j'}$, the kernel of the natural surjection 
$$R^{\leq(2,-1), \tau'_{\tw'}}_{\rhobar}\bbra{X_1,\dots,X_{2f}} \longtwoheadrightarrow R^{\lambda, \tau'_{\tw'}}_{\rhobar}\bbra{X_1,\dots,X_{2f}}$$ 
is generated by the prime ideal $\sum_{j'=0}^{2f-1} \big(\fp^{(j'),\lambda_{2f-1-j'}}+I_{\textnormal{sym},\infty}^{(j',j'+f)}\big)$ of $R'$, where the ideals $\fp^{(j'),\lambda_{2f-1-j'}}$ of $R'$ are found in \cite[Tables 1--3]{BHHMS} (after shifting the superscript $\lambda$ by $(\underline{1},\underline{1})$).
\end{thm}

\begin{proof}
We outline the main ideas, following \cite[Pf. of Prop. 4.2.1]{BHHMS}, and using similar notation for ease of comparison.  Note that we use the isomorphism \eqref{isom-def-rings} in order to identify the deformation ring $R^{\leq (2,-1), \tau'_{\tw'}}_{\rhobar}$ with $R^{\leq (3,0), \tau'_{\tw'}}_{\rhobar(1)}$, so that we may apply the same strategy as in \emph{op. cit.}.  Note also that our Hodge--Tate weights and shapes differ from those of \emph{op. cit.} by a shift of $(\underline{1}, \underline{1})$.

Denote by $(\overline{\fM},\overline{\iota}) \in Y^{\leq (3,0),\tau'_{\tw'}}_{\pol}(\bbF)$ the polarized Kisin module furnished by Lemma \ref{kisin-mod-existence}, and let $\overline{\beta}$ be the standard basis of $\overline{\fM}$ (relative to which the matrix of partial Frobenius is $\overline{A}^{(j')} = D_{j'}\tw'_{j'}$).  Thus, $\overline{\beta}$ is a gauge basis of $(\overline{\fM},\overline{\iota})$ in the sense of \cite[Def. 5.16]{koziolmorra}.

We let $D_{\overline{\fM},\overline{\beta},\pol}^{\leq (3,0),\tau'_{\tw'}}$ denote the deformation problem defined by sending a local Artinian $\cO$-algebra $R''$ to the groupoid of tuples $(\fM,\iota,\beta,\jmath)$, where 
\begin{itemize}
  \item $\fM$ is a Kisin module over $R''$ with descent datum of type $\tau'_{\tw'}$ and $p$-adic Hodge type $\leq (3,0)$; 
  \item $\iota$ is a polarization on $\fM$; 
  \item $\beta$ is a gauge basis of $(\fM,\iota)$;
  \item $\jmath: \fM \otimes_{R''} \bbF \stackrel{\sim}{\longrightarrow} \overline{\fM}$ is an isomorphism sending $\beta\otimes 1$ to $\overline{\beta}$ and identifying $\iota \otimes 1$ with $\overline{\iota}$.  
\end{itemize}
(When $(3,0)$ is replaced by $(1,0)$, this deformation problem was denoted $D^{\tau'_{\tw'}, \overline{\beta}}_{\overline{\fM},\pol}$ in \cite[\S 5C1]{koziolmorra}.  The main difference is that the finite height condition which we impose is now given by equation \eqref{finheight} for all $j' \in \cJ'$.)

As in \cite[Thm.\ 4.17]{LLLM}, using now Subsection \ref{subsec:FH}, we see that $D_{\overline{\fM},\overline{\beta},\pol}^{\leq (3,0),\tau'_{\tw'}}$ is represented by the maximal reduced $\cO$-flat quotient of 
$$\widehat{\bigotimes}_{\cO, 0\leq j\leq f-1}R^{(j',j'+f)}\Big/\left(I^{(j'),\leq (3,0)}+I^{(j'+f),\leq (3,0)}+I^{(j',j'+f)}_{\textnormal{sym},\infty}\right),$$
where we define $R^{(j',j' + f)} := R^{(j')}\widehat{\otimes}_{\cO}R^{(j'+f)}$, and where $I^{(j',j'+f)}_{\textnormal{sym},\infty}$ denotes the symmetry ideal of $R^{(j',j' + f)}$ defined in Table \ref{Table1}.  (The generators of $I^{(j',j'+f)}_{\textnormal{sym},\infty}$ come from equation \eqref{polarizationeqs}, while those of $I^{(j'), \leq (3,0)}$ come from equation \eqref{finheight}.  See \cite[Tables 1--3, Row 4]{BHHMS} for explicit generators.)  We denote this quotient by $R_{\overline{\fM},\overline{\beta},\pol}^{\leq (3,0),\tau'_{\tw'}}$.

Exactly as in \cite{BHHMS}, we obtain generators for the ideals $I^{(j'),\nabla}$ defined by the monodromy condition of Proposition \ref{prop:mon:cond}.  These are recorded in row 5 of \cite[Tables 1--3]{BHHMS}.  As in \emph{op. cit.}, we note that the $O(p^{N - 2})$ and $O(p^{N - 3})$ tails depend on variables of all embeddings.  In particular, the ideals $I^{(j'),\nabla}$ are \emph{not} ideals of $R^{(j')}$ in general, but rather of $R'$.

Let $R^{\leq(3,0),\tau'_{\tw'},\nabla}_{\overline{\fM},\overline{\beta},\pol}$ be the maximal reduced and $\cO$-flat quotient of 
$$R'\Big/ \Big(\sum_{j\in\cJ'}(I^{(j'),\leq (3,0)}+I^{(j'),\nabla}+I^{(j',j'+f)}_{\textnormal{sym},\infty})\Big).$$ 
Using that $\textnormal{ad}^0(\rhobar)$ is cyclotomic free, we may proceed as in \cite[Thm. 5.12]{LLLM} (see also \cite[\S\S 5C8 -- 5C10]{koziolmorra}) to conclude that
\begin{equation}
\label{eq:def:ring}
R^{\leq(2,-1), \tau'_{\tw'}}_{\rhobar}\bbra{X_1,\dots,X_{2f}} \cong R^{\leq(3,0), \tau'_{\tw'}}_{\rhobar(1)}\bbra{X_1,\dots,X_{2f}} \cong R^{\leq(3,0),\tau'_{\tw'},\nabla}_{\overline{\fM},\overline{\beta},\pol}\bbra{Y_1,\dots,Y_4}.
\end{equation}
Thus, in order to prove the first part of theorem, it suffices to find an explicit description of $R^{\leq (3,0),\tau'_{\tw'},\nabla}_{\overline{\fM},\overline{\beta},\pol}$.  To this end, let us define 
$$I_\infty := \ker\left(R' \longtwoheadrightarrow R^{\leq(3,0),\tau'_{\tw'},\nabla}_{\overline{\fM},\overline{\beta},\pol}\right).$$
Our task will be to explicitly identify $I_\infty$.

For each $j' \in \cJ'$, we define a dense polynomial sub-$\cO$-algebras $R^{(j')}_{\mathrm{poly}}$ of $R^{(j')}$ by
\begin{align}
\notag  R^{(j')}_{\mathrm{poly}} := \cO\left[c_{1,1},~ d_{1,1},~ x_{1,1}^*,~ \frac{c_{1,2}}{e_{1,1}^*},~ c_{2,1},~ \frac{d_{2,1}}{d_{2,2}^*},~ c_{2,2},~ x_{2,2}^*\right] & \qquad\text{if $\tw'_{2f-1-j'} = t_{(1,0)}$},\\ 
\label{polyrings}  R^{(j')}_{\mathrm{poly}} := \cO\left[c_{1,1},~ \frac{d_{1,1}}{d^*_{1,2}},~ c_{1,2},~ x^*_{1,2},~ c_{2,1},~ x_{2,1}^*,~ c_{2,2},~ \frac{d_{2,2}}{d_{2,1}^*}\right] & \qquad\text{if $\tw'_{2f-1-j'} = \fw t_{(1,0)}$},\\
\notag  R^{(j')}_{\mathrm{poly}} := \cO\left[c_{1,1},~ x^*_{1,1},~ c_{1,2},~ \frac{d_{1,2}}{d_{1,1}^*},~ \frac{c_{2,1}}{e_{2,2}^*},~ c_{2,2},~ d_{2,2},~ x_{2,2}^*\right] & \qquad\text{if $\tw'_{2f-1-j'} = t_{(0,1)}$}. 
\end{align}
We also define several ideals of these polynomial algebras:
\begin{itemize}
\item  For $j' \in \cJ'$, we let $I^{(j',j' + f)}_{\textnormal{sym},\textnormal{poly}} \subset R^{(j')}_{\textnormal{poly}}\otimes_{\cO}R^{(j' + f)}_{\textnormal{poly}}$ denote the ideal $I^{(j',j' + f)}_{\textnormal{sym},\infty} \cap R^{(j')}_{\textnormal{poly}}\otimes_{\cO}R^{(j' + f)}_{\textnormal{poly}}$.  We note that the generators of the symmetry ideal $I^{(j', j' + f)}_{\textnormal{sym},\infty}$ in Table \ref{Table1} are already contained in $R^{(j')}_{\textnormal{poly}}\otimes_{\cO}R^{(j' + f)}_{\textnormal{poly}}$.

\item For $j' \in \cJ'$, we let $I^{(j')}_{\mathrm{poly}}$ denote the ideal of $R_{\mathrm{poly}}^{(j')}$ generated by the elements in row 6 of \cite[Tables 1--3]{BHHMS} \emph{without their $O(p^{N-8})$ tails}.

\item Let $R'_{\mathrm{poly}} := \bigotimes_{\cO,j'\in \cJ'}R_{\mathrm{poly}}^{(j')}$, and define the ideal $I_{\mathrm{poly}} := \sum_{0 \leq j' \leq f - 1}I^{(j')}_{\mathrm{poly}}+I^{(j'+f)}_{\mathrm{poly}}+I_{\textnormal{sym}, \textnormal{poly}}^{(j',j'+f)}$.
\end{itemize}

As in the proof of \cite[Prop. 4.2.1]{BHHMS}, we get that $I^{(j')}_{\mathrm{poly}}\subset (I_{\infty},\ p^{N-5})$ for all $j'$, hence $I^{(j')}_{\mathrm{poly}}+I^{(j'+f)}_{\mathrm{poly}}+I_{\textnormal{sym},\textnormal{poly}}^{(j',j'+f)}\subset (I_{\infty},\ p^{N-5})$ for all $j'$ and finally $I_{\mathrm{poly}}\subset (I_{\infty},\ p^{N-5})$.

Now fix $j'\in \cJ'$.  We examine the following commutative diagram:
\begin{equation}
\label{commdiag:poly}
\begin{tikzcd}[column sep=12ex]
\cO \ar[r] \ar[d] & R'/I_\infty \ar[d, twoheadrightarrow] \\
\Big(R^{(j')}_{\mathrm{poly}}\otimes_{\cO}R^{(j'+f)}_{\mathrm{poly}}\Big)\Big/\Big(I^{(j')}_{\mathrm{poly}}+I^{(j'+f)}_{\mathrm{poly}}+I_{\textnormal{sym},\textnormal{poly}}^{(j',j'+f)}\Big) \ar[r, "\phi^{(j',j' + f)}"] & R'/(I_\infty, p^{N - 5})\\
\end{tikzcd}
\end{equation}
where $\phi^{(j',j' + f)}$ is induced by the inclusion $R^{(j')}_{\mathrm{poly}} \otimes_{\cO} R^{(j' + f)}_{\textnormal{poly}}\longhookrightarrow R'$.  Further, we let $H^{(j', j' + f)}$ be the Elkik ideal relative for the finitely presented algebra 
$$\cO \longrightarrow \Big(R^{(j')}_{\mathrm{poly}}\otimes_{\cO}R^{(j'+f)}_{\mathrm{poly}}\Big)\Big/\Big(I^{(j')}_{\mathrm{poly}}+I^{(j'+f)}_{\mathrm{poly}}+I_{\textnormal{sym},\textnormal{poly}}^{(j',j'+f)}\Big),$$
as in \cite[\S 0.2]{elkik}. (Thus $H^{(j',j' + f)}$ is an ideal of the polynomial ring $R_{\poly}^{(j')}\otimes_{\cO}R_{\poly}^{(j' + f)}$.)

We record several results about the ideal $I^{(j')}_{\mathrm{poly}} + I^{(j'+f)}_{\mathrm{poly}} + I_{\textnormal{sym},\textnormal{poly}}^{(j',j'+f)}$.  Note first that by using the explicit generators of $I^{(j', j' + f)}_{\sym, \poly}$ given in Table \ref{Table1}, any element of $I^{(j' + f)}_{\mathrm{poly}}$ can be written as a sum of elements in $I^{(j')}_{\poly}$ and $I^{(j',j' + f)}_{\sym,\poly}$ (for example, for $\tw'_{2f - 1 - j'} = t_{(1,0)}$, we replace $c_{1,2}^{(j' + f)}$ by $-c_{1,2}^{(j')}e_{1,1}^{*(j' + f)}d_{2,2}^{*(j' + f)}$, etc.).  We therefore deduce the following.

\begin{sublemma}
\label{lem:symideal}
We have $I^{(j')}_{\mathrm{poly}} + I^{(j'+f)}_{\mathrm{poly}} + I_{\textnormal{sym},\textnormal{poly}}^{(j',j'+f)} = I^{(j')}_{\mathrm{poly}} + I_{\textnormal{sym},\textnormal{poly}}^{(j',j'+f)}$.  
\end{sublemma}

We now imitate the proof of \cite[Lem. 4.2.3]{BHHMS}.

\begin{sublemma}
\label{lem:Elkik:1}
We have $p^{3}\in H^{(j',j'+f)} + I^{(j')}_{\poly} + I^{(j'+f)}_{\poly} + I_{\sym,\poly}^{(j',j'+f)}$.
\end{sublemma}

\begin{proof}[Proof of Sublemma \ref{lem:Elkik:1}]
We suppose $\tw'_{2f - 1 - j'} = \tw'_{f - 1 - j'} = \fw t_{(1,0)}$, the other cases being similar.  Let us label the five generators of $I^{(j')}_{\poly}$ by $G_1^{(j')}, \ldots, G_5^{(j')}$ (and similarly for $I^{(j' + f)}_{\poly}$), and label the eight generators of $I^{(j', j' + f)}_{\sym, \poly}$ by $\Sigma_1, \ldots, \Sigma_8$.  For brevity we denote the ideal $I^{(j')}_{\mathrm{poly}}+I^{(j'+f)}_{\mathrm{poly}}+I_{\textnormal{sym},\textnormal{poly}}^{(j',j'+f)}$ by $J$, and write $R^{(j')}_{\poly}\otimes_{\cO}R^{(j' + f)}_{\poly} = \cO[X_1^{(j')}, \ldots, X_8^{(j')}, X_1^{(j' + f)}, \ldots, X_8^{(j' + f)}]$ (relative to the ordering given in equations \eqref{polyrings}).

The Elkik ideal $H^{(j',j' + f)}$ (cf. \cite[\S0.2]{elkik}) is defined as 
$$H^{(j', j' + f)} = \sum_{\alpha} ((\alpha):J)\Delta_{(\alpha)},$$
where $\alpha$ is a subset of the generators $\{G_1^{(j')}, \ldots, G_5^{(j')}, G_1^{(j' + f)}, \ldots, G_5^{(j' + f)}, \Sigma_1, \ldots, \Sigma_8\}$ of $J$, $(\alpha)$ is the ideal of $R^{(j')}_{\poly}\otimes_{\cO}R^{(j' + f)}_{\poly}$ generated by $\alpha$, and $\Delta_{(\alpha)}$ denotes the ideal generated by the $|\alpha|\times|\alpha|$-minors of the Jacobian matrix $(\partial F/\partial X_\ell^{(k)})_{F \in \alpha, 1 \leq \ell \leq 8, k = j', j' + f}$.  In particular, if we take 
$$\alpha' = \{G_1^{(j')}, \ldots, G_5^{(j')}, \Sigma_1, \ldots, \Sigma_8\},$$
then Sublemma \ref{lem:symideal} implies that $((\alpha'):J) = R^{(j')}_{\poly} \otimes_\cO R^{(j' + f)}_{\poly}$.  Thus, $H^{(j', j' + f)}$ contains $\Delta_{(\alpha')}$.

The ideal $\Delta_{(\alpha')}$ contains the determinant of the $13 \times 13$ matrix obtained by deleting the rows corresponding to the variables $X_4^{(j')}, X_6^{(j')},$ and $X_8^{(j')}$.  This matrix is block upper-triangular (as $\partial G_i^{(j')}/\partial X_\ell^{(j' + f)} = 0$), and one checks that its determinant is equal to
\begin{equation}
\label{blockdet}
\frac{\partial G_5^{(j')}}{\partial X_2^{(j')}}\det\left(\frac{\partial \Sigma_i}{\partial X_\ell^{(j' + f)}}\right)_{1 \leq i, \ell \leq 8}.
\end{equation}
In the original coordinates \eqref{polyrings}, the determinant $\det(\partial \Sigma_i/\partial X_\ell^{(j' + f)})_{1 \leq i, \ell \leq 8}$ is equal to 
$$-\left(d_{1,2}^{*(j')}d_{2,1}^{*(j')}\right)^7d_{1,2}^{*(j' + f)}d_{2,1}^{*(j' + f)}.$$
Multiplying equation \eqref{blockdet} by $(d_{1,2}^{*(j' + f)}d_{2,1}^{*(j' + f)})^6$ gives an element of $H^{(j', j' + f)}$; using the fact that $(d_{1,2}^{*(j')}d_{2,1}^{*(j')}d_{1,2}^{*(j' + f)}d_{2,1}^{*(j' + f)})^7 - 1 \in I^{(j',j' + f)}_{\sym,\poly}$ we conclude that
$$\frac{\partial G_5^{(j')}}{\partial X_2^{(j')}} \in H^{(j', j' + f)} + I^{(j', j' + f)}_{\sym,\poly}.$$
The rest of the proof now follows exactly as in \cite[Lem. 4.2.3]{BHHMS}.
\end{proof}

Now let us write $A = R'/I_\infty$, $R^{(j')}_{\poly}\otimes_{\cO}R^{(j' + f)}_{\poly} = \cO[X_1^{(j')}, \ldots, X_8^{(j')}, X_1^{(j' + f)}, \ldots, X_8^{(j' + f)}]$, and let $J$ denote the ideal of the latter ring corresponding to $I^{(j')}_{\poly} + I^{(j'+f)}_{\poly} + I_{\sym,\poly}^{(j',j'+f)}$.  Pushing out the top left corner of the diagram \eqref{commdiag:poly} and using the universal property of pushouts gives
\begin{center}
\begin{tikzcd}
\cO \ar[r] \ar[d]\arrow[dr, phantom, "\lrcorner", very near end] & A \ar[d] \ar[ddr, twoheadrightarrow, bend left] & \\
\cO[X_1^{(j')}, \ldots, X_8^{(j' + f)}]/J \ar[r] \ar[rrd, "\phi^{(j',j' + f)}", bend right=20]& A[X_1^{(j')}, \ldots, X_8^{(j' + f)}]/J \ar[dr, "\exists \psi"] & \\
 & &  A/p^{N - 5}\\
\end{tikzcd}
\end{center}
By choosing a lift of $\psi$, we obtain a tuple $\underline{a} = (a_1^{(j')}, \ldots, a_8^{(j')}, a_1^{(j' + f)}, \ldots, a_8^{(j' + f)}) \in A^{16}$ which satisfies $J(\underline{a}) \subset p^{N - 5}A$.

Now set $B = A[X_1^{(j')}, \ldots, X_8^{(j' + f)}]/J$, and let $H_B^{(j', j' + f)} \subset A[X_1^{(j')}, \ldots, X_8^{(j' + f)}]$ denote the Elkik ideal for the finitely presented algebra $A \longrightarrow B$.  In particular, $H_B^{(j', j' + f)}$ contains the image of $H^{(j', j' + f)}$, and thus Sublemma \ref{lem:Elkik:1} gives 
$$p^3 \in H^{(j', j' + f)}(\underline{a}) + J(\underline{a}) \subset H_B^{(j', j' + f)}(\underline{a}) + p^{N - 5}A.$$
Since $A$ is $p$-adically complete and $N - 5 > 3$, this implies $p^3 \in H_B^{(j', j' + f)}(\underline{a})$.  As in \cite{BHHMS}, we apply \cite[Lem. 1]{elkik} to obtain an element $\widetilde{\underline{a}} \in A^{16}$ congruent to $\underline{a}$ modulo $p^{N - 8}$ and which satisfies $J(\widetilde{\underline{a}}) = 0$.  This means that we can construct an $\cO$-algebra homomorphism 
\begin{eqnarray*}
\widetilde{\phi}^{(j', j' + f)}: \cO[X_1^{(j')}, \ldots, X_8^{(j' + f)}]/J & \longrightarrow & R'/I_\infty \\
X_{\ell}^{(k)} & \longmapsto & \widetilde{a}_\ell^{(k)}
\end{eqnarray*}
which agrees with $\phi^{(j', j' + f)}$ modulo $p^{N - 8}$.  Taking a tensor product over $0 \leq j' \leq f - 1$, we obtain an  $\cO$-algebra homomorphism
$$\widetilde{\phi}: R'_{\poly}/I_\poly \longrightarrow R'/I_\infty$$
which agrees with the map induced by the natural inclusion $R_\poly \longhookrightarrow R'$ modulo $p^{N - 8}$.  Since $N \geq 10$, this map is continuous and extends to a surjection
$$\widetilde{\phi}: R'/I_\poly \longtwoheadrightarrow R'/I_\infty.$$

By Sublemma \ref{lem:symideal}, we have a (non-canonical) isomorphism $(R^{(j)}\otimes_{\cO}R^{(j' + f)})/(I^{(j')}_{\mathrm{poly}}+I^{(j' + f)}_{\mathrm{poly}}+I_{\textnormal{sym}}^{(j', j'+f)})\cong R^{(j')}/I^{(j')}_{\mathrm{poly}}$, with the latter being reduced, $\cO$-flat, with two geometrically integral irreducible components of relative dimension 3 over $\cO$ by \cite[Lem. 3.3.1]{BHHMS}.  As in \textit{op. cit.}, the surjection $\widetilde{\phi}$ is an isomorphism provided that the target $R'/I_\infty \cong R^{\leq (3,0), \tau'_{\tw'}, \nabla}_{\overline{\fM}, \overline{\beta}, \pol}$ has at least $2^f$ irreducible components.  By equation \eqref{eq:def:ring}, it suffices to show $R^{\leq(2,-1),\tau'_{\tw'}}_{\rhobar}$ has at least $2^f$ irreducible components.

In order to verify the above claim, we proceed as follows.  Let $\lambda = (\lambda_{j'}) \in \{(2,-1),(1,0)\}^{2f}$ be an element satisfying $\lambda_{j' + f} = \lambda_{j'}$.  We first claim that
\begin{equation}
\label{eq:int}
\JH\left(\overline{\sigma(\tau'_{\tw'})\otimes_E \bigotimes_{E,0 \leq j' \leq f - 1} V_E(\lambda_{j'}-(1,0))^{(j')}}\right) \cap \nW^?(\rhobar) \neq \emptyset
\end{equation}
(see Subsection \ref{subsec:WM} for the definition of $V_E(\lambda_{j'}-(1,0))^{(j')}$).  Indeed, proceeding as in \cite{BHHMS}, we see that the left-hand side contains $\JH(\overline{\sigma(\tau'_{\tw'})}) \cap \nW^?(\rhobar)$, which is non-empty thanks to Lemma \ref{serre-wt-intn-JHtype}.  Applying Corollary \ref{main-global-lift-cor}, we see that $\rhobar$ admits a potentially crystalline lift with Hodge--Tate weights given by $\lambda$, of inertial type $\tau'_{\tw'}$, and whose multiplier is $\varepsilon$.  Since such a lift lies on the irreducible component $\Spec(R^{\lambda, \tau'_{\tw'}}_{\rhobar})$ of $\Spec(R^{\leq (2,-1),\tau'_{\tw'}}_{\rhobar})$, the claim of the previous paragraph is verified, and $\widetilde{\phi}$ is an isomorphism.

As $\widetilde{\phi}$ is an isomorphism, we have explicitly identified the structure of the ring $R^{\leq (3,0), \tau'_{\tw'},\nabla}_{\overline{\fM},\overline{\beta},\pol}$.  Furthermore, since $\widetilde{\phi}$ induces the natural map modulo $p^{N-8}$, we get that $(I_{\mathrm{poly}},p^{N-8})=(I_\infty,p^{N-8})$.

To conclude, we examine the irreducible components of $R'/I_\infty$.  The following result is proved exactly as \cite[Lem. 4.2.4]{BHHMS}.

\begin{sublemma}\label{lem:automorphism}
There exists an automorphism of local $\cO$-algebras $\psi:R'\congto R'$ such that
\begin{center}
\begin{tikzcd}
R'\ar[r, "\psi", "\sim"'] \ar[d, twoheadrightarrow] &R'\ar[d, twoheadrightarrow]\\
R'/I_{\mathrm{poly}} \ar[r, "\wt{\phi}", "\sim"'] & R'/I_\infty
\end{tikzcd}
\end{center}
commutes.  Further, $\psi$ may be chosen so that it induces the identity modulo $p^{N-8}$ and $\psi(\sum_{j'\in \cJ'}I_{\sym,\infty}^{(j',j' + f)}) = \sum_{j'\in \cJ'}I_{\sym,\infty}^{(j',j' + f)}$.
\end{sublemma}

\begin{proof}[Proof of Sublemma \ref{lem:automorphism}]
  This follows in a similar manner to \cite[Lem. 4.2.4]{BHHMS}: if $X$ is a power series generator of $R^{(j')}$ with $0 \leq j' \leq f - 1$, we define $\psi(X) = X + p^{N - 8}\varepsilon(X)$ (with $\varepsilon$ as in \textit{op. cit.}), and if $\Sigma$ is a generator of $\sum_{j' \in \cJ'}I_{\sym,\infty}^{(j', j' + f)}$ we define $\psi(\Sigma) = \Sigma$.  This uniquely determines a ring isomorphism satisfying the required properties.
\end{proof}

Thus, $\psi$ identifies $I_{\mathrm{poly}}$ with $I_\infty$, and we obtain 
$$I_\infty=\sum_{j' = 0}^{f - 1} I^{(j')}+I^{(j'+f)}+I_{\textnormal{sym},\infty}^{(j',j'+f)},$$ 
where $I^{(j')} := \psi(I^{(j')}_{\poly})$ is the ideal of $R'$ given by the ``explicit'' generators in row 6 of \cite[Tables 1--3]{BHHMS}.  Moreover, it follows that the distinct minimal primes containing $I_\infty$ are exactly the ideals 
$$\fp^{\lambda} := \sum_{j' = 0}^{f - 1}\big(\fp^{(j'),\lambda_{2f-1-j'}} + I_{\textnormal{sym},\infty}^{(j',j'+f)}\big)$$ 
for $\lambda\in\{(2,-1),~(1,0)\}^{2f}$ satisfying $\lambda_{j'} = \lambda_{j' + f}$, where the $\fp^{(j'),\lambda_{2f-1-j'}}$ are defined in rows 7 and 8 of \cite[Tables 1--3]{BHHMS}.  The last statement in the theorem is proved exactly as in the last two paragraphs of the proof of \cite[Prop. 4.2.1]{BHHMS}.
\end{proof}

\begin{cor}
  \label{cor:single-type-cor}
  Suppose $\lambda = (\lambda_{j'}) \in \{(2,-1),~(1,0)\}^{2f}$ satisfies $\lambda_{j' + f} = \lambda_{j'}$, and let $\tw' \in \textnormal{Adm}'^\vee(t_{\eta'})^\sym$.  Then the special fiber of $R^{\lambda, \tau'_{\tw'}}_{\rhobar}$ is reduced and all its irreducible components are formally smooth over $\bbF$.
\end{cor}

\begin{proof}
  Using Theorem \ref{thm:height+monodromy} and Sublemma \ref{lem:automorphism}, this follows in exactly the same way as the proof of \cite[Cor. 4.2.6]{BHHMS} (noting that the isomorphism of Sublemma \ref{lem:automorphism} fixes the generators of the symmetry ideal, so that we obtain the analog of \cite[equation (27)]{BHHMS}, and therefore we may reduce geometric questions about $R^{\lambda, \tau'_{\tw'}}_{\rhobar}$ to the corresponding ``polynomial'' ideals).
\end{proof}

\subsection{Multi-type deformation rings for essentially conjugate self-dual representations}
\label{sec:multi-type}

We maintain the setup of the previous section, so that $\rhobar: \Gamma_K \longrightarrow {}^C\nU_{1,1}(\bbF)$ is a tamely ramified $L$-parameter satisfying $\widehat{\imath}\circ \rhobar = \omega$, and such that $\rhobar|_{I_K} \cong \overline{\tau}(s,\mu + \eta)$ for some $\mu$ lying $N$-deep in the fundamental $p$-alcove.  We fix $\sigma \in \nW^?(\rhobar)$, and let $R^{\leq (2,-1),\sigma}_{\rhobar}$ denote the maximal reduced, $\cO$-flat quotient of $R^{\Box}_{\rhobar}$ which parametrizes potentially crystalline lifts $\rho$ of $\rhobar$ which satisfy the following:
\begin{itemize}
  \item $\HT_{j'}(\BC(\rho)) \in \{(2,-1),~(1,0)\}$ for all $j' \in \cJ'$;
  \item $\HT_{j'}(\BC(\rho)) = \HT_{j' + f}(\BC(\rho))$ for all $j' \in \cJ'$;
  \item $\BC(\rho)$ has inertial type $\tau'$ for some tame inertial type satisfying $(\tau')^{\varphi^{f}} \cong \tau'^\vee$ and $\sigma \in \JH(\overline{\sigma(\tau')})$;
  \item $\widehat{\imath}\circ \rho = \varepsilon$.
\end{itemize}
Analogously to the isomorphism \eqref{isom-def-rings}, the map $\rho \longmapsto \rho(1)$ induces an isomorphism between $R^{\leq (2,-1),\sigma}_{\rhobar}$ and the deformation rings considered in \cite[\S 4.3]{BHHMS}.  To see this, note that 
\begin{eqnarray*}
  \sigma \in \nW^?(\rhobar) & \stackrel{\textnormal{\cite[Thm. 4.9]{koziolmorra}}}{\Longleftrightarrow} & \BC(\sigma) \in \nW^?\big(\BC(\rhobar)\big)\\
  & \stackrel{\textnormal{\cite
[Prop. 6.23]{herzig:duke}}}{\Longleftrightarrow} & \BC(\sigma)\otimes_{\bbF}\nN_{k_{K_2}/\bbF_p}\circ \det \in \nW^?\big(\BC(\rhobar(1))\big)
\end{eqnarray*}
and similarly
\begin{eqnarray*}
  \sigma \in \JH\big(\overline{\sigma(\tau')}\big) & \stackrel{\textnormal{\cite[Lem. 3.26]{koziolmorra}}}{\Longleftrightarrow} & \BC(\sigma) \in \JH\big(\overline{\sigma'(\tau')}\big) \\
  & \Longleftrightarrow & \BC(\sigma)\otimes_{\bbF}\nN_{k_{K_2}/\bbF_p}\circ \det \in \JH\big(\overline{\sigma'(\tau')}\otimes_{\bbF}\nN_{k_{K_2}/\bbF_p}\circ \det\big),
\end{eqnarray*}
which verifies the claimed isomorphism.  This will allow us to use the strategy of \cite[\S 4.3]{BHHMS}.

We let $\tw'_\sigma \in \textnormal{Adm}'^\vee(t_{\eta'})^\sym$ denote the unique element satisfying $(\tw'_\sigma)_{j'} = \theta(\sigma)_{j'~(\textnormal{mod}~f)}$ for all $j' \in \cJ'$ (where $\theta$ is the bijection of Lemma \ref{serre-wt-intn-JHtype}).  We then define 
$$X(\sigma) := \{ \tw' \in \mathrm{Adm}'^\vee(t_{\eta'})^\sym : \tw'_{j'}\neq (\tw'_\sigma)_{j'}~\textnormal{for all $j' \in \cJ'$}\},$$
so that
$\Spec R^{\leq(2,-1),\sigma}_{\rhobar}$ is the flat closure of $\bigcup_{\tw' \in X(\sigma)} \Spec R^{\leq(2,-1),\tau'_{\tw'}}_{\rhobar}[1/p]$ inside $\Spec R^{\square}_{\rhobar}$.

\begin{thm}\label{prop:multitype-def-ring}
  Let $\rhobar$ and $\sigma$ be as above.  We have an isomorphism
$$
R^{\leq(2,-1),\sigma}_{\rhobar}\bbra{X_1,\dots,X_{2f}}\cong\bigg(
S'\Big/\bigcap_{\tw' \in X(\sigma)} \sum_{j' = 0}^{2f - 1} \big(I_{\tw'}^{(j')} + I_{\tw'}^{(j' + f)} + J_{\textnormal{sym},\infty}^{(j',j'+f)}\big)
\bigg)\bbra{Y_1,\dots,Y_4},
$$
where $S' := \widehat{\bigotimes}_{\cO, j'\in\cJ'}S^{(j')}$
and the ring $S^{(j')}$ and the ideals $I_{\tw'}^{(j')}$ of $S'$ are as in \cite[Table 4]{BHHMS} if
$(\tw'_\sigma)_{2f - 1 - j'} = t_{(0,1)}$, whereas $S^{(j')}$ and the ideals $I_{\tw'}^{(j')}$ of $S'$ are as in \cite[Table 5]{BHHMS} if $(\tw'_\sigma)_{2f - 1 - j'} = t_{(1,0)}$, and the ideals $J_{\textnormal{sym},\infty}^{(j',j'+f)}$ are defined in Table \ref{Table2}.

Via this isomorphism, for any choice of $\lambda = (\lambda_j) \in \{(2,-1), (1,0)\}^{2f}$ satisfying $\lambda_{j' + f} = \lambda_{j'}$ for all $j'$, and any choice of $\tw' \in X(\sigma)$, the kernel of the natural surjection
$$R^{\leq(2,-1),\sigma}_{\rhobar}\bbra{X_1,\dots,X_{2f}} \longtwoheadrightarrow R^{\lambda, \tau'_{\tw'}}_{\rhobar}\bbra{X_1,\dots,X_{2f}}$$ 
is generated by the prime ideal $\sum_{j' = 0}^{2f - 1} (\fp^{(j'),\lambda_{2f - 1 - j'}}_{\tw'} + J_{\textnormal{sym},\infty}^{(j',j'+f)})$ of $S'$, where the ideals $\fp^{(j'),\lambda_{2f - 1 - j'}}_{\tw'}$ of $S'$ are found in \cite[Tables 4,5]{BHHMS}.
\end{thm}

\begin{proof}
As with the proof of Theorem \ref{thm:height+monodromy}, we may apply the twist $\rho \longmapsto \rho(1)$ to put ourselves in the setting of \cite[Prop. 4.3.1]{BHHMS}.  We proceed as in that reference.

The \'etale $\varphi$-module associated to $(\BC(\rhobar)\otimes \omega)|_{\Gamma_{K_{2,\infty}}}$ is determined by $\textnormal{Mat}(\varphi^{(j')}) = (D(s,s)^*t_{(\BC(\mu) + \eta' + (\underline{1},\underline{1}))^*})_{j'}$ in a suitable basis, for some $D = (D_{j'})_{j'\in \cJ'} \in \underline{T}'(\bbF)$.  Note that we may choose the basis such that $D_{j' + f} = (-1)^{\delta_{j'(\textnormal{mod}~f),f - 1}}\dot{\fw}D_{j'}^{-1}\dot{\fw}^{-1}$.  Let us define $\delta_{1,2}^{(j')},\delta_{2,1}^{(j')} \in \cO^\times$ to be the Teichm\"uller lifts of the diagonal entries of $D_{2f - 1 - j'}$.

Let $\overline{S}'$ denote the ring 
$$\left(\widehat{\bigotimes}_{\cO, j' \in \cJ'} S^{(j')}\right)\Big/ \bigcap_{\tw' \in X(\sigma)} \sum_{j' \in \cJ'}  I^{(j')}_{\tw'},$$
and let $\cM$ denote the $\varphi$-module over $\cO_{\cE,\overline{S}'}$ given by 
\begin{flushleft}
$\displaystyle{\textnormal{Mat}(\varphi_\cM^{(2f - 1 - j')}) = \begin{pmatrix}
    (v + p)(\delta_{1,2}^{(j)} + x_{1,2}^{*(j')}) + c_{1,2}^{(j')} + \frac{b_{1,2}^{(j')}}{v} & \frac{1}{v}\big((v + p)d_{1,1}^{(j')} + c_{1,1}^{(j')}\big) \\ 
    (v + p)d_{2,2}^{(j')} + c_{2,2}^{(j')} & (v + p)(\delta_{2,1}^{(j')} + x_{2,1}^{*(j')}) + c_{2,1}^{(j')} + \frac{b_{2,1}^{(j')}}{v}
  \end{pmatrix}}$
\end{flushleft}
\begin{flushright}
$\displaystyle{ \cdot s_{j'}^{-1}v^{(\BC(\mu) + \eta')_{j'}}}$,  
\end{flushright}
where $b_{2,1}^{(j')} = 0$ if $(\tw'_\sigma)_{2f - 1 - j'} = t_{(0,1)}$ and $b_{1,2}^{(j')} = 0$ if $(\tw'_\sigma)_{2f - 1 - j'} = t_{(1,0)}$.

Consider first the base extension $\cM_{\bbF} := \cM\otimes_{\overline{S}'}\bbF$.  By construction, we have an isomorphism $\bbV^*_{K_2}(\cM_{\bbF}) \cong (\BC(\rhobar)\otimes\omega)|_{\Gamma_{K_{2,\infty}}}$.  We fix also a basis $\gamma_\bbF$ of $\bbV^*_{K_2}(\cM_{\bbF})$ consisting of $\Gamma_{K_{2,\infty}}$-eigenvectors, and fix an $\overline{S}'$-basis $\gamma$ of $\bbV^*_{K_2}(\cM)$ which lifts $\gamma_\bbF$.  This choice of $\gamma$ gives a Galois representation of $\Gamma_{K_{2,\infty}}$ valued in $\nG\nL_2(\overline{S}')$, which is a lift of $(\BC(\rhobar)\otimes\omega)|_{\Gamma_{K_{2,\infty}}}$.  Consequently, the $\varphi$-module $\cM$ along with the choice of $\gamma$ produces a morphism
\begin{equation}
  \label{multitype:eqn1}
  R_{\BC(\rhobar(1))|_{\Gamma_{K_{2,\infty}}}}^\Box \longrightarrow \overline{S}'.
\end{equation}

Next, we examine the $\varphi$-module $\cM_{\overline{S}'/J_{\sym,\infty}} := \cM\widehat{\otimes}_{\overline{S}'} \overline{S}'/J_{\sym,\infty}$, where $J_{\sym,\infty} := \sum_{j' \in \cJ'}J_{\sym, \infty}^{(j', j' + f)}$, and the ideals $J_{\sym,\infty}^{(j', j' + f)}$ are described in Table \ref{Table2}.  By construction, the Frobenius pullback $(\sigma^f)^*\cM_{\overline{S}'/J_{\sym,\infty}}$ and the dual $\cM_{\overline{S}'/J_{\sym,\infty}}^\vee$ are isomorphic; we fix such an isomorphism
$$\iota:(\sigma^f)^*\cM_{\overline{S}'/J_{\sym,\infty}} \stackrel{\sim}{\longrightarrow} \cM_{\overline{S}'/J_{\sym,\infty}}^\vee.$$  Applying $\bbV^*_{K_2}$ to $\cM_{\overline{S}'/J_{\sym,\infty}}$, we obtain a homomorphism (using the basis $\gamma \otimes 1$ of $\bbV^*_{K_2}(\cM_{\overline{S}'/J_{\sym,\infty}})$)
$$\rho'_{\overline{S}'/J_{\sym,\infty}}:\Gamma_{K_{2,\infty}} \longrightarrow \nG\nL_2(\overline{S}'/J_{\sym,\infty}),$$
with a compatible polarization $\bbV^*_{K_2}(\iota)$ (in the sense of \cite[\S~ 5C2]{koziolmorra}), with respect to the character $\varepsilon^3$.  Therefore, $\rho'_{\overline{S}'/J_{\sym,\infty}}$ extends to an $L$-parameter
$$\rho_{\overline{S}'/J_{\sym,\infty}}:\Gamma_{K_\infty} \longrightarrow {}^C\nU_{1,1}(\overline{S}'/J_{\sym, \infty})$$
which lifts $\rhobar(1)|_{\Gamma_{K_\infty}}$.  Therefore we get a morphism
$$R_{\rhobar(1)|_{\Gamma_{K_\infty}}}^\Box \longrightarrow \overline{S}'/J_{\sym,\infty},$$
which fits into a commutative diagram
\begin{center}
  \begin{tikzcd}
    R_{\BC(\rhobar(1))|_{\Gamma_{K_{2,\infty}}}}^\Box \ar[r,"\eqref{multitype:eqn1}"] \ar[d,"\BC"] &  \overline{S}' \ar[d, twoheadrightarrow] \\
    R_{\rhobar(1)|_{\Gamma_{K_\infty}}}^\Box \ar[r] & \overline{S}'/J_{\sym,\infty}.
  \end{tikzcd}
\end{center}
Letting $Y_1, \ldots, Y_4$ denote formal variables, we extend this diagram to
\begin{equation}
  \label{multitype:commdiag1}
  \begin{tikzcd}
    R_{\BC(\rhobar(1))|_{\Gamma_{K_{2,\infty}}}}^\Box \ar[r,"\eqref{multitype:eqn1}"] \ar[d, "\BC"] &  \overline{S}' \ar[d, twoheadrightarrow] \ar[r, hookrightarrow] &  \overline{S}'\llbracket Y_1, \ldots, Y_4\rrbracket \ar[d, twoheadrightarrow] \\
    R_{\rhobar(1)|_{\Gamma_{K_\infty}}}^\Box \ar[r] & \overline{S}'/J_{\sym,\infty} \ar[r, hookrightarrow] & (\overline{S}'/J_{\sym, \infty})\llbracket Y_1, \ldots, Y_4\rrbracket
  \end{tikzcd}
\end{equation}

By base extension, the $\Gamma_{K_{2,\infty}}$-representation $\bbV^*_{K_2}(\cM_{\overline{S}'\llbracket Y_1, \ldots, Y_4 \rrbracket})$ and the $\overline{S}'\llbracket Y_1, \ldots, Y_4 \rrbracket$-basis $(1 + \begin{pmatrix} Y_1 & Y_2 \\ Y_3 & Y_4\end{pmatrix})(\gamma \otimes 1)$ give rise to a morphism
$$\psi_0': R_{\BC(\rhobar(1))|_{\Gamma_{K_{2,\infty}}}}^\Box \longrightarrow  \overline{S}'\llbracket Y_1, \ldots, Y_4 \rrbracket.$$
Note that $\psi_0'$ is obtained from the top row of \eqref{multitype:commdiag1} by conjugating $\bbV^*_{K_2}(\cM_{\overline{S}'\llbracket Y_1, \ldots, Y_4 \rrbracket})$ by $1 + \begin{pmatrix} Y_1 & Y_2 \\ Y_3 & Y_4\end{pmatrix}$.  
Similarly conjugating the bottom row of \eqref{multitype:commdiag1}, we get the following commutative diagram:
\begin{center}
  \begin{tikzcd}
    R_{\BC(\rhobar(1))|_{\Gamma_{K_{2,\infty}}}}^\Box  \ar[r, "\psi_0'"] \ar[d,"\BC"] & \overline{S}'\llbracket Y_1, \ldots, Y_4 \rrbracket \ar[d, twoheadrightarrow]  \\ 
    R_{\rhobar(1)|_{\Gamma_{K_\infty}}}^\Box \ar[r,"\psi_0"] & \big(\overline{S}'/J_{\sym,\infty}\big)\llbracket Y_1, \ldots, Y_4 \rrbracket.
  \end{tikzcd}
\end{center}

We define a morphism
$$\psi': R_{\BC(\rhobar(1))|_{\Gamma_{K_{2,\infty}}}}^\Box\llbracket X_1, X_2, \ldots, X_{4f}\rrbracket \longrightarrow \overline{S}'\llbracket Y_1, \ldots, Y_4 \rrbracket$$
exactly as in the proof of \cite[Prop. 4.3.1]{BHHMS} (noting that for $\rhobar$ tamely ramified, $\BC(\rhobar(1))$ is always a direct sum of two characters).  In particular, the restriction of $\psi'$ to the base ring $R_{\BC(\rhobar(1))|_{\Gamma_{2,\infty}}}^\Box$ is $\psi'_0$.  Moreover, by Claim 1 in the cited proof, the map $\psi'$ is surjective.  We now define 
$$\psi_1: R_{\rhobar(1)|_{\Gamma_{K_{\infty}}}}^\Box\llbracket X_1, X_2, \ldots, X_{4f}\rrbracket \longrightarrow (\overline{S}'/J_{\sym,\infty})\llbracket Y_1, \ldots, Y_4 \rrbracket$$
by the conditions that $\psi_1(X_i)$ is equal to the image of $\psi'(X_i)$ in $(\overline{S}'/J_{\sym,\infty})\llbracket Y_1, \ldots, Y_4 \rrbracket$, and $\psi_1|_{R_{\rhobar(1)|_{\Gamma_{K_\infty}}}^\Box} = \psi_0$.  Therefore, we obtain a commutative diagram
\begin{center}
  \begin{tikzcd}
    R_{\BC(\rhobar(1))|_{\Gamma_{K_{2,\infty}}}}^\Box\llbracket X_1, X_2, \ldots, X_{4f}\rrbracket  \ar[r, "\psi'", twoheadrightarrow] \ar[d,"\BC"] & \overline{S}'\llbracket Y_1, \ldots, Y_4 \rrbracket \ar[d, twoheadrightarrow]  \\ 
    R_{\rhobar(1)|_{\Gamma_{K_\infty}}}^\Box\llbracket X_1, X_2, \ldots, X_{4f}\rrbracket \ar[r,"\psi_1", twoheadrightarrow] & \big(\overline{S}'/J_{\sym,\infty}\big)\llbracket Y_1, \ldots, Y_4 \rrbracket
  \end{tikzcd}
\end{center}
Note that in particular $\psi_1$ is surjective. We define 
\begin{equation}\label{defofrho}
  \psi: R_{\rhobar(1)|_{\Gamma_{K_\infty}}}^\Box\llbracket X_1, \ldots, X_{f - 1}, X_{f + 1},\ldots, X_{2f - 1}, X_{2f}, X_{4f}\rrbracket \longrightarrow \big(\overline{S}'/J_{\sym, \infty}\big)\llbracket Y_1, \ldots, Y_4\rrbracket
\end{equation} 
to be the restriction of $\psi_1$ to the subring on the left-hand side above.  

\begin{sublemma}
  \label{claim1}
  The map $\psi$ is surjective.  
\end{sublemma}

\begin{proof}[Proof of Sublemma \ref{claim1}]
In order to see that $\psi$ is surjective, we check that the induced map on reduced tangent vectors is injective.  Thus, suppose $t:(\overline{S}'/J_{\sym, \infty})\llbracket Y_1, \ldots, Y_4\rrbracket \longrightarrow \bbF[\varepsilon]/(\varepsilon^2)$ is a continuous homomorphism, and suppose $t\circ\psi = \psi^*t = \psi^*t_0 = t_0\circ \psi$, where $t_0$ is the ``zero vector'' sending each generator of the maximal ideal of $(\overline{S}'/J_{\sym, \infty})\llbracket Y_1, \ldots, Y_4\rrbracket$ to 0.  The images of these generators under $t$ will be denoted as follows:
$$t(b^{(j')}_{i,k}) = \varepsilon \mathsf{b}^{(j')}_{i,k},\quad t(c^{(j')}_{i,k}) = \varepsilon \mathsf{c}^{(j')}_{i,k},\quad t(d^{(j')}_{i,k}) = \varepsilon \mathsf{d}^{(j')}_{i,k},\quad t(x^{*(j')}_{i,k}) = \varepsilon \mathsf{x}^{(j')}_{i,k},\quad t(Y_i) = \varepsilon \mathsf{y}_i,$$
where the sans-serif letters lie in $\bbF$.  Our goal is to show that these sans-serif variables are 0 in $\bbF$.

As in \cite[pf. of Prop. 4.3.1, Claim 1]{BHHMS}, we have an isomorphism 
$$\lambda: \cM_{(\overline{S}'/J_{\sym,\infty})\llbracket \underline{Y}\rrbracket} \widehat{\otimes}_{(\overline{S}'/J_{\sym,\infty})\llbracket \underline{Y}\rrbracket, t} \bbF[\varepsilon]/(\varepsilon^2) \stackrel{\sim}{\longrightarrow} \cM_{(\overline{S}'/J_{\sym,\infty})\llbracket \underline{Y}\rrbracket} \widehat{\otimes}_{(\overline{S}'/J_{\sym,\infty})\llbracket \underline{Y}\rrbracket, t_0} \bbF[\varepsilon]/(\varepsilon^2)$$
such that the matrix of $\lambda$ is given by a $\cJ'$-tuple of matrices 
$$(1 + \varepsilon M_{j'})_{j' \in \cJ'} \in \nG\nL_2(\cO_{\cE, \bbF[\varepsilon]/\varepsilon^2})^{\cJ'}.$$
Exactly as in \cite{BHHMS}, from the isomorphism $\lambda$ we obtain that 
$$\sfc_{1,1}^{(j')} = \sfc_{1,2}^{(j')} = \sfc_{2,1}^{(j')} = \sfc_{2,2}^{(j')} = \sfb_{1,2}^{(j')} = \sfb_{2,1}^{(j')} = \sfd_{1,1}^{(j')} = \sfd_{2,2}^{(j')} = 0.$$
for all $j'\in\cJ'$.

Suppose now that $s_{0} = 1$ (so that $s_{j'} = 1$ for all $j' \in \cJ'$).  (The case $s_{0} = \fw$ is analogous, and left as a calculation for the reader.)  As in \cite{BHHMS}, we get the equalities 
\begin{eqnarray}
  \delta_{1,2}^{(0),-1}\sfx_{1,2}^{(0)} & = & ((M_{2f - 1})_{1,1} - (M_{0})_{1,1})|_{v = 0} \nonumber \\
  \delta_{2,1}^{(0),-1}\sfx_{2,1}^{(0)} & = & ((M_{2f - 1})_{2,2} - (M_{0})_{2,2})|_{v = 0} \nonumber \\
  \delta_{1,2}^{(1),-1}\sfx_{1,2}^{(1)} & = & ((M_{0})_{1,1} - (M_{1})_{1,1})|_{v = 0} \nonumber \\
  \delta_{2,1}^{(1),-1}\sfx_{2,1}^{(1)} & = & ((M_{0})_{2,2} - (M_{1})_{2,2})|_{v = 0} \label{Mcoeffs}\\
  & \vdots & \nonumber \\
  \delta_{1,2}^{(2f - 1),-1}\sfx_{1,2}^{(2f - 1)} & = & ((M_{2f - 2})_{1,1} - (M_{2f - 1})_{1,1})|_{v = 0} \nonumber\\
  \delta_{2,1}^{(2f - 1),-1}\sfx_{2,1}^{(2f - 1)} & = & ((M_{2f - 2})_{2,2} - (M_{2f - 1})_{2,2})|_{v = 0} \nonumber
\end{eqnarray}
From the assumption $\psi^*t = \psi^*t_0$ and the definition of the map $\psi$ in \cite{BHHMS}, we obtain $\sfx_{1,2}^{(j')} = 0$ for $j' \neq f - 1, 2f - 1$.  Thus, the above equations give the following chains of equalities:
\begin{equation}
  \label{multitype:chain1}
  (M_{2f - 1})_{1,1}|_{v = 0} = (M_{0})_{1,1}|_{v = 0} = (M_{1})_{1,1}|_{v = 0} = \ldots = (M_{f - 2})_{1,1}|_{v = 0},
\end{equation}
\begin{equation}
  \label{multitype:chain2}
  (M_{f - 1})_{1,1}|_{v = 0} = (M_{f})_{1,1}|_{v = 0} = (M_{f + 1})_{1,1}|_{v = 0} = \ldots = (M_{2f - 2})_{1,1}|_{v = 0}.
\end{equation}
Now, suppose $j' \neq f - 1, 2f - 1$.  In $\overline{S}'/J_{\sym,\infty}$, we have the relation 
$$\left(x_{1,2}^{*(j')} + \delta_{1,2}^{(j')}\right)\left(x_{2,1}^{*(j' + f)} + \delta_{2,1}^{(j' + f)}\right) - 1 = 0.$$
Applying $t$ gives
$$\delta_{1,2}^{(j')}\left(\sfx_{2,1}^{(j' + f)} + \delta_{2,1}^{(j' + f)}\right) - 1 = 0.$$
Since $\delta_{1,2}^{(j')}\cdot \delta_{2,1}^{(j' + f)} = 1$, we conclude that $\sfx_{2,1}^{(j')} = 0$ for $j' \neq f - 1, 2f - 1$.  Substituting this into the equations \eqref{Mcoeffs} above, we get the following chains of equalities:
\begin{equation}
  \label{multitype:chain3}
  (M_{2f - 1})_{2,2}|_{v = 0} = (M_{0})_{2,2}|_{v = 0} = (M_{1})_{2,2}|_{v = 0} = \ldots = (M_{f - 2})_{2,2}|_{v = 0},
\end{equation}
\begin{equation}
  \label{multitype:chain4}
  (M_{f - 1})_{2,2}|_{v = 0} = (M_{f})_{2,2}|_{v = 0} = (M_{f + 1})_{2,2}|_{v = 0} = \ldots = (M_{2f - 2})_{2,2}|_{v = 0}.
\end{equation}

There is one more piece of structure which we exploit.  Note that $\cM_{(\overline{S}'/J_{\sym,\infty})\llbracket \underline{Y}\rrbracket} \widehat{\otimes}_{(\overline{S}'/J_{\sym,\infty})\llbracket \underline{Y}\rrbracket, t} \bbF[\varepsilon]/(\varepsilon^2)$ and $\cM_{(\overline{S}'/J_{\sym,\infty})\llbracket \underline{Y}\rrbracket} \widehat{\otimes}_{(\overline{S}'/J_{\sym,\infty})\llbracket \underline{Y}\rrbracket, t_0} \bbF[\varepsilon]/(\varepsilon^2)$ are polarized $\varphi$-modules, and that $\lambda$ can be chosen to respect the polarizations.  In  terms of the matrices $1 + \varepsilon M_{j'}$, this means that
$$(1 + \varepsilon M_{j'})^{-\top}\dot{\fw}(1 + \varepsilon M_{j' + f})\dot{\fw}^{-1} = 1,$$
or, equivalently,
\begin{equation}
  \label{conjselfdualMmatrix}
  M_{j'}^{\top} = \dot{\fw}M_{j' + f}\dot{\fw}^{-1}.
\end{equation}
In particular, this implies $(M_{2f - 1})_{1,1}|_{v = 0} = (M_{f - 1})_{2,2}|_{v = 0}$ and $(M_{2f - 1})_{2,2}|_{v = 0} = (M_{f - 1})_{1,1}|_{v = 0}$.  We denote the common value of \eqref{multitype:chain1} and \eqref{multitype:chain4} by $\sfa$ and the common value of \eqref{multitype:chain2} and \eqref{multitype:chain3} by $\sfb$.  Thus, the equations \eqref{Mcoeffs} for $j' = f - 1$ and $2f - 1$ give 
\begin{equation}
  \label{emb-f-1-reln1}
  \delta_{1,2}^{(f - 1),-1}\sfx_{1,2}^{(f - 1)}  =  \sfa - \sfb = \delta_{2,1}^{(2f - 1),-1}\sfx_{2,1}^{(2f - 1)}.
\end{equation}
On the other hand, in $\overline{S}'/J_{\sym,\infty}$, we have the relation 
$$x_{1,2}^{*(f - 1)}x_{2,1}^{*(2f - 1)} + \delta_{2,1}^{(2f - 1)}x_{1,2}^{*(f - 1)} + \delta_{1,2}^{(f - 1)}x_{2,1}^{*(2f - 1)} = \left(x_{1,2}^{*(f - 1)} + \delta_{1,2}^{(f - 1)}\right)\left(x_{2,1}^{*(2f - 1)} + \delta_{2,1}^{(2f - 1)}\right) + 1 = 0$$
(note the sign change!).  Applying $t$ gives
$$\delta_{2,1}^{(2f - 1)}\sfx_{1,2}^{(f - 1)} + \delta_{1,2}^{(f - 1)}\sfx_{2,1}^{(2f - 1)} = 0.$$
Combining this with equation \eqref{emb-f-1-reln1} gives $\sfx_{1,2}^{(f - 1)} = \sfx_{2,1}^{(2f - 1)} = 0$, and analogously we obtain $\sfx_{1,2}^{(2f - 1)} = \sfx_{2,1}^{(f - 1)} = 0$.

It remains to show that $\sfy_i = 0$.  By definition of the map $\psi$, we have 
$$0 = t_0(Y_1) = (t_0 \circ \psi)(X_{2f}) = (t\circ\psi)(X_{2f}) = t(Y_1) = \varepsilon \sfy_1,$$
$$0 = t_0(Y_4) = (t_0 \circ \psi)(X_{4f}) = (t\circ\psi)(X_{4f}) = t(Y_4) = \varepsilon \sfy_4.$$
The fact that $\sfy_2 = \sfy_3 = 0$ follows exactly as in \cite{BHHMS}, which finally shows that $t = t_0$.
\end{proof}

We now examine the map $\psi_0$.

\begin{sublemma}
  \label{claim2}
  The map $\psi_0:R_{\rhobar(1)|_{\Gamma_{K_\infty}}}^\Box \longrightarrow (\overline{S}'/J_{\sym,\infty})\llbracket Y_1, Y_2, Y_3, Y_4 \rrbracket$ factors through $R_{\rhobar(1)|_{\Gamma_{K_\infty}}}^\Box \longtwoheadrightarrow R_{\rhobar(1)}^{\leq (3,0),\sigma}$, where the latter ring satisfies the same conditions as $R^{\leq (2,-1),\sigma}_{\rhobar}$, except the Hodge--Tate weights are $\leq (3,0)$, and the multiplier is $\varepsilon^3$.
\end{sublemma}

\begin{proof}[Proof of Sublemma \ref{claim2}]
  This follows exactly as in \cite[pf. of Claim 2, Prop. 4.3.1]{BHHMS}.  Namely, let $x$ denote a closed point of $\Spec((\overline{S}'/J_{\sym,\infty})\llbracket Y_1, Y_2, Y_3, Y_4 \rrbracket[1/p])$ with associated maximal ideal $\fp_x$, and choose $\tw' \in X(\sigma)$ such that $\sum_{j' \in \cJ'}(I_{\tw'}^{(j')} + I_{\tw'}^{(j' + f)} + J_{\sym,\infty}^{(j',j' + f)}) \subset \fp_x$.

  The \'etale $\varphi$-module $\cM_{(\overline{S}'/J_{\sym,\infty})\llbracket \underline{Y}\rrbracket} \widehat{\otimes}_{(\overline{S}'/J_{\sym,\infty})\llbracket \underline{Y}\rrbracket, x} \kappa(x)$ has type $\tau'_{\tw'}$, using the same change-of-variables map as described in Figure 2 of \cite{BHHMS} (note that these change-of-variables maps identify the relevant symmetry ideals).  Thus, the proof of \cite[Prop. 4.2.1]{BHHMS} implies that
  $$\bbV^*_{K_2}\left(\cM_{(\overline{S}'/J_{\sym,\infty})\llbracket \underline{Y}\rrbracket} \widehat{\otimes}_{(\overline{S}'/J_{\sym,\infty})\llbracket \underline{Y}\rrbracket, x} \kappa(x)\right)$$
  is the restriction to $\Gamma_{K_{2,\infty}}$ of a potentially crystalline representation $\rho'_x$ of $\Gamma_{K_2}$ over $\kappa(x)$, with symmetric Hodge--Tate weights $\leq(3,0)$ and inertial type $\tau'_{\tw'}$.  In this way, we obtain a framed deformation of $\BC(\rhobar(1))$.

  Furthermore, the polarization $\iota$ on $\cM_{\overline{S}'/J_{\sym,\infty}}$ induces a polarization $\bbV^*_{K_2}(\iota_x)$  on $\rho'_x|_{\Gamma_{K_{2,\infty}}}$.  
This polarization extends to a polarization of $\rho'_x$ itself (using \cite[Thm. 5.12]{LLLM}; compare equation \eqref{eq:def:ring}), which implies that $\rho'_x$ may be further extended to a morphism
  $$\rho_x:\Gamma_K \longrightarrow {}^C\nU_{1,1}(\kappa(x))$$
  lifting $\rhobar(1)$.  This gives the desired result.
\end{proof}

\begin{sublemma}
  \label{claim3}
  The ring $\overline{S}'/J_{\sym,\infty}$ is reduced, $\cO$-flat, and has $4^f$ irreducible components, each of relative dimension $3f$ over $\cO$.
\end{sublemma}

\begin{proof}[Proof of Sublemma \ref{claim3}]
  This follows exactly as in \cite[pf. of Claim 3, Prop. 4.3.1]{BHHMS}: for each $\tw' \in X(\sigma)$, we have an isomorphism 
  $$S'\Big/\left(\sum_{j' \in \cJ}I_{\tw'}^{(j')} + I_{\tw'}^{(j' + f)} + J_{\sym,\infty}^{(j', j' + f)}\right) \cong R'\Big/\left(\sum_{j' \in \cJ'}I^{(j')} + I^{(j' + f)} + I_{\sym,\infty}^{(j', j' + f)}\right)$$
  (the latter ring relative to the inertial type $\tau'_{\tw'}$), and the geometry of the ring on the right-hand side was determined in the proof of Theorem \ref{thm:height+monodromy}.  We may proceed exactly as in \cite{BHHMS} in order to show that as $\tw'$ ranges over $X(\sigma)$, the ideals $\sum_{j' \in \cJ}I_{\tw'}^{(j')} + I_{\tw'}^{(j' + f)} + J_{\sym,\infty}^{(j', j' + f)}$ are pairwise coprime.  
\end{proof}

We now finish the proof of the theorem.  By renumbering the variables $X_i$, Sublemmas \ref{claim1} and \ref{claim2} imply that we have a surjection
\begin{equation}
  \label{surj-multi-type}
  R^{\leq (3,0),\sigma}_{\rhobar(1)}\llbracket X_1, \ldots, X_{2f}\rrbracket \longtwoheadrightarrow (\overline{S}'/J_{\sym,\infty})\llbracket Y_1, Y_2, Y_3, Y_4 \rrbracket.
\end{equation}
The ring $R^{\leq (3,0),\sigma}_{\rho(1)}$ is reduced, $\cO$-flat, and by \cite[Thm. 3.3.7]{bellovingee}, each irreducible component has relative dimension $f + 4$ over $\cO$.  By Theorem \ref{thm:height+monodromy}, it has $4^f$ irreducible components.  Combining these facts with Sublemma \ref{claim3}, we conclude that the surjection \eqref{surj-multi-type} is an isomorphism.  Finally, we may use Theorem \ref{thm:height+monodromy} in order to identify irreducible components.  
\end{proof}

To conclude, we record the statement of \cite[Lem. 4.3.2, Prop. 4.3.3]{BHHMS}, which we will use below.  For the proposition below, we define $i:\textnormal{Adm}'^\vee(t_{\eta'})^\vee \longrightarrow \{1,2,3\}^{2f}$ exactly as in the discussion prior to \cite[Prop. 4.3.1]{BHHMS}, and define ideals $\fq^{(j'),(1,0)}_1, \fq^{(j'),(1,0)}_2, \fq^{(j'),(1,0)}_3 \subset S^{(j')}$ as in \cite[Lem. 4.3.2]{BHHMS} (with the shapes and superscripts shifted by $(1,1)$).

\begin{propn}
\label{prop:p:in:inter}
   Maintain the same setup as in Theorem \ref{prop:multitype-def-ring}, and suppose $\tw' \in X(\sigma)$.  We have:
   \begin{itemize}
    \item if $\lambda_{2f - 1 - j'} = (1,0)$, then $\fq^{(j'),(1,0)}_{i(\tw')_{2f - 1 - j'}}\subset \fp^{\lambda}_{\tw'}$;
    \item if $\lambda_{2f - 1 - j'} = (1,0)$ for all $j'$, then $\sum_{j' = 0}^{2f - 1} \fq^{(j'),(1,0)}_{i(\tw')_{2f - 1 - j'}} = \fp^{\lambda}_{\tw'}$.
   \end{itemize}
Suppose moreover that $\tw'_{2f - 1 - j'} = \fw t_{(1,0)}$ for some fixed $j' \in \cJ'$.  Then:
   \begin{itemize}
    \item if $(\tw'_\sigma)_{2f - 1 - j'} = t_{(0,1)}$, then $p \in (\fq_1^{(j'),(1,0)} \cap \fq_2^{(j'),(1,0)}) + \fp^{(j'),(2,-1)}_{\tw'}$;
    \item if $(\tw'_\sigma)_{2f - 1 - j'} = t_{(1,0)}$, then $p \in (\fq_3^{(j'),(1,0)} \cap \fq_2^{(j'),(1,0)}) + \fp^{(j'),(2,-1)}_{\tw'}$.
   \end{itemize}
\end{propn}

\begin{table}[ht]
  \captionsetup{justification=centering}
  \caption[Foo content]{\textbf{Symmetry ideal -- single type} \\ \footnotesize{For each (symmetric) shape $\tw' \in \textnormal{Adm}'^\vee(t_\eta')^\sym$, we use the same numbering and notation for variables as in \cite[Tables 1--3]{BHHMS} (up to shifting the shape by $t_{(\underline{1},\underline{1})}$).}}
  \label{Table1}
  \centering
  \adjustbox{max width=\textwidth}{
  \begin{tabular}{| c | c |}
  \hline
  &\\
  Shape & $\begin{array}{c} I^{(j', j' + f)}_{\textnormal{sym},\infty} \subset R^{(j')} \widehat{\otimes}_{\cO} R^{(j' + f)}, \\ j' \neq f - 1, 2f - 1 \\ \textnormal{(for $j' = f - 1, 2f - 1$, switch $\pm$ to $\mp$)} \end{array}$ \\
  &\\
  \hline
  \hline
   & \\
   $\tw'_{2f - 1 - j'} = \tw'_{f - 1 - j'} = t_{(1,0)}$ & $\begin{array}{ll} 
    e_{1,1}^{*(j' + f)}d_{2,2}^{*(j')} - 1, & d_{2,2}^{*(j' + f)}e_{1,1}^{*(j')} - 1,\\
      d_{1,1}^{(j' + f)}e_{1,1}^{*(j')}d_{2,2}^{*(j')} - d_{1,1}^{(j')}, & c_{1,1}^{(j' + f)}e_{1,1}^{*(j')}d_{2,2}^{*(j')} - c_{1,1}^{(j')}, \\
      c_{1,2}^{(j' + f)}e_{1,1}^{*(j')}d_{2,2}^{*(j')} + c_{1,2}^{(j')}, & d_{2,1}^{(j' + f)}e_{1,1}^{*(j')}d_{2,2}^{*(j')} + d_{2,1}^{(j')}, \\ 
     c_{2,1}^{(j' + f)}e_{1,1}^{*(j')}d_{2,2}^{*(j')} + c_{2,1}^{(j')},  & c_{2,2}^{(j' + f)}e_{1,1}^{*(j')}d_{2,2}^{*(j')} - c_{2,2}^{(j')}
    \end{array}$ \\
    & \\
    \hline
    & \\ 
    $\tw'_{2f - 1 - j'} = \tw'_{f - 1 - j'} = \fw t_{(1,0)}$ &  
    $\begin{array}{ll} 
      d_{1,2}^{*(j' + f)}d_{2,1}^{*(j')} - 1, & d_{2,1}^{*(j' + f)}d_{1,2}^{*(j')} - 1, \\
      d_{1,1}^{(j' + f)}d_{1,2}^{*(j')}d_{2,1}^{*(j')} + d_{1,1}^{(j')}, & c_{1,1}^{(j' + f)}d_{1,2}^{*(j')}d_{2,1}^{*(j')} + c_{1,1}^{(j')}, \\
      c_{1,2}^{(j' + f)}d_{1,2}^{*(j')}d_{2,1}^{*(j')} - c_{1,2}^{(j')}, & c_{2,1}^{(j' + f)}d_{1,2}^{*(j')}d_{2,1}^{*(j')} - c_{2,1}^{(j')}, \\
      d_{2,2}^{(j' + f)}d_{1,2}^{*(j')}d_{2,1}^{*(j')} + d_{2,2}^{(j')}, & c_{2,2}^{(j' + f)}d_{1,2}^{*(j')}d_{2,1}^{*(j')} + c_{2,2}^{(j')}
      \end{array}$\\
    & \\
    \hline
    & \\
    $\tw'_{2f - 1 - j'} = \tw'_{f - 1 - j'} = t_{(0,1)}$ & 
    $\begin{array}{ll}
      d_{1,1}^{*(j' + f)}e_{2,2}^{*(j')} - 1, & e_{2,2}^{*(j' + f)}d_{1,1}^{*(j')} - 1,\\
      c_{1,1}^{(j' + f)}d_{1,1}^{*(j')}e_{2,2}^{*(j')} - c_{1,1}^{(j')}, & d_{1,2}^{(j' + f)}d_{1,1}^{*(j')}e_{2,2}^{*(j')} + d_{1,2}^{(j')}, \\
      c_{1,2}^{(j' + f)}d_{1,1}^{*(j')}e_{2,2}^{*(j')} + c_{1,2}^{(j')}, & c_{2,1}^{(j' + f)}d_{1,1}^{*(j')}e_{2,2}^{*(j')} + c_{2,1}^{(j')}, \\
      d_{2,2}^{(j' + f)}d_{1,1}^{*(j')}e_{2,2}^{*(j')} - d_{2,2}^{(j')}, & c_{2,2}^{(j' + f)}d_{1,1}^{*(j')}e_{2,2}^{*(j')} - c_{2,2}^{(j')}
      \end{array}$ \\
    & \\
    \hline
  \end{tabular}
  }
  \end{table}

  \begin{table}[ht]
    \captionsetup{justification=centering}
    \caption[Foo content]{\textbf{Symmetry ideal -- multi-type} \\ \footnotesize{For each (symmetric) shape $\tw'_\sigma \in \textnormal{Adm}'^\vee(t_\eta')^\sym$, we use the same numbering and notation for variables as in \cite[Tables 4--5]{BHHMS} (up to shifting the shape by $t_{(\underline{1},\underline{1})}$).}}
    \label{Table2}
    \centering
    \adjustbox{max width=\textwidth}{
    \begin{tabular}{| c | c |}
    \hline
    &\\
    Shape & $\begin{array}{c} J^{(j', j' + f)}_{\textnormal{sym},\infty} \subset S^{(j')} \widehat{\otimes}_{\cO} S^{(j' + f)}, \\ j' \neq f - 1, 2f - 1 \\ \textnormal{(for $j' = f - 1, 2f - 1$, switch $\pm$ to $\mp$)} \end{array}$ \\
    &\\
    \hline
    \hline
     & \\
     $(\tw'_\sigma)_{2f - 1 - j'} = (\tw'_\sigma)_{f - 1 - j'} = t_{(0,1)}$ & $\begin{array}{c} 
      \textnormal{same generators as those in Row 2 of Table \ref{Table1},} \\
      \textnormal{in addition to}\\
      b_{1,2}^{(j' + f)}d_{1,2}^{*(j')}d_{2,1}^{*(j')} - b_{1,2}^{(j')}
      \end{array}$ \\
      & \\
      \hline
      & \\ 
      $(\tw'_\sigma)_{2f - 1 - j'} = (\tw'_\sigma)_{f - 1 - j'} = t_{(1,0)}$ &  
      $\begin{array}{c} 
        \textnormal{same generators as those in Row 2 of Table \ref{Table1},} \\
        \textnormal{in addition to}\\  
        b_{2,1}^{(j' + f)}d_{1,2}^{*(j')}d_{2,1}^{*(j')} - b_{2,1}^{(j')}
        \end{array}$ \\
      & \\
      \hline
    \end{tabular}
    }
    \end{table}

\newpage

\section{Upper bound on Gelfand--Kirillov dimension}

\label{sec:BP}

Our next task will be to obtain an upper bound on the Gelfand--Kirillov dimension of certain smooth $\nU_{1,1}(K)$-representations.  We first define the representations $D_0(\rhobar)$ and $D_1(\rhobar)$ and analyze their properties.  

\subsection{Breuil--Pa\v{s}k\={u}nas diagrams for unitary groups}  
\label{subsec:BP:diagrams}
We fix throughout a tamely ramified $1$-generic $L$-parameter $\rhobar: \Gamma_K \longrightarrow {}^C\nU_{1,1}(\bbF)$ which satisfies $\widehat{\imath}\circ \rhobar = \omega$.  Recall from Subsection \ref{subsub:SW} the set $\nW^?(\rhobar)$ of predicted Serre weights of $\rhobar$.  Write $\rhobar|_{I_K} \cong \overline{\tau}(s,\mu + \eta)$ with $\mu\in X^*(\underline{T}_{\nU})$ being $1$-deep.  Then \cite[Prop. 4.6]{koziolmorra} and the equivalence of \cite[\S 4.1]{herzig:duke} (for $(\nu,\sigma) = (-\underline{\fw}(\eta),\underline{\fw})$) gives
\begin{equation}
\label{eq:tr:1}
\nW^?(\rhobar) = \JH\left( \overline{R_{\un{\fw}s}\big(\un{\fw}(\mu) + (p - 1)\eta\big)} \right) =  \JH\left(\overline{R_{\underline{\fw}s}\left(\mu + \alpha_{\underline{\fw}s} \right)}\right),
\end{equation}
where, for a given element $w\in\un{W}$, we set $\alpha_w := \sum_{\substack{j \in \cJ\\ w_j = \fw}} \alpha_j$.  Note that the character appearing on the right-hand side lies in $X^*(\underline{T}_{\nU})$.  In particular, we see that the elements of the set $\nW^?(\rhobar)$, which are \emph{a priori} representations of the group $\un{G}_0(\bbF_p) = H(k_K)$, descend to $\nU_{1,1}(k_K)$.  Note also that the condition of $1$-genericity implies $\nW^?(\rhobar)$ is multiplicity free of size $2^f$.  (We may alternatively deduce these results using Section \ref{appendix:EGC}.)

We first give the main definition which will be relevant in global settings.
\begin{defn}
  \label{def:D0}
\begin{enumerate}
\item \label{def:D0:1} We denote by $D_0(\rhobar)$ the unique (up to isomorphism) finite-dimensional representation of $\nU_{1,1}(k_K)$ which satisfies the following three conditions:
\begin{enumerate}
\item \label{def:D0:1a} $\soc_{\nU_{1,1}(k_K)}(D_0(\rhobar)) = \bigoplus_{\sigma \in \nW^?(\rhobar)} \sigma$;
\item \label{def:D0:1b} $[D_0(\rhobar):\sigma] = 1$ for all $\sigma \in \nW^?(\rhobar)$;
\item \label{def:D0:1c} $D_0(\rhobar)$ is maximal among all representations satisfying the above two points.
\end{enumerate}
We occasionally view $D_0(\rhobar)$ as a representation of $\sfK_{\nU} = \nU_{1,1}(\cO_K)$ by inflation.
\item \label{def:D0:2} We define $D_1(\rhobar) := D_0(\rhobar)^{I_{\nU,1}}$, where $I_{\nU,1}$ denotes the upper-triangular pro-$p$-Iwahori subgroup of $\sfK_\nU$.  We view $D_1(\rhobar)$ as a (semisimple) representation of $I_\nU$.  
\end{enumerate}
\end{defn}

The existence and uniqueness of $D_0(\rhobar)$ follows exactly as in the proof of \cite[Prop. 13.1]{BP}.  In particular, we have a decomposition
$$D_0(\rhobar) \cong \bigoplus_{\sigma \in \nW^?(\rhobar)} D_{0,\sigma}(\rhobar)$$
with $\soc_{\nU_{1,1}(k_K)}(D_{0,\sigma}(\rhobar)) \cong \sigma$.

\subsection{Transfer to $\textnormal{GL}_2$}
\label{subsec:transfertoGL2}

Maintain the setup of the previous section.  We let $\rhobar_{\nG}:\Gamma_K \longrightarrow \nG\nL_2(\bbF)$ denote any fixed choice of continuous semisimple Galois representation which satisfies
$$\rhobar_{\nG}|_{I_K} \cong \overline{\tau}_{\nG}(s,\mu + \eta_\nG),$$
where $\overline{\tau}_{\nG}$ denotes the inertial type relative to the dual group of $\textnormal{Res}_{\cO_K/\bbZ_p}(\nG\nL_{2/\cO_K})$, and where $\eta_\nG$ denotes the character of the diagonal maximal torus of that group which is equal to $(1,0)$ in each embedding.  Here, we are using the canonical identification between the Weyl group of $\textnormal{Res}_{\cO_K/\bbZ_p}(\nG\nL_{2/\cO_K})$ and $\underline{W}$, along with the (non-equivariant) isomorphism $\underline{\textnormal{res}}_\nG \circ \underline{\textnormal{sec}}_\nU: X^*(\underline{T}_\nU) \stackrel{\sim}{\longrightarrow} X^*(\underline{T}_\nG)$ which is the identity on the underlying sets.

We define $\nW^?(\rhobar_\nG)$, the set of Serre weights of $\rhobar_\nG$, as in \cite[Def. 9.2.5]{GHS}.  By \cite[Rmk. 4.7]{koziolmorra} (with $K_2$ replaced by $K$) we have
\begin{equation}
\label{eq:tr:2}
\nW^?(\rhobar_\nG) = \JH\left(\overline{R_{\nG,\underline{\fw}s}(\underline{\fw}(\mu) + (p - 1)\eta_\nG)}\right) = \JH\left(\overline{R_{\nG,\underline{\fw}s}(\mu + \alpha_{\underline{\fw}s})}\right)
\end{equation}
where we append a ``$\nG$'' to the notation for Deligne--Lusztig representations to remind ourselves that the group $\nG\nL_2(k_K)$ is acting.

We will now employ the general strategy of Section \ref{sec:transfer} in order to transfer results about $\nW^?(\rhobar_\nG)$ to the analogous results about $\nW^?(\rhobar)$.

\subsubsection{}  
\label{subsub:transfer:k}
We put ourselves in the context of Examples \ref{exs}\eqref{pt1} and \ref{exs}\eqref{pt2} above.  
Suppose we are given a representation $\pi$ of the finite group $\nU_{1,1}(k_K)$ over $\bbF$, which admits a central character given by $\mu|_{Z(\nU_{1,1}(k_K))}$.  We extend the action to the group $\nG\nU_{1,1}(k_K)$ by letting the center $Z(\nG\nU_{1,1}(k_K))$ act by $\underline{\textnormal{sec}}_\nU(\mu)|_{Z(\nG\nU_{1,1}(k_K))}$ (by decomposition \eqref{U11decomp:finite}). We then restrict this action to $\nG\nL_2(k_K)$ (by decomposition \eqref{GL2decomp}), and denote the resulting representation by $\pi_\nG$.  Note that $\pi_\nG$ has central character $\mu|_{Z(\nG\nL_2(k_K))}$.  We may also go in the other direction, and let $\kappa \longmapsto \kappa_\nU$ denote the construction of a $\nU_{1,1}(k_K)$-representation from a $\nG\nL_2(k_K)$-representation.  We have $(\pi_\nG)_\nU \cong \pi$ and $(\kappa_\nU)_\nG \cong \kappa$.

Applying the construction of the previous paragraph to $\overline{R_{\underline{\fw}s}(\mu + \alpha_{\underline{\fw}s})}^{\textnormal{ss}}$, we get
$$
\left(\overline{R_{\underline{\fw}s}(\mu + \alpha_{\underline{\fw}s})}^{\textnormal{ss}}\right)_\nG \cong \overline{R_{\nG,\underline{\fw}s}(\mu + \alpha_{\underline{\fw}s})}^{\textnormal{ss}}.
$$
In particular, we obtain from \eqref{eq:tr:1} and \eqref{eq:tr:2}  that
\begin{equation}
  \label{transfer:serrewts}
  \nW^?(\rhobar)_\nG = \nW^?(\rhobar_\nG).  
\end{equation}

\begin{lemma}
  \label{transfer:D0}
  Let $\rhobar$ be as in Subsection \ref{subsec:BP:diagrams}.  We then have an isomorphism of $\nG\nL_2(k_K)$-representations
  $$D_0(\rhobar)_\nG \cong D_0(\rhobar_\nG),$$
  where the latter is defined in \cite[Prop. 13.1]{BP}.
\end{lemma}

\begin{proof}
  We examine the $\nG\nL_2(k_K)$-representation $D_0(\rhobar)_\nG$.  Define $D_0(\rhobar)_{\nG\nU}$ to be the representation of $\nG\nU_{1,1}(k_K)$ obtained from $D_0(\rhobar)$ by letting the center act by $\underline{\textnormal{sec}}_\nU(\mu)|_{Z(\nG\nU_{1,1}(k_K))}$.  By definition, we have $D_0(\rhobar)_{\nG\nU}|_{\nG\nL_2(k_K)} = D_0(\rhobar)_\nG$.  We check several properties.
  \begin{itemize}
    \item By Lemma \ref{resprops}\eqref{resprops-4}, we have 
    \begin{eqnarray*}
      \soc_{\nG\nU_{1,1}(k_K)}\left(D_0(\rhobar)_{\nG\nU}\right)|_{\nU_{1,1}(k_K)} & \cong & \soc_{\nU_{1,1}(k_K)}\left(D_0(\rhobar)\right) = \bigoplus_{\sigma \in \nW^?(\rhobar)}\sigma \\
      \soc_{\nG\nU_{1,1}(k_K)}\left(D_0(\rhobar)_{\nG\nU}\right)|_{\nG\nL_2(k_K)} & \cong & \soc_{\nG\nL_2(k_K)}\left(D_0(\rhobar)_\nG\right).
    \end{eqnarray*}
    As the center of $\nG\nU_{1,1}(k_K)$ acts on $D_0(\rhobar)_{\nG\nU}$ by a character, the above isomorphisms imply
    \begin{equation}
      \label{transfer:socle}
      \bigoplus_{\sigma \in \nW^?(\rhobar)}\sigma_\nG \cong \soc_{\nG\nL_2(k_K)}(D_0(\rhobar)_\nG).
    \end{equation}
    \item By equation \eqref{transfer:serrewts}, any Serre weight of $\rhobar_\nG$ is of the form $\sigma_\nG$ for $\sigma \in \nW^?(\rhobar)$.  Applying Corollary \ref{multcor} twice, we obtain
    \begin{equation}
      \label{transfer:multwts}
      [D_0(\rhobar)_{\nG}:\sigma_\nG] = [D_0(\rhobar):\sigma] = 1
    \end{equation}
    for any $\sigma_\nG \in \nW^?(\rhobar_\nG)$.  
    \item Suppose $D' \supset D_0(\rhobar)_\nG$ is any finite-dimensional representation satisfying the properties \eqref{transfer:socle} and \eqref{transfer:multwts}.  By transferring to $\nU_{1,1}(k_K)$, we obtain a representation $D'_\nU \supset D_0(\rhobar)$ satisfying the first two items of Definition \ref{def:D0}.  By maximality, we obtain $D'_\nU = D_0(\rhobar)$, and by transferring back to $\nG\nL_2(k_K)$, we get $D' = D_0(\rhobar)_\nG$.  
  \end{itemize}
  Thus, we see that $D_0(\rhobar)_\nG$ satisfies points (i), (ii), and (iii) of \cite[Prop. 13.1]{BP}, and the result now follows from the uniqueness claim of the cited proposition.
\end{proof}

\subsubsection{}
\label{subsub:transfer:K}
We extend slightly the transfer procedure from above.  Suppose now that $\pi$ is a smooth representation of $\sfK_\nU = \nU_{1,1}(\cO_K)$ over $\bbF$ on which the center $Z(\sfK_\nU)$ acts by the character $\mu|_{Z(\nU_{1,1}(k_K))}$.  (In particular, the action of the center factors through $Z(\sfK_\nU)/Z(\sfK_\nU)_1 \cong Z(\nU_{1,1}(k_K))$.)  We extend the action to the group $\sfK_{\nG\nU} = \nG\nU_{1,1}(\cO_K)$ by letting the center $Z(\nG\nU_{1,1}(\cO_K))$ act by $\underline{\textnormal{sec}}_\nU(\mu)|_{Z(\nG\nU_{1,1}(k_K))}$ (using the decomposition \eqref{U11decomp}), and then restrict this action to $\sfK_\nG = \nG\nL_2(\cO_K)$ (using the decomposition \eqref{GL2decomp}).  We denote the resulting representation by $\pi_\nG$.  Note that $\pi_\nG$ has central character $\mu|_{Z(\nG\nL_2(k_K))}$.  As above, we may also perform this construction in the reverse direction to obtain a quasi-inverse $\kappa \longmapsto \kappa_\nU$ from representations of $\sfK_\nG$ to representations of $\sfK_\nU$.

\begin{lemma}
  \label{transfer:soc}
  Suppose $\pi$ is a smooth representation of $\sfK_\nU$ over $\bbF$ with central character given by $\mu|_{Z(\nU_{1,1}(k_K))}$.  Then we have an isomorphism of semisimple $\sfK_\nG$-representations
  $$\left(\soc_{\sfK_\nU}(\pi)\right)_\nG \cong \soc_{\sfK_\nG}(\pi_\nG).$$
\end{lemma}

\begin{proof}
  The $\sfK_\nG$-subrepresentation $(\soc_{\sfK_\nU}(\pi))_\nG$ of $\pi_\nG$ is semisimple, and therefore we obtain 
  $$(\soc_{\sfK_\nU}(\pi))_\nG \subset \soc_{\sfK_\nG}(\pi_\nG).$$  
  On the other hand, by symmetry we have
  $$(\soc_{\sfK_\nG}(\pi_\nG))_\nU \subset \soc_{\sfK_\nU}((\pi_\nG)_\nU) = \soc_{\sfK_\nU}(\pi).$$  
  Applying $\pi' \longmapsto \pi'_\nG$ gives the reverse inclusion of the first line.
\end{proof}

\begin{cor}
  \label{transfer:JH-K}
  Suppose $\pi$ is a smooth representation of $\sfK_\nU$ over $\bbF$ with central character given by $\mu|_{Z(\nU_{1,1}(k_K))}$.  Then we have
  $$\JH(\soc_{\sfK_\nU}(\pi))_\nG = \JH(\soc_{\sfK_\nG}(\pi_\nG)).$$
\end{cor}

\subsubsection{}  
\label{subsub:transfer:Lambda}
We now put ourselves in the context of Examples \ref{exs}\eqref{pt3} and \ref{exs}\eqref{pt4}.  Suppose $\pi$ is a smooth representation of $\sfK_\nU$ over $\bbF$ on which the center $Z(\sfK_\nU)$ acts by the character $\mu|_{Z(\nU_{1,1}(k_K))}$.  We examine the action of the Iwasawa algebra $\bbF\llbracket \sfK_\nU\rrbracket$ on $\pi$.  In particular, consider the subspace $\pi[\fm_{\nU}^2]$ of $\fm_{\nU}^2$-torsion.  This is a module over the quotient ring $\bbF\llbracket \sfK_\nU\rrbracket/\fm_\nU^2$.  By letting the center of $\sfK_{\nG\nU}$ act by $\underline{\sec}_{\nU}(\mu)|_{Z(\nG\nU_{1,1}(k_K))}$, we extend the $\bbF\llbracket \sfK_\nU\rrbracket/\fm_\nU^2$-action on $\pi[\fm_\nU^2]$ to $\bbF\llbracket \sfK_{\nG\nU}\rrbracket/\fm_{\nG\nU}^2$, and further restrict to obtain a module over $\bbF\llbracket \sfK_{\nG}\rrbracket/\fm_{\nG}^2$.  We denote this module by $\pi[\fm_{\nU}^2]_\nG$.  As above, we may perform this construction in the reverse direction to obtain a quasi-inverse $\kappa[\fm_\nG^2] \longmapsto \kappa[\fm_\nG^2]_\nU$, from $\bbF\llbracket \sfK_{\nG}\rrbracket/\fm_{\nG}^2$-modules to $\bbF\llbracket \sfK_{\nU}\rrbracket/\fm_{\nU}^2$-modules.

\begin{lemma}
  \label{K1square-tors}
  Suppose $\pi$ is a smooth representation of $\sfK_\nU$ over $\bbF$ on which the center $Z(\sfK_\nU)$ acts by the character $\mu|_{Z(\nU_{1,1}(k_K))}$.  Then we have an isomorphism of $\bbF\llbracket \sfK_\nG \rrbracket/\fm_\nG^2$-modules
  $$\pi[\fm_\nU^2]_\nG \cong \pi_\nG[\fm_\nG^2].$$
\end{lemma}

\begin{proof}
  Using the decomposition $\sfK_{\nJ,1} = \nS\nL_2(\cO_K)_1 \times Z_{\nJ,1}$ for $\nJ \in \{\nG, \nU, \nG\nU\}$, we see that the $\sfK_\nG$-subrepresentation $\pi[\fm_\nU^2]_\nG$ of $\pi_\nG$ is annihilated by $\fm_{\nG}^2$, which implies
  $$\pi[\fm_\nU^2]_\nG \subset \pi_\nG[\fm_\nG^2].$$
  On the other hand, by symmetry we have
  $$\pi_\nG[\fm_\nG^2]_\nU \subset (\pi_\nG)_\nU[\fm_\nU^2] = \pi[\fm_\nU^2].$$
  Applying $\pi' \longmapsto \pi'_\nG$ gives the reverse inclusion of the first line.
\end{proof}

\subsubsection{} 
\label{subsub:transfer:Iwahori}
We now consider Iwahori subgroups.  Let $\pi$ denote a smooth representation of $I_\nU$ over $\bbF$ on which the center $Z(I_\nU)$ acts by the character $\mu|_{Z(\nU_{1,1}(k_K))}$.  We use the ``Iwahori versions'' of the decompositions \eqref{GL2decomp} and \eqref{U11decomp} to extend the action of $I_\nU$ on $\pi$.  Precisely, we let $Z(I_{\nG\nU})$ act by $\underline{\sec}_\nU(\mu)|_{Z(\nG\nU_{1,1}(k_K))}$ to obtain a representation of $I_{\nG\nU}$, and then restrict the action to $I_\nG$.  We denote the resulting representation of $I_\nG$ by $\pi_\nG$.  The construction for the quasi-inverse is denoted $\kappa \longmapsto \kappa_\nU$.

\begin{lemma}
  \label{transfer:I1invts}
  Suppose $\pi$ is a smooth representation of $I_\nU$ over $\bbF$ on which the center $Z(I_\nU)$ acts by the character $\mu|_{Z(\nU_{1,1}(k_K))}$.  Then we have an isomorphism of $I_\nG$-representations
  $$(\pi^{I_{\nU,1}})_\nG \cong (\pi_\nG)^{I_{\nG,1}}.$$
\end{lemma}

\begin{proof}
  This follows exactly as in the proof of Lemma \ref{K1square-tors}, using the setting of Examples \ref{exs}\eqref{pt5} and \ref{exs}\eqref{pt6} (for $n = 1$).
\end{proof}

\begin{cor}
  \label{transfer:D1cor}
  We have an isomorphism of $I_\nG$-representations
  $$D_1(\rhobar)_\nG \cong D_1(\rhobar_\nG).$$
\end{cor}

\begin{proof}
  This follows from the previous lemma and Lemma \ref{transfer:D0}, using the definitions of $D_1(\rhobar)$ and $D_1(\rhobar_\nG)$.
\end{proof}

\subsection{The upper bound}  
\label{subsec:UB}
We now obtain the upper bound on Gelfand--Kirillov dimension.  We refer to \cite[\S 5]{BHHMS} for the relevant definitions, and denote by $\dim_{\nU_{1,1}(\cO_K)}(\pi)$ the Gelfand--Kirillov dimension of a smooth $\nU_{1,1}(\cO_K)$-representation $\pi$.

We maintain the setup of Subsection \ref{subsec:BP:diagrams}.

\begin{thm}
\label{thm:abstract:GKdim}
  Let $\rhobar:\Gamma_K \longrightarrow {}^C\nU_{1,1}(\bbF)$ be a tamely ramified $1$-generic $L$-parameter which satisfies $\widehat{\imath}\circ\rhobar = \omega$, and write $\rhobar|_{I_K} \cong \overline{\tau}(s,\mu + \eta)$ with $\mu \in X^*(\un{T}_\nU)$ being $1$-deep.  Suppose $\pi$ is an admissible smooth representation of $\sfK_\nU = \nU_{1,1}(\cO_K)$ over $\bbF$ with central
  character $\mu|_{Z(\nU_{1,1}(k_K))}$.  Assume that:
  \begin{enumerate}
    \item we have $\JH(\soc_{\sfK_\nU}(\pi)) = \nW^?(\rhobar)$ (up to multiplicity);
    \item for all $\sigma \in \nW^?(\rhobar)$, we have $[\pi[\fm_{\nU}^2]:\sigma] = [\soc_{\sfK_\nU}(\pi):\sigma]$;
    \item we have $\JH(\pi^{I_{\nU,1}}) = \JH(D_1(\rhobar))$ (up to multiplicity).
  \end{enumerate}
  We then have $\dim_{\nU_{1,1}(\cO_K)}(\pi) \leq f$.
\end{thm}

\begin{proof}
  Our strategy will be to transfer the relevant information to $\sfK_\nG = \nG\nL_2(\cO_K)$.  Let $\pi_\nG$ denote the transfer of $\pi$ (Subsection \ref{subsub:transfer:K}).  Thus, $\pi_\nG$ is an admissible smooth representation of $\sfK_\nG$ with central character $\mu|_{Z(\nG\nL_2(k_K))}$, and moreover we have $\pi|_{\nS\nL_2(\cO_K)} \cong \pi_\nG|_{\nS\nL_2(\cO_K)}$.

  We check several properties of $\pi_\nG$.
  \begin{itemize}
    \item Up to multiplicity, we have
    $$\JH(\soc_{\sfK_\nG}(\pi_\nG)) \stackrel{\textnormal{Cor. \ref{transfer:JH-K}}}{=} \JH(\soc_{\sfK_\nU}(\pi))_\nG = \nW^?(\rhobar)_\nG \stackrel{\textnormal{Eq. \eqref{transfer:serrewts}}}{=} \nW^?(\rhobar_\nG).$$
    \item Any Serre weight of $\nW^?(\rhobar_\nG)$ is of the form $\sigma_\nG$ for some $\sigma\in \nW^?(\rhobar)$.  Given $\sigma_\nG \in \nW^?(\rhobar_\nG)$, we have
    \begin{flushleft}
      $[\pi_\nG[\fm_{\nG}^2]:\sigma_\nG] \stackrel{\textnormal{Lem. \ref{K1square-tors}}}{=} [\pi[\fm_\nU^2]_\nG:\sigma_\nG] \stackrel{\textnormal{Cor. \ref{multcor}}}{=} [\pi[\fm_\nU^2]:\sigma]$
    \end{flushleft}
    \begin{flushright}
      $ = [\soc_{\sfK_\nU}(\pi):\sigma] \stackrel{\textnormal{Cor. \ref{multcor}}}{=} [\soc_{\sfK_\nU}(\pi)_\nG:\sigma_\nG] \stackrel{\textnormal{Lem. \ref{transfer:soc}}}{=} [\soc_{\sfK_\nG}(\pi_\nG): \sigma_\nG].$
    \end{flushright}
    
    \item Up to multiplicity, we have
    $$\JH((\pi_\nG)^{I_{\nG,1}}) \stackrel{\textnormal{Lem. \ref{transfer:I1invts}}}{=} \JH(\pi^{I_{\nU,1}})_\nG = \JH(D_1(\rhobar))_\nG \stackrel{\textnormal{Cor. \ref{transfer:D1cor}}}{=} \JH(D_1(\rhobar_\nG)).$$
  \end{itemize}
We therefore see that the representation $\pi_\nG$ satisfies the conditions (i), (ii), and (iii) of \cite[Thm. 6.4.7]{BHHMS}.  Note also that we may relax the hypothesis in \emph{op. cit.} from ``$\pi$ is an admissible smooth $\nG\nL_2(L)$-representation'' to ``$\pi$ is an admissible smooth $\nG\nL_2(\cO_L)$-representation'': in the proof of the cited theorem, one replaces the reference to Proposition 6.4.6 with \cite[Prop. 4.15]{huwang}, whose proof only depends on the structure of $\pi|_{\nG\nL_2(\cO_L)}$.  Thus, we conclude that $\dim_{\nG\nL_2(\cO_K)}(\pi_\nG) \leq f$.

  Finally, since we have decompositions $\sfK_{\nG,1} = \nS\nL_2(\cO_K)_1\times Z_{\nG,1}$ and $\sfK_{\nU,1} = \nS\nL_2(\cO_K)_1\times Z_{\nU,1}$, and since both $\pi$ and $\pi_\nG$ have the same restriction to $\nS\nL_2(\cO_K)$, by \cite[Lem. 5.1.2]{BHHMS} we get 
  \begin{flushleft}
    $\dim_{\nG\nL_2(\cO_K)}(\pi_\nG) = \dim_{\sfK_{\nG,1}}(\pi_\nG|_{\sfK_{\nG,1}}) = \dim_{\nS\nL_2(\cO_K)_1}(\pi_\nG|_{\nS\nL_2(\cO_K)_1})$  
  \end{flushleft}
  \begin{flushright}
    $= \dim_{\nS\nL_2(\cO_K)_1}(\pi|_{\nS\nL_2(\cO_K)_1}) = \dim_{\sfK_{\nU,1}}(\pi|_{\sfK_{\nU,1}}) = \dim_{\nU_{1,1}(\cO_K)}(\pi).$  
  \end{flushright}
\end{proof}

\section{Global applications}
\label{sec:global}

We now work in a global setting, and deduce our main results on Gelfand--Kirillov dimension of representations appearing in cohomology.

\subsection{Setup}
\label{global:setup}

\subsubsection{}
We work in the global setting of \cite[\S 6]{koziolmorra}.  Thus, we fix an imaginary CM field $F$ with maximal totally real subfield $F^+$.  We let $c$ denote the generator of $\Gal(F/F^+)$.  We suppose that $F^+/\bbQ$ is unramfied at $p$, that $F/F^+$ is unramified at all finite places, and that every place of $F^+$ above $p$ is inert in $F$.  We also assume that there is a place $v$ of $F^+$ above $p$ satisfying $F^+_v \cong K$.  The construction of \cite[\S 6A1]{koziolmorra} gives a reductive group $\bbG$ defined over $\cO_{F^+}$, the ring of integers of $F^+$, such that 
\begin{itemize}
  \item if $v$ is a place of $F^+$ which splits in $F$ as $v =ww^c$, then we have an isomorphism $\iota_w:\bbG(\cO_{F^+_v}) \stackrel{\sim}{\longrightarrow} \nG\nL_2(\cO_{F_w})$;
  \item if $v$ is a place of $F^+$ which is inert in $F$, then $\bbG(\cO_{F^+_v}) \cong \nU_{1,1}(\cO_{F^+_v})$;
  \item if $\kappa^+:F^+ \longhookrightarrow \bbR$ is an embedding, then $\bbG(F^+_{\kappa^+})$ is compact, and isomorphic to the compact unitary group $\nU_2(\bbR)$.
\end{itemize}

\subsubsection{}
We let $\bbA_{F^+}^\infty$ denote the finite ad\`eles of $F^+$, and let $\sfK = \prod_v \sfK_v$ denote a compact open subgroup of $\bbG(\bbA^\infty_{F^+})$.  Given such a $\sfK$, we set
$$\sfK_p := \prod_{v \in \Sigma_p^+} \sfK_v,\qquad \sfK^p := \prod_{v\not\in \Sigma_p^+} \sfK_v,$$
where $\Sigma_p^+$ denotes the set of places of $F^+$ lying above $p$. Suppose further that $\sfK_p \subset \bbG(\cO_{F^+,p}) = \prod_{v\in \Sigma_p^+} \bbG(\cO_{F^+_v})$ (where $\cO_{F^+,p}:=\cO_{F^+}\otimes_{\bbZ}\bbZ_p=\prod_{v\in \Sigma_p^+} \cO_{F^+_v}$), and let $V$ be an $\cO$-module endowed with an action of $\sfK_p$.  We define the space of algebraic automorphic forms on $\bbG(\bbA^\infty_{F^+})$ of level $\sfK$ with coefficients in $V$ as the $\cO$-module
$$S_\bbG(\sfK, V) := \left\{ f:\bbG(F^+)\backslash \bbG(\bbA^\infty_{F^+}) \longrightarrow V: f(gk) = k_p^{-1}\cdot (f(g)) ~ \textnormal{for all}~ g\in \bbG(\bbA^\infty_{F^+}), k\in \sfK \right\}$$
where we write $k_p$ for the projection of $k$ to $\sfK_p$.

\subsubsection{}
\label{subsub:Hecke}
We now consider Galois representations (using the group $\cG_2$ defined in Subsection \ref{def-of-Cgroup}).  

Fix a continuous representation $\rbar: \Gamma_{F^+} \longrightarrow \cG_2(\bbF)$ such that 
\begin{itemize}
  \item $\nu\circ \rbar = \omega^{-1}$;
  \item $\rbar^{-1}(\nG\nL_2(\bbF)\times \bbF^\times) = \Gamma_F$;
  \item $\BC'(\rbar)(\Gamma_F) \supset \nG\nL_2(\bbF_p)$;
\end{itemize}
We also let $\sfK = \prod_v \sfK_v$ be a sufficiently small compact open subgroup of $\bbG(\bbA^\infty_{F^+})$.  Recall that this means that for all $t \in \bbG(\bbA_{F^+}^{\infty})$, the finite group $t^{-1}\bbG(F^+)t \cap \sfK$ does not contain an element of order $p$.  Let $T$ denote a finite set of places of $F^+$, which contains all inert places $v$ for which $\sfK_v$ is not hyperspecial and all split places $v$ for which $\sfK_v \neq \bbG(\cO_{F^+_v})$.  We define $\bbT^T$ to be the universal Hecke algebra away from $T$, i.e., the commutative polynomial $\cO$-algebra generated by formal variables $T_w^{(i)}$ for $i = 1,2,$ and $w$ a place of $F$ split over $w|_{F^+}$ and satisfying $w|_{F^+} \not\in T$.

Suppose that $\rbar$ is unramified at all finite places $v$ of $F^+$ which split in $F$ and for which $v \not\in T$.  We denote by $\fm_{\rbar} \subset \bbT^T$ the maximal ideal associated to $\rbar$, defined by
$$\fm_\rbar := \left\langle \varpi,~ T_w^{(1)} - \textnormal{Tr}(\BC'(\rbar)(\textnormal{Frob}_w)),~ T_w^{(2)} - \nN(w)^{-1}\det(\BC'(\rbar)(\textnormal{Frob}_w))\right\rangle{}_{w|_{F^+} \not\in T},$$
where $w|_{F^+} = v \not\in T$ splits as $v = ww^c$ in $F$.

Given a $\sfK$-representation $V$ over $\cO$ as above, the universal Hecke algebra $\bbT^T$ acts on $S_{\bbG}(\sfK,V)$: the generator $T_w^{(i)}$ acts via the characteristic function of the double coset
$$\sfK_v \iota_w^{-1}\left(\begin{pmatrix} \varpi_w 1_i & \\ & 1_{2 - i}\end{pmatrix}\right)\sfK_v \cdot \sfK^v,$$
where $\varpi_w$ denotes a choice of uniformizer of $F_w$, and $v = w|_{F^+}$.

\subsubsection{}
\label{subsub:modularity}
We fix now a sufficiently small compact open subgroup $\sfK$ of $\bbG(\bbA^\infty_{F^+})$ such that $\sfK_v$ is hyperspecial whenever $v$ is a place of $F^+$ which is inert in $F$, and such that that $\sfK_p=\bbG(\cO_{F^+,p})$.  Let $\rbar$ and $T$ be as above and assume further that $\rbar$ is unramified at each split place of $F^+$ not in $T$.

Let $V$ be a Serre weight for $\bbG(\cO_{F^+,p})$, i.e., a representation of $\bbG(\cO_{F^+,p}) = \prod_{v\in \Sigma_p^+} \bbG(\cO_{F^+_v})$ of the form $\bigotimes_{v\in\Sigma_p^+,\bbF}\sigma_v$, where $\sigma_v$ is a Serre weight for $\bbG(\bbF_v^+)$, inflated to $\bbG(\cO_{F^+_v})$.  (Here $\bbF_v^+$ denotes the residue field at $v \in \Sigma_p^+$.)  We say that $\rbar$ is 
\begin{enumerate}
  \item \emph{modular of weight $V$ and level $\sfK$} if $S_{\bbG}(\sfK,V^\vee)_{\fm_{\rbar}}\neq 0$;
  \item \emph{modular of weight $V$} if there exists $\sfK$ and $T$ as above such that $\rbar$ is modular of weight $V$ and level $\sfK$;
  \item \emph{modular} if there exists a Serre weight $V$  for $\bbG(\cO_{F^+,p})$ such that $\rbar$ is modular of weight $V$.
\end{enumerate}

\subsection{Patching functors}
\label{subsec:PF}
\subsubsection{}
\label{subsub:local-conds}
We now define a patching functor, slightly expanding on the construction of \cite[\S 7A2]{koziolmorra}.  We assume from this point onwards that $\cO = W(\bbF)$, i.e., that $E$ is unramified over $\bbQ_p$.

Let $\rbar: \Gamma_{F^+} \longrightarrow \cG_2(\bbF)$ be a continuous representation as in Subsection \ref{global:setup}, and assume that $\rbar$ satisfies the following further conditions:
\begin{itemize}
  \item $\rbar$ is modular;
  \item $\rbar|_{\Gamma_{F^+_v}}$ is tamely ramified and 4-generic for all $v\in \Sigma_p^+$;
  \item $\rbar$ is unramified outside $\Sigma_p^+$;
  \item $\overline{F}^{\ker(\textnormal{ad}^0(\rbar))}$ does not contain $F(\zeta_p)$.
\end{itemize}
We remark that one can always find a totally real field $F^+$ and $\rbar$ satisfying these conditions: this follows from Corollary \ref{appendix:maincor}, noting that the globalization of $\varrhobar$ obtained in the cited result is modular in the sense of Subsection \ref{global:setup} above by the last three sentences of the proof of \cite[Prop. 7.2]{koziolmorra}.

\subsubsection{}
We begin by fixing some local conditions.  Let us choose an auxiliary prime $v_1$ of $F^+$ satisfying the conditions of \cite[\S 7A1]{koziolmorra}.  We note that $\tR^\Box_{v_1}$ (the maximal reduced and $\cO$-flat quotient of $R^\Box_{\rbar|_{\Gamma_{F^+_{v_1}}}}$ parametrizing lifts $\varrho$ of $\rbar|_{\Gamma_{F^+_{v_1}}}$ satisfying $\nu \circ \varrho = \varepsilon^{-1}$) is formally smooth over $\cO$ of relative dimension 4 by \cite[Lem. 2.5]{CEGGPS}.

Suppose now that $v'$ is an arbitrary element of $\Sigma_p^+$, and let $\cJ_{v'}=\Hom(\bbF^+_{v'}, \bbF)$.
Let $\lambda \in (\bbZ^2)^{\cJ_{v'}}$ denote an element satisfying $\lambda_{j,1} > \lambda_{j,2}$ for all $j \in \cJ_{v'}$, and fix a tame inertial type $\tau'$ such that $(\tau')^{\varphi^{[\bbF_{v'}^+:\bbF_p]}} \cong \tau'^\vee$.  We let $R_{v'}^{\Box,\lambda,\tau'}$ denote the maximal reduced and $\cO$-flat quotient of $R^{\Box}_{\rbar|_{\Gamma_{F^+_{v'}}}}$ which parametrizes potentially crystalline lifts $\varrho$ of $\rbar|_{\Gamma_{F^+_{v'}}}$ which satisfy
\begin{itemize}
  \item $\HT_j(\BC'(\varrho)) = \lambda_j - (1,1)$ for $0 \leq j \leq [\bbF_{v'}^+:\bbF_p] - 1$;
  \item $\HT_j(\BC'(\varrho)) = -\fw(\lambda_j)$ for $[\bbF_{v'}^+:\bbF_p] \leq j \leq 2[\bbF_{v'}^+:\bbF_p] - 1$;
  \item $\BC'(\varrho)$ has inertial type $\tau'$;
  \item $\nu\circ\varrho = \varepsilon^{-1}$.
\end{itemize}
Under the isomorphism $\cG_2 \cong {}^C\nU_{1,1}$ of \cite[\S 2D]{koziolmorra}, these lifts $\varrho$ are in bijection with ${}^C\nU_{1,1}$-valued $L$-parameters $\rho$ which satisfy
\begin{itemize}
  \item $\HT_j(\BC(\rho)) = \lambda_j$ for $0 \leq j \leq [\bbF_{v'}^+:\bbF_p] - 1$;
  \item $\HT_j(\BC(\rho)) = (1,1) - \fw(\lambda_j)$ for $[\bbF_{v'}^+:\bbF_p] \leq j \leq 2[\bbF_{v'}^+:\bbF_p] - 1$;
  \item $\BC(\rho)$ has inertial type $\tau'$;
  \item $\widehat{\imath}\circ\rho = \varepsilon$.
\end{itemize}
In particular, if $\lambda_j \in \{(1,0),~(2,-1)\}$ for all $j\in \cJ_{v'}$, then the ring $R^{\Box,\lambda,\tau'}_{v'}$ is an irreducible component of the deformation ring considered in Subsection \ref{subsec:single-type}.  (Note, however, that the ring $R^{\Box,\lambda,\tau'}_{v'}$ of the present paper would be denoted $R^{\Box,(\un{0},\un{1}) - \un{\fw}(\lambda),\tau'}_{v'}$ in the notation of \cite[\S 7A2]{koziolmorra}, due to normalizations of Hodge--Tate weights.)

Recall that we have fixed a place $v\in \Sigma_p^+$ satisfying $F^+_v \cong K$.  For each $v'\in \Sigma_p^+\smallsetminus \{v\}$, we fix a conjugate self-dual tame inertial type $\tau'_{v'}$ such that $\nW^?(\rbar|_{\Gamma_{F^+_{v'}}}) \cap \JH(\overline{\sigma(\tau'_{v'})})$ is a singleton, which we denote $\sigma_{v'}$ (here we are using the isomorphism ${}^C\nU_{1,1} \cong \cG_2$ in defining $\nW^?(\rbar|_{\Gamma_{F^+_{v'}}})$).  This choice is possible by \cite[Props. 3.16 and 4.6]{koziolmorra}, and the ring $R^{\Box,(\un{1},\un{0}),\tau'_{v'}}_{v'}$ is formally smooth over $\cO$ by \emph{op. cit}. \S 5C10 and Theorem 5.19.  In particular, given a Serre weight $\sigma$ for $\sfK_{\nU} = \nU_{1,1}(\cO_{K})$, we say that $\sigma$ is \emph{modular for $\rbar$} if the Serre weight $\sigma\otimes\bigotimes_{v'\in\Sigma_p^+\setminus\{v\}}\sigma_{v'}$ for $\bbG(\cO_{F^+,p})$ is modular for $\rbar$ in the sense of Subsection \ref{subsub:modularity}.

\subsubsection{}
We define $\sfK_m \subset \bbG(\bbA_{F^+}^\infty)$ exactly as in \cite[\S 7A4]{koziolmorra}.  In particular, $\sfK_{m,v_1}$ is an Iwahori subgroup, and if $v'\in \Sigma_p^+$, then $\sfK_{0,v'} \cong \nU_{1,1}(\cO_{F^+_{v'}})$.  Moreover, the subgroups $\sfK_m$ are sufficiently small for each $m \geq 0$.  We define $\fm_\rbar \subset \bbT^{\Sigma_p^+ \cup \{v_1\}}$ as above, and let $\fm'_\rbar \subset \bbT^{\Sigma_p^+}$ denote the ideal generated by $\fm_\rbar$ and the elements $T_{\tv_1}^{(1)} - \delta_1,~ T_{\tv_1}^{(2)} - \nN(v_1)^{-1}\delta_1\delta_2$, where $\delta_1,\delta_2$ are the distinct eigenvalues of $\BC'(\rbar)(\textnormal{Frob}_{\tv_1})$.

\begin{rmk}
\label{rmk:mod:m'}
We note that a Serre weight $\sigma$ is modular for $\rbar$ if and only if 
$$\Hom_{\sfK_0}\left(\sigma\otimes\bigotimes_{v'\in \Sigma_p^+\smallsetminus\{v\}} \sigma_{v'},~ S_{\mathbb{G}}(\sfK_0^p,\bbF)[\fm_{\rbar}']\right)\neq 0.$$
The ``only if'' direction is clear.  We prove the ``if'' direction.  We start by noting that if $\sigma$ is modular for $\rbar$, then by modifying slightly the argument of \cite[Prop. 7.2]{koziolmorra} (in particular, by choosing inertial types $\tau'_{v'}$ as above, and choosing $\tau'_\sigma$ such that $\nW^?(\rbar|_{\Gamma_{F^+_v}}) \cap \JH(\overline{\sigma(\tau'_\sigma)}) = \{\sigma\}$), we have 
$$\Hom_{\sfK_0}\left(\sigma\otimes\bigotimes_{v'\in \Sigma_p^+\smallsetminus\{v\}} \sigma_{v'},~ S_{\mathbb{G}}(\sfK_0^p,\bbF)[\fm_{\rbar}]\right)\neq 0.$$
Using \cite[Lem. 6.3, Thm. 6.1]{koziolmorra}, there exists a cuspidal automorphic representation $\pi$ of $\bbG(\bbA_{F^+})$ such that 
\begin{itemize}
  \item $\pi_\infty$ is trivial;
  \item $(\pi_v\otimes\bigotimes_{v'\in \Sigma_p^+\smallsetminus\{v\}}\pi_{v'})|_{\sfK_{0,p}}$ contains $(\sigma(\tau'_{\sigma})\otimes\bigotimes_{v'\in \Sigma_p^+\smallsetminus\{v\}}\sigma(\tau'_{v'}))\otimes_{E}\bbC$;
  \item $\pi_{v_1}$ satisfies $\pi_{v_1}^{\sfK_{0,v_1}} \neq 0$.
\end{itemize}
The condition $\pi_{v_1}^{\sfK_{0,v_1}}\neq 0$ implies that $\pi_{v_1}$ is a subquotient of an unramified principal series representation of $\bbG(F^+_{v_1}) \cong \nG\nL_2(F^+_{v_1})$ (see \cite[Prop. 14.3]{bushnellhenniart}).  
Following the conventions of \cite{koziolmorra}, by local/global compatibility (see \cite[Thm. 6.1(ii)]{koziolmorra}) and the assumption on the eigenvalues of $\BC'(\rbar)(\textnormal{Frob}_{\tv_1})$ we conclude that $\pi_{v_1}$ must be an unramified principal series.  Lemma 1.6(ii) of \cite{taylor:lfns} calculates the action of the Hecke operators $T_{\tv_1}^{(j)}$ on $\pi_{v_1}^{\sfK_{0,v_1}}$, which implies that $\pi$ contributes to the space
$S_\bbG(\sfK_0,(\sigma(\tau'_\sigma)\otimes \bigotimes_{v'\in \Sigma_p^+\smallsetminus\{v\}}\sigma(\tau'_{v'}))^\vee)[\fm'_\rbar]$.  Applying \cite[Lem. 6.3]{koziolmorra}, we conclude that
$$\Hom_{\sfK_0}\left(\sigma\otimes\bigotimes_{v'\in \Sigma_p^+\smallsetminus\{v\}} \sigma_{v'},~ S_{\mathbb{G}}(\sfK_0^p,\bbF)[\fm_{\rbar}']\right)\neq 0.$$
\end{rmk}

\subsubsection{}
\label{subsub:patching-functor}
Returning to the patching procedure, the reader can check that construction of \cite[\S 7C3]{koziolmorra} remains valid when replacing $\fm_{\rbar}$ with $\fm'_{\rbar}$.  Thus, we obtain a functor $M_\infty(-)$ from the category of finitely generated $\cO$-modules with a $\bbG(\cO_{F^+,p})$-action to the category of finitely generated $R_\infty$-modules, where
$$R_\infty := \left(\left(\widehat{\bigotimes}_{v'\in\Sigma_p^+}\tR_{v'}^{\Box}\right)\widehat{\otimes} \tR^{\Box}_{v_1}\right)[\![x_1,\dots,x_{q-[F^+:\bbQ]}]\!],$$
the completed tensor products taken over $W(\bbF) = \cO$.  Recall that $\widetilde{R}_{v'}^\Box$ denotes the maximal reduced and $p$-torsion-free quotient of the universal framed deformation ring parametrizing lifts $\rho$ of $\rbar|_{\Gamma_{F^+_{v'}}}$ satisfying $\nu\circ \rho = \varepsilon^{-1}$.

Next, we focus our attention on the fixed place $v \in \Sigma_p^+$.  Let $\tau'_{v'}$ and $\sigma_{v'}$ be as in Subsection \ref{subsub:local-conds}.
Define the ring $R_\infty^v$ by
$$R_\infty^v := \left(\left(\widehat{\bigotimes}_{v'\in\Sigma_p^+\smallsetminus\{ v\}}R_{v'}^{\Box,(\un{1},\un{0}),\tau'_{v'}}\right)\widehat{\otimes}\tR^{\Box}_{v}\widehat{\otimes}\tR^{\Box}_{v_1}\right)[\![x_1,\dots,x_{q-[F^+:\bbQ]}]\!]$$
and for each  $v'\in\Sigma_p^+\smallsetminus\{ v\}$ fix a $\cO$-lattice $\sigma(\tau'_{v'})^\circ$  inside $\sigma(\tau'_{v'})$.  
We then define an exact functor $M_\infty^{v}(-)$ from the category of finitely generated $\cO$-modules with an action of $\bbG(\cO_{F^+_v})\cong \sfK_\nU$ to finitely generated $R^v_\infty$-modules by 
$$M_\infty^{v}: V \longmapsto M_\infty\left(V\otimes_{\cO}\bigotimes_{v'\in\Sigma_p^+\smallsetminus\{v\}}\sigma(\tau'_{v'})^{\circ}\right).
$$

The following proposition records the properties of the functor $M_\infty^v(-)$ which we will use.  For use in the proposition (and in subsequent results), we introduce the representation
\begin{equation}
  \label{defofpi}
  \pi := \varinjlim_{\sfK_v'}\Hom_{\sfK_0^v/\sfK_1^v}\left(\bigotimes_{v'\in \Sigma_p^+\smallsetminus\{v\}} \sigma_{v'},~ S_{\bbG}(\sfK^v_1\sfK'_v,\bbF)[\fm'_\rbar]\right),
\end{equation}
where $\sfK'_v$ runs over compact open subgroups of $\bbG(F^+_v)$.  It an admissible smooth representation of $\bbG(F^+_v) \cong \nU_{1,1}(K)$ over $\bbF$.

\begin{propn}
\label{prop:PF}
In the above setup we have:
\begin{enumerate}
\item \label{prop:PF:1}
Suppose $\lambda \in (\bbZ^2)^{\cJ}$ satisfies $\lambda_{j,1} > \lambda_{j,2}$ for all $j \in \cJ$, and suppose $\tau'_v$ is a tame inertial type which satisfies $(\tau'_v)^{\varphi^{f}} \cong {\tau'_v}^\vee$.  Then the action of $R^v_\infty$ on $M_\infty^v(V((\un{0},\un{1}) - \un{\fw}(\lambda))^{\textnormal{d}}\otimes\sigma(\tau'_v)^\circ)$ factors through $R^{v}_{\lambda,\tau'_v,\infty} := R^v_\infty\widehat{\otimes}_{\tR_v^{\Box}} R_v^{\Box,\lambda,\tau'_v}$.  Moreover, if $M_\infty^v(V((\un{0},\un{1}) - \un{\fw}(\lambda))^{\textnormal{d}}\otimes\sigma(\tau'_v)^\circ) \neq 0$, then it is maximal Cohen--Macaulay over $R^{v}_{\lambda,\tau'_v,\infty}$, and the support of $M_\infty^v(V((\un{0},\un{1}) - \un{\fw}(\lambda))^{\textnormal{d}}\otimes \sigma(\tau'_v)^\circ)$ is a union of irreducible components of $R^v_{\lambda,\tau'_v,\infty}$.  (Here $V((\un{0},\un{1}) - \un{\fw}(\lambda))^{\textnormal{d}}$ denotes the Schikhof dual of $V((\un{0},\un{1}) - \un{\fw}(\lambda))$.)

\item \label{prop:PF:2}
Let $\lambda$ and $\tau'_v$ be as in the previous point.  If $M_\infty^v(V((\un{0},\un{1}) - \un{\fw}(\lambda))^{\textnormal{d}}\otimes\sigma(\tau'_v)^\circ)$ is nonzero, then the same is true of $M_\infty^v(\overline{V((\un{0},\un{1}) - \un{\fw}(\lambda))^{\textnormal{d}}\otimes\sigma(\tau'_v)^\circ})$.  In this case, $M_\infty^v(\overline{V((\un{0},\un{1}) - \un{\fw}(\lambda))^{\textnormal{d}}\otimes\sigma(\tau'_v)^\circ})$ is maximal Cohen--Macaulay over $\overline{R}^v_{\lambda,\tau'_v,\infty} := R^v_{\lambda,\tau'_v,\infty} \otimes_{\cO}\bbF$, and its support is a union of irreducible components of $\overline{R}^v_{\lambda,\tau'_v,\infty}$.  
\item \label{prop:PF:3}
Let $(R^{v}_{\lambda,\tau'_v,\infty})'$ denote the maximal quotient of $R^{v}_{\lambda,\tau'_v,\infty}$ which acts faithfully on $M_\infty^v(V((\un{0},\un{1}) - \un{\fw}(\lambda))^{\textnormal{d}}\otimes\sigma(\tau'_v)^\circ)$.  Then $M_\infty^v(V((\un{0},\un{1}) - \un{\fw}(\lambda))^{\textnormal{d}}\otimes\sigma(\tau'_v)^\circ)[1/p]$ is locally free of rank $1$ over $(R^{v}_{\lambda,\tau'_v,\infty})'[1/p]$.  
\item \label{prop:PF:4}
Let $\sigma$ be a Serre weight for $\sfK_\nU$ with highest weight $\mu\in X_1(\underline{T}_\nU) \subset (\bbZ^2)^{\cJ}$.  Then we have $M_\infty^v(\sigma)\neq 0$ if and only if $\sigma\in \nW^?(\rbar|_{F_v^+})$.  In this case, the action of $R^v_\infty$ on $M_\infty^v(\sigma)$ factors through $(R^v_\infty\widehat{\otimes}_{\tR_v^{\Box}} R_v^{\Box,(\un{1},\un{0}) - \un{\fw}(\mu),\mathbf{1}})\otimes_{\cO}\bbF$, and $M_\infty^v(\sigma)$ is maximal Cohen--Macaulay over the latter ring.
\item \label{prop:PF:5} Let $\pi$ be as in \eqref{defofpi}, and let $V$ denote a representation of $\sfK_\nU$ over $\cO$ which factors through a finite quotient of $\sfK_\nU$.  Then $M_\infty^v(V)/\fm_\infty \cong \Hom_{\sfK_\nU}(V,\pi)^\vee$, where $\fm_\infty$ denotes the maximal ideal of $R_\infty^v$.
\end{enumerate}
\end{propn}

\begin{proof}
  Everything except \eqref{prop:PF:2} follows from \cite[Prop. 7.3, Cor. 7.5]{koziolmorra} (after suitably adapting the arguments to incorporate the ideal $\fm_\rbar'$; in particular, this implies that ``locally free of rank 2'' of \textit{op. cit.} becomes  ``locally free of rank 1'' in the present setup).  To verify \eqref{prop:PF:2}, assume $M_\infty^v(V((\un{0},\un{1}) - \un{\fw}(\lambda))^\textnormal{d}\otimes\sigma(\tau'_v)^\circ)\neq 0$ and note that $p$ is a non-zerodivisor in the ring $R^v_{\lambda,\tau'_v,\infty}$ (by $\cO$-flatness), and also $p$ is a non-zerodivisor on $M_\infty^v(V((\un{0},\un{1}) - \un{\fw}(\lambda))^{\textnormal{d}}\otimes\sigma(\tau'_v)^\circ)$ (by \cite[Prop. 7.3(ii)]{koziolmorra}).  The latter reference also implies $M_\infty^v(\overline{V((\un{0},\un{1}) - \un{\fw}(\lambda))^{\textnormal{d}}\otimes \sigma(\tau'_v)^\circ}) \neq 0$, and applying \cite[Tags \stacks{090R}, \stacks{00KW}]{stacks} shows that maximal Cohen--Macaulayness descends to the mod $p$ quotient.
\end{proof}

\subsubsection{Component matching}

Our next task will be to match modular Serre weights with certain ideals in local deformation rings.

\begin{propn}
\label{prop:BM}
Let $\rhobar:\Gamma_K \longrightarrow {}^C\nU_{1,1}(\bbF)$ be a $4$-generic tamely ramified $L$-parameter  which satisfies $\widehat{\imath}\circ \rhobar = \omega$.
Then there is an injective map 
\begin{eqnarray*}
\nW^?(\rhobar) & \longrightarrow & \Spec(R_{\rhobar}^{\Box})\\
\sigma & \longmapsto & \fp(\sigma)
\end{eqnarray*}
which is uniquely characterized by the property that if $\tau'$ is a tame inertial type which satisfies $(\tau')^{\varphi^{f}}\cong \tau^{\prime\vee}$ then 
$$\Spec(\overline{R}_{\rhobar}^{\tau'})=\bigcup_{\sigma\in \nW^?(\rhobar)\cap \JH(\overline{\sigma(\tau')})}\Spec(R_{\rhobar}^\Box/\fp(\sigma))$$
is the decomposition of $\overline{R}^{\tau'}_{\rhobar}$ into irreducible components, where $\overline{R}_{\rhobar}^{\tau'}$ denotes the special fiber of the potentially crystalline deformation ring defined in \cite[\S 5C]{koziolmorra}.
\end{propn}
\begin{proof}
The proof closely follows the proof of \cite[Prop. 3.6.1(1)]{LLLM2}.

Firstly, if such an assignement exists, it must be unique: by \cite[Props. 3.16(ii), 4.6]{koziolmorra}, given $\sigma\in \nW^?(\rhobar)$ there exists a conjugate self-dual tame inertial type $\tau'_\sigma$ such that $\nW^?(\rhobar)\cap \JH(\overline{\sigma(\tau'_\sigma)}) = \{\sigma\}$, and hence $ \fp(\sigma) = \textnormal{Ann}_{R_{\rhobar}^{\Box}}(\overline{R}_{\rhobar}^{\tau'_\sigma})$.

We now prove that the assignment $\sigma \longmapsto \fp(\sigma)$ defined above satisfies the desired conditions.  By Corollary \ref{appendix:maincor}, there exists a modular globalization $\rbar$ of $\rhobar$ (or, more precisely, of $\varrhobar$, the composition of $\rhobar$ with the isomorphism $\cG_2 \cong {}^C\nU_{1,1}$), and hence a functor $M_\infty^v$ as in Proposition \ref{prop:PF}.  Given $\sigma\in \nW^?(\rhobar)$ and $\tau'_\sigma$ as in the first paragraph, we have
$$M_\infty^v(\overline{\sigma(\tau'_\sigma)^\circ}) = M^v_\infty(\sigma) \neq 0$$
by exactness of the functor $M_\infty^v$ and Proposition \ref{prop:PF}\eqref{prop:PF:4}.  Hence, by Proposition \ref{prop:PF}\eqref{prop:PF:2}, we have that $M^v_\infty(\sigma)$ is maximal Cohen--Macaulay over $\overline{R}^v_{(\un{1},\un{0}),\tau'_\sigma,\infty}$ and its support is a union of irreducible components of $\Spec(\overline{R}^v_{(\un{1},\un{0}),\tau'_\sigma,\infty})$.  Since the ring $\overline{R}^{v}_{(\un{1},\un{0}),\tau'_\sigma,\infty}$ is an integral domain (see \cite[\S 5C10; Table 3]{koziolmorra}), we see that this support is all of $\Spec(\overline{R}^v_{(\un{1},\un{0}),\tau'_\sigma,\infty})$.  Using reducedness, we conclude that 
$$\sqrt{\mathrm{Ann}_{R_\infty^v}(M^v_\infty(\sigma))} = \sqrt{\mathrm{Ann}_{R_\infty^v}(\overline{R}^v_{(\un{1},\un{0}),\tau'_\sigma,\infty})} = \mathrm{Ann}_{R_\infty^v}(\overline{R}^v_{(\un{1},\un{0}),\tau'_\sigma,\infty}) = \fp(\sigma)R_\infty^v.$$

Now let $\tau'$ denote a conjugate self-dual tame inertial type such that $M_\infty^v(\sigma(\tau')^\circ)\neq 0$ (note that $\tau'$ is then at least $2$-generic).  
Using Proposition \ref{prop:PF}\eqref{prop:PF:1}, the module $M_\infty^v(\sigma(\tau')^\circ)$ has full support over $R^v_{(\un{1},\un{0}),\tau',\infty}$, since this ring is an integral domain.  Applying \cite[Prop. 2.2.13]{emertongee}, we see that $M_\infty^v(\overline{\sigma(\tau')^\circ})$ has full support over $\overline{R}^v_{(\un{1},\un{0}),\tau',\infty}$.  Since $M_\infty^v(\overline{\sigma(\tau')^\circ})$ has a filtration with graded pieces isomorphic to $M^v_\infty(\sigma)$ ($\sigma\in \nW^?(\rhobar)\cap \JH(\overline{\sigma(\tau')})$), we obtain
\begin{eqnarray*}
\Spec(\overline{R}^v_{(\un{1},\un{0}),\tau',\infty}) & = &
\textnormal{Supp}(M_\infty^v(\overline{\sigma(\tau')^\circ}))\\
 & = & \bigcup_{\sigma\in \nW^?(\rhobar)\cap \JH(\overline{\sigma(\tau')})}\textnormal{Supp}(M_\infty^v(\sigma))\\
 & = & 
\bigcup_{\sigma\in \nW^?(\rhobar)\cap \JH(\overline{\sigma(\tau')})}\Spec(R_{\infty}^v/\fp(\sigma)R_{\infty}^v).
\end{eqnarray*}
The decomposition above descends to $\overline{R}^{\tau'}_{\rhobar}$ and $R^{\Box}_{\rhobar}$, which finishes the argument.
\end{proof}

Next, we put ourselves in the setting of \cite[\S 5C]{koziolmorra} (and refer to \textit{op.~cit.}~for the notation below).  By passing to special fibers, we get the following diagram, obtained in a manner analogous to \cite[Diag. (5.9)]{LLLM}:
\begin{equation}
\label{diag:f.s.prime}
\begin{tikzcd}
\textnormal{Spf}(\overline{R}^{\tau',\overline{\beta},\Box}_{\overline{\fM},\textnormal{pol};\rhobar})\ar[rr,"\textnormal{f.s.}"]
\ar[d,"\textnormal{f.s.}"'] \ar[drr,phantom,"\square"] & & \textnormal{Spf}(\overline{R}_{\tw'}^{\textnormal{expl}})\ar[d,"\textnormal{f.s.}"]\ar[r,"\sim"]\ar[dr,phantom,"\square"] & \overline{D}_{\overline{\fM},\textnormal{pol}}^{\tau',\overline{\beta}} \ar[d,"\textnormal{f.s.}"]
\\
\textnormal{Spf} (\overline{R}^{\tau'}_{\rhobar})=\textnormal{Spf}(\overline{R}^{\tau',\Box}_{\overline{\fM},\textnormal{pol};\rhobar})\ar[rr,"\textnormal{f.s.}"]\ar[d, hookrightarrow] \ar[drr,phantom,"\square"]& & \left[\textnormal{Spf}(\overline{R}_{\tw'}^{\textnormal{expl}})/\widehat{\bf{G}}_m^{2f}\right]\ar[d,hookrightarrow, "\imath_{\tau'}"]\ar[r,"\sim"]&
 \overline{Y}^{\mu,\tau'}_{\overline{\fM},\textnormal{pol}} \ar[dl, hookrightarrow,"\imath_{\tau'}'"]
\\
\Phi\text{-}\textnormal{Mod}^{\text{\'et},\Box}_{\overline{\cM},\textnormal{pol}}\ar[rr,"\textnormal{f.s.}"]
& &\Phi\text{-}\textnormal{Mod}^{\text{\'et}}_{\overline{\cM},\textnormal{pol}} & 
\end{tikzcd}
\end{equation}
where $\tw' := \tw'(\rhobar,\tau') \in \textnormal{Adm}'^\vee(t_{\eta'})^{\textnormal{sym}}$ is the shape of the Kisin module $\overline{\fM}$, $\overline{R}_{\tw'}^{\textnormal{expl}} := \widehat{\bigotimes}_{j\in\{f,\dots,2f-1\}}\overline{R}_{\tw_j}^{\textnormal{expl}}$ and  ``f.s.'' stands for a formally smooth morphism.  
(Our convention on shapes and dual shapes in the present paper is consistent with \cite{koziolmorra} and \cite{LLLM}, and is opposite to \cite{LLLM3}.)
Note that the isomorphism in the middle row follows form the discussion in \cite[\S 5C1]{koziolmorra}, and the injectivity of the diagonal map is a consequence of \emph{loc.~cit}.~Lemma 5.20.

Thus, given $\tau'$ we have bijections
\begin{eqnarray*}
\Big(  \Sigma_j\cap (\tw'(\rhobar,\tau')_j^{*})^{-1}(\Sigma_j)\Big)_{j\in\cJ} & \stackrel{\textnormal{Props. \ref{prop:SW:extgr:U2}, \ref{prop:JH:fct:extgr:U2}}}{\longleftrightarrow} & \nW^?(\rhobar)\cap \JH(\overline{\sigma(\tau')})\\
   & \stackrel{\textnormal{Prop. \ref{prop:BM}}}{\longleftrightarrow} & \textnormal{Irr}(\textnormal{Spec}(\overline{R}_{\rhobar}^{\tau'})) \\
   & \stackrel{\textnormal{Diag. \eqref{diag:f.s.prime}}}{\longleftrightarrow}&
  \textnormal{Irr}(\textnormal{Spec}(\overline{R}_{\tw'}^{\textnormal{expl}})).
\end{eqnarray*}
Here we use the notation $\Sigma_j := \{0,\rhobar_j\} \subset \Lambda_{\textnormal{wt}}$.

Following the procedure of \cite[\S 3.6.1]{LLLM2} we obtain:
\begin{propn}
\label{prop:cmpt:match}
The composition of the bijections above gives a bijection
\begin{eqnarray*}
\big(\Sigma_j\cap (\tw'(\rhobar,\tau')_j^*)^{-1}(\Sigma_j)\big)_{j\in\cJ} & \longrightarrow & \textnormal{Irr}(\textnormal{Spec}(\overline{R}_{\tw'}^{\textnormal{expl}}))
\\
\omega & \longmapsto & \langle z_{2f-1-j}(\omega)\rangle_{j\in\cJ}
\end{eqnarray*}
where $z_{2f-1-j}(\omega) = 0$ if $\Sigma_j \cap (\tw'(\rhobar,\tau')_j^{*})^{-1}(\Sigma_j)$ is a singleton, and otherwise $z_{2f-1-j}(\omega) = c_{2,2}$ if $\omega_j = 0$, $z_{2f-1-j}(\omega) = c_{1,1}$ if $\omega_j = \rhobar_j$.  Here, $c_{1,1}$ and $c_{2,2}$ denote coefficients appearing in the universal matrices of \cite[Table 2]{koziolmorra} with $i$ in \emph{loc.~cit}.~taken to be $2f-1-j$ here.
\end{propn}

\begin{proof}
  We record the relevant information appearing in Table \ref{Table:Match} justifying the proposition, for those embeddings $j$ where $\Sigma_j \cap (\tw'(\rhobar,\tau')_j^{*})^{-1}(\Sigma_j)$ has size 2.  The type $\tau'_{\omega}$ is characterized by the property $\nW^?(\rhobar)\cap\JH(\overline{\sigma(\tau'_\omega)})=\{F(\mathfrak{t}_{\mu}(s\omega))\}$.  The passage from the second to the third column is justified by Propositions \ref{prop:SW:extgr:U2} and \ref{prop:JH:fct:extgr:U2}.  The passage from the third to the fourth column is justified by the version of \cite[Lem. 3.6.3, Cor. 3.6.7]{LLLM2} for polarized Kisin modules, and the identities 
  \begin{eqnarray*}
    \begin{pmatrix} c_{1,1} & c_{1,2}^* \\ vc_{2,1}^* & c_{2,2}\end{pmatrix}~ \textnormal{modulo}~ c_{1,1} & \equiv & \begin{pmatrix} vc_{1,2}^* & 0 \\ vc_{2,2} & c_{2,1}^* \end{pmatrix} v^{(-1,1)}\fw \\ 
    \begin{pmatrix} c_{1,1} & c_{1,2}^* \\ vc_{2,1}^* & c_{2,2} \end{pmatrix}~ \textnormal{modulo}~ c_{2,2} & \equiv & \begin{pmatrix} c_{1,2}^* & c_{1,1} \\ 0 & vc_{2,1}^* \end{pmatrix}\fw.  
  \end{eqnarray*}
\end{proof}

\begin{table}[ht]
  \captionsetup{justification=centering}
  \caption[foo content]{Translation from $\Sigma\cap \tw'(\rhobar,\tau')^{-1}(\Sigma)$ to $\textnormal{Irr}(\textnormal{Spec}(\overline{R}_{\tw'}^{\textnormal{expl}}))$.}
\label{Table:Match}
\centering
\begin{tabular}{| c | c | c | c | }
\hline
& & &\\
$A^{(2f-1-j)}$ & $\omega_j\in \Sigma_j\cap (\tw'(\rhobar,\tau')_j^{*})^{-1}(\Sigma_j)$ & $(\tw'(\rhobar,\tau'_\omega)^*)_{2f-1-j}^{-1}\tw'(\rhobar,\tau')_{2f-1-j}^*$ & $z_{2f-1-j}(\omega)$\\
& & &\\
\hline\hline
& & &\\
 \multirow{4}{*}{$\begin{pmatrix}c_{1,1} & c_{1,2}^*\\ vc_{2,1}^* & c_{2,2}\end{pmatrix}$}&$\rhobar_j$&$t_{(-1,1)}\fw$&$c_{1,1}$\\
 & & &\\
 \cline{2-4}
 & & &\\
 & 0 & $\fw$ & $c_{2,2}$\\
 & & & \\
\hline
\end{tabular}
\end{table}

Recall from \cite[Def. 3.21]{koziolmorra} the notion of graph distance $\textnormal{dgr}(\sigma,\sigma')$ for a pair of regular Serre weights $\sigma,\sigma'$ with same central character; in the setup of Proposition \ref{prop:JH:fct:extgr:U2}, this quantity is equal to the number of embeddings $j$ for which $\omega_j \not\equiv \omega'_j$, where $\omega,\omega' \in \Sigma$ correspond to $\sigma,\sigma'$.
\begin{cor}
  \label{cor:diff-emb}
Suppose $\sigma,\sigma'\in \nW^?(\rhobar)\cap \JH(\overline{\sigma(\tau')})$ correspond to $\omega,\omega'\in \Sigma$, and assume $\textnormal{dgr}(\sigma,\sigma') = 1$.  
Then there exists a unique embedding $2f-1-j \in \cJ$ for which $z_{2f-1-j}(\omega)\neq z_{2f-1-j}(\omega')$.
\end{cor}

\subsection{Freeness of patched modules}
\label{sub:freeness}

We maintain the above setup.  Our next goal will be to show freeness of various patched modules $M_\infty^v(V)$.  The proofs are simpler than those of \cite{BHHMS}, since our patching functors have generic fibers which are locally free of rank 1.

Let $N\geq 4$ be an integer.  We fix throughout a continous representation $\rbar:\Gamma_{F^+} \longrightarrow \cG_2(\bbF)$ as in Subsection \ref{subsub:patching-functor}, and denote by $\rhobar:\Gamma_K \longrightarrow {}^C\nU_{1,1}(\bbF)$ the composition of $\rbar|_{\Gamma_{F^+_v}}$ with the isomorphism $\cG_2 \cong {}^C\nU_{1,1}$.  We suppose $\rhobar$ is $N$-generic, and let $M_\infty^v(-)$ denote a patching functor relative to $\rbar$.

\subsubsection{Freeness of Serre weights and types}

\begin{propn}
  Maintain the assumptions above.
\label{prop:freeness:1}
\begin{enumerate}
\item
\label{it:freeness:1:1}
Let $\sigma\in \nW^?(\rhobar)$.  Then $M_\infty^v(\sigma)$ is free of rank $1$ over $\overline{R}^v_{(1,0),\tau'_\sigma,\infty}$, where $\tau'_\sigma$ is a tame inertial type satisfying $(\tau'_\sigma)^{\varphi^{f}} \cong \tau'^{\vee}_\sigma$ and $\nW^?(\rhobar) \cap \JH(\overline{\sigma(\tau'_\sigma)}) = \{\sigma\}$.  

\item
\label{it:freeness:1:2}
Let $\tau'$ be a tame inertial type which satisfies $(\tau')^{\varphi^{f}}\cong \tau^{\prime\vee}$ and $\nW^?(\rhobar)\cap \JH(\overline{\sigma(\tau')})\neq\emptyset$, and let $\sigma(\tau')^\circ$ be a choice of a $\sfK_\nU$-stable $\cO$-lattice with irreducible cosocle.  Then $M_\infty^v(\sigma(\tau')^\circ)$ is free of rank $1$ over the integral domain $R_{(1,0),\tau',\infty}$.  
\end{enumerate}
\end{propn}

\begin{proof}
  \begin{enumerate}
    \item  Let $\tau'_\sigma$ be as in the statement of the proposition.  Using Proposition \ref{prop:PF}\eqref{prop:PF:4}, we have $M_\infty^v(\overline{\sigma(\tau'_\sigma)^\circ}) = M_\infty^v(\sigma) \neq 0$, and Proposition \ref{prop:PF}\eqref{prop:PF:1} then shows that $M_\infty^v(\sigma(\tau'_\sigma)^\circ)$ is maximal Cohen--Macaulay over $R^v_{(1,0),\tau'_\sigma,\infty}$.  By \cite[\S 5C10; Table 3]{koziolmorra}, the ring $R^v_{(1,0),\tau'_\sigma,\infty}$ is a regular local ring, and we therefore conclude that $M_\infty^v(\sigma(\tau'_\sigma)^\circ)$ is free over it (see \cite[Tag \stacks{00NT}]{stacks}).  The precise value of the rank follows from Proposition \ref{prop:PF}\eqref{prop:PF:3}, and we deduce the desired result by passing to mod $p$ quotients.

    \item When $|\nW^?(\rhobar)\cap \JH(\overline{\sigma(\tau')})| = 1$, the result was obtained during the proof of \eqref{it:freeness:1:1}.  We therefore assume $|\nW^?(\rhobar)\cap \JH(\overline{\sigma(\tau')})| > 1$.  
We first note, by the results of Section \ref{sec:transfer} and \cite[Thm. 5.1.1]{EGS}, that
\begin{enumerate}
\item 
\label{it:EGS:1}
if $\sigma, \sigma'\in \JH(\overline{\sigma(\tau')})$ are such that $\textnormal{dgr}(\sigma,\sigma') = 1$, then we have saturated inclusions $p\sigma(\tau')^\sigma \longhookrightarrow \sigma(\tau')^{\sigma'} \longhookrightarrow \sigma(\tau')^\sigma$;
\item 
\label{it:EGS:2}
if $\sigma\in \JH(\overline{\sigma(\tau')})$ we have $\sigma\cong \sigma(\tau')^\sigma/\sum_{\textnormal{dgr}(\sigma,\sigma')=1}\sigma(\tau')^{\sigma'}$.
\end{enumerate}
The notation $\sigma(\tau')^{\sigma}$ denotes the unique up to homothety $\cO$-lattice inside $\sigma(\tau')$ for which the cosocle of $\overline{\sigma(\tau')^{\sigma}}$ is $\sigma$ (and analogously for $\sigma'$).

Suppose that $\sigma,\sigma' \in \nW^?(\rhobar) \cap \JH(\overline{\sigma(\tau')})$.  Using \eqref{it:EGS:1} and Corollary \ref{cor:diff-emb}, we have a unique embedding $j \in \cJ$ for which $z_{2f - 1 - j}(\omega) \neq z_{2f - 1 - j}(\omega')$, where $\omega,\omega'$ are the elements of $\Sigma$ corresponding to $\sigma,\sigma'$.  In particular, we have $z_{2f - 1 - j}(\omega)z_{2f - 1 - j}(\omega') = pu$, where $u$ is a unit of $R^{\textnormal{expl}}_{\tw'(\rhobar,\tau')}$.  The argument in the second paragraph of \cite[pf. of Lem. 5.1.3]{LLLM2} gives $z_{2f - 1 - j}(\omega)M_\infty^v(\sigma(\tau')^{\sigma}) = M_\infty^v(\sigma(\tau')^{\sigma'})$.  (In the argument of \emph{op.~cit}., we note that $c_{1,1}$ and $c_{2,2}$ are not zero-divisors on $M_\infty^v(\sigma(\tau')^{\sigma})$ as the latter is $p$-torsion free; hence, given $i\in\{1,2\}$, the equality $c_{1,1}c_{2,2}M_\infty^v(\sigma(\tau')^{\sigma'}) = c_{i,i}M_\infty^v(\sigma(\tau')^{\sigma})$  is such that $c_{i,i}\neq z_{2f - 1 - j}(\omega)$ implies $z_{2f - 1 - j}(\omega)M_\infty^v(\sigma(\tau')^{\sigma}) = M_\infty^v(\sigma(\tau')^{\sigma'})$.)

We conclude by exactness of $M_\infty^v$ and \eqref{it:EGS:2} that
$$M_\infty^v(\sigma)\cong M_\infty^v(\sigma(\tau')^\sigma)/(z_{2f - 1 - j}(\omega))_{j\in\cJ}.$$
As $M_\infty^v(\sigma)$ is generated by $1$ element over $\overline{R}_{(1,0),\tau'_\sigma,\infty}$ by item \eqref{it:freeness:1:1}, the same is true of its mod $\fm_\infty$ fiber, and we conclude by Nakayama's lemma that $M_\infty^v(\sigma(\tau')^\sigma)$ is generated by $1$ element over $R_{(1,0),\tau',\infty}$.
As $M_\infty^v(\sigma(\tau')^\sigma)[1/p]$ is locally free of rank $1$ by Proposition \ref{prop:PF}\eqref{prop:PF:3}, and since $R_{(1,0),\tau',\infty}$ is $p$-torsion free, we conclude that $M_\infty^v(\sigma(\tau')^\sigma)$ is free of rank $1$ over $R_{(1,0),\tau',\infty}$.
\end{enumerate}
\end{proof}

\subsubsection{Freeness of projective envelopes}

For a Serre weight $\sigma \in \nW^?(\rhobar)$ we let 
$$P_\sigma := \textnormal{Proj}_{\bbF[\nU_{1,1}(k_K)]}(\sigma), \qquad \tP_\sigma := \textnormal{Proj}_{\cO[\nU_{1,1}(k_K)]}(\sigma)$$ 
denote the $\nU_{1,1}(k_K)$-projective envelope of $\sigma$ over $\bbF$ and $\cO = W(\bbF)$, respectively.

\begin{propn}
\label{prop:freeness:2}
Let $\sigma\in \nW^?(\rhobar)$.  Then $M^v_\infty(\tP_\sigma)$ is free of rank $1$ over $R^v_\infty/\bigcap_{\tau'}\fp_{\tau'}$, where the intersection runs over the tame inertial types $\tau'$ such that $(\tau')^{\varphi^{f}}\cong \tau'^{\vee}$ and $\sigma\in \JH(\overline{\sigma(\tau')})$, and $\fp_{\tau'}$ denotes the kernel of the map $R^v_\infty \longtwoheadrightarrow R_{(1,0),\tau',\infty}$.  Moreover, if $\pi$ denotes the representation defined in \eqref{defofpi}, we then have $\pi^{{\sfK_{\nU,1}}} \cong D_0(\rhobar)$ and hence $\pi^{{I_{\nU,1}}} \cong D_1(\rhobar)$.
\end{propn}

\begin{proof}
The proof closely follows that of \cite[Prop. 8.2.6]{BHHMS}, and we point out the main differences.
\begin{itemize}
    \item Equation (68) is replaced by Proposition \ref{prop:PF}\eqref{prop:PF:5}.  
    \item All references to Proposition 8.2.3 are replaced by Proposition \ref{prop:freeness:1}.
    \item Equation (72) is replaced by Proposition \ref{prop:PF}\eqref{prop:PF:3}.
    \item The $\nG\nL_2(k)$-representation $D_0(\rbar_v^\vee)$ is replaced by the $\nU_{1,1}(k_K)$-representation $D_{0}(\rhobar)$ from Section \ref{sec:BP} (and note that the second displayed equation in the proof of \emph{op.~cit}.~is guaranteed by Definition \ref{def:D0}\eqref{def:D0:1b}).
\end{itemize}
Thus, analogously to the proof of step (i) in \cite[Prop. 8.2.6]{BHHMS}, it suffices to show $\pi^{\sfK_{\nU,1}} \cong D_0(\rhobar)$.
\begin{itemize}
  \item By \cite[Cor. 7.5]{koziolmorra} and Definition \ref{def:D0} we have an inclusion $D_{0}(\rhobar) \longhookrightarrow \pi^{\sfK_{\nU,1}}.$  Indeed, the statement and proof of \cite[Lem. 9.2]{breuil-buzzati} goes through with $S^D_\psi(U,k_E)$ replaced by $S_{\mathbb{G}}(\sfK_{\nU,1}\sfK^v,\bbF)$, $\fm_{\rhobar^\vee}$ replaced by $\fm_\rbar'$, and its hypotheses are satisfied when $R = D_{0}(\rhobar)$ and $R' = \soc_{\sfK_{\nU}}(D_{0}(\rhobar)) = \bigoplus_{\sigma \in \nW^?(\rhobar)} \sigma$.  
  \item The reverse inclusion $\pi^{\sfK_{\nU,1}}\longhookrightarrow D_0(\rhobar)$ follows exactly as in \cite[Lem. 4.5, Prop. 4.6]{LMS}.  More precisely, we may apply the procedure of Subsection \ref{subsec:transfertoGL2} to transfer the relevant representation-theoretic results of \cite[\S\S 2,3]{LMS} to the group $\nU_{1,1}(k_K)$ (as long as we fix compatible central characters on all representations appearing).  In particular, Lemma \ref{resprops} guarantees that the socle and radical filtrations are preserved by this transfer procedure.
\end{itemize}

The above arguments prove step (i) of \cite[Prop. 8.2.6]{BHHMS} in our setup, and step (ii) follows analogously (invoking Proposition \ref{prop:freeness:1}).
\end{proof}

\subsubsection{Freeness of lattices}
\label{freenesslattices}

Next, we examine some locally algebraic representations.  
From now on, we assume $N\geq 12$.
Fix a Serre weight $\sigma \in \nW^?(\rhobar)$ and an embedding $j \in \cJ$, and consider the locally algebraic representation $R_{2,j} := V(\alpha_j)\otimes_{\cO}\tP_\sigma$ of $\sfK_{\nU}$ over $\cO$.  Here, $V(\alpha_j)$ denotes the representation $(\textnormal{Sym}^2(\cO^2)\otimes_\cO \det^{-1})^{(j)}$ of $\sfK_\nU$, where $\sfK_\nU$ acts via the embeddings $\sfK_\nU \subset \nG\nL_2(\cO_{K_2}) \stackrel{\sigma'_j}{\longhookrightarrow} \nG\nL_2(\cO)$ (see Subsection \ref{subsec:WM}). Note also that using the embedding $\sigma'_{j + f}$ gives an isomorphic representation of $\sfK_\nU$.  Moreover, the representation $R_{2,j}$ has a central character by Lemma \ref{integralres}\eqref{integralres-2}, and we may therefore apply the transfer procedure of Subsection \ref{subsub:transfer:K} to $R_{2,j}$ in order to produce a representation $(R_{2,j})_\nG$ of $\sfK_\nG = \nG\nL_2(\cO_K)$.  Using Lemma \ref{integralres}\eqref{integralres-3} (twice), we conclude that $(R_{2,j})_\nG$ is exactly the representation denoted $R_{2,j}$ in \cite[\S 8.3]{BHHMS}.  We may therefore proceed as in \emph{op. cit.}, and transfer the results back to $\sfK_\nU$.

The mod $p$ reduction of $R_{2,j}$ is given by
$$R_{2,j}/pR_{2,j} = \big(V(\alpha_j) \otimes_{\cO} \tP_\sigma\big)\otimes_{\cO}\bbF \cong P_{\sigma_{+,j}} \oplus P_{\sigma_{-,j}} \oplus P_\sigma,$$
where $\sigma_{\pm,j} := F(\lambda \pm \alpha_j)$, and $\lambda$ is the highest weight of $\sigma$.  Indeed, by extending central characters to $\sfK_{\nG\nU}$, we can apply the transfer procedure of Subsection \ref{subsub:transfer:K} and use Lemma \ref{resprops}\eqref{resprops-2}.  We can and do fix an embedding
$$\iota_j:P_\sigma \longhookrightarrow \big(V(\alpha_j) \otimes_{\cO} \tP_\sigma\big)\otimes_{\cO}\bbF$$
and define an $\cO$-lattice $R_{2,j}'$ in $V(\alpha_j)\otimes_{\cO}\tP_\sigma$ by the condition
$$R_{2,j}' := \{x\in R_{2,j}, x\otimes 1\in \iota_j(P_\sigma)\}.$$
Note that, by extending central characters appropriately, proceeding as in Subsection \ref{subsub:transfer:K} gives a lattice $(R_{2,j}')_{\nG}$ with a $\sfK_\nG$-action.  This is the lattice denoted by $R_{2,j}'$ in \cite[\S 7.3]{BHHMS}.  
Using \cite[Prop. 7.3.1]{BHHMS} and Lemma \ref{integralres}, the mod $p$ reduction of $R_{2,j}'$ sits in a short exact sequence of $\bbF[\sfK_\nU]$-modules
\begin{equation}
\label{eq:ses:lattice:}
0 \longrightarrow P_{\sigma_{+,j}}\oplus P_{\sigma_{-,j}} \longrightarrow R_{2,j}'/pR_{2,j}' \longrightarrow P_\sigma \longrightarrow 0.
\end{equation}

\begin{propn}
\label{prop:cyclic:lattice:1}
The $R_\infty^v$-module $M_\infty^v(R'_{2,j})$ is free of rank $1$ over $R_\infty^v/\bigcap_{\tau'}\fp_{\tau'}$, where the intersection runs over the conjugate self-dual tame inertial types $\tau'$ for which $\sigma \in \JH(\overline{\sigma(\tau')})$, and $\fp_{\tau'}$ denotes the kernel of the map $R_\infty^v \longtwoheadrightarrow R_{(2,-1)_j,\tau',\infty}$, where $(2,-1)_j$ denotes the character which is $(2,-1)$ in embedding $j$, and $(1,0)$ otherwise.
\end{propn}

\begin{proof}
The proof closely follows those of \cite[Lem. 8.3.2, Prop. 8.3.3, Thm. 8.3.4]{BHHMS} and we explain the differences in the proofs.  Firstly, by Lemma \ref{resprops} and equation \eqref{transfer:serrewts}, all references to \cite{BP} concerning Serre weight combinatorics are valid in our setup, and in particular we have
\begin{equation}
\label{eq:BP:comb}
\nW^?(\rhobar) \subset \JH(P_{\sigma}).
\end{equation}

The proof of \cite[Lem. 8.3.2]{BHHMS} is modified as follows:
\begin{itemize}
  \item All references to Propositions 8.2.3(i) and 8.2.6 are replaced by Propositions \ref{prop:freeness:1}\eqref{it:freeness:1:1} and \ref{prop:freeness:2}.  (Note that by Proposition \ref{prop:PF}, $M_\infty^v$ is exact, and $\nW^?(\rhobar)$ is precisely the set of Serre weights on which the functor does not vanish.)
  \item   The statement of \cite[Prop. 7.4.3]{BHHMS} remains valid when $K$ in \emph{loc.~cit}.~is replaced by $\sfK_\nU$.  
  \item By Lemma \ref{resprops}, Corollary \ref{multcor}, and equation \eqref{transfer:serrewts}, the statement of \cite[Lem. 8.3.1]{BHHMS} is valid in our setup.  
\end{itemize}
Hence, the proof of \cite[Lem. 8.3.2]{BHHMS} applies in our setting to show that for all $j$, the $R_\infty^v$-module $M_\infty^v(R_{2,j}')$ is generated by 1 element.

The proof of \cite[Prop. 8.3.3]{BHHMS} is modified as follows:
\begin{itemize}
  \item All references to Proposition 8.2.3 are replaced by Proposition \ref{prop:freeness:1}. 
  \item All references to Lemmas 2.4.2 and 2.4.3 are replaced by Propositions \ref{prop:SW:extgr:U2} and \ref{prop:JH:fct:extgr:U2}.
  \item All references to Lemma 2.4.6 are replaced by Proposition \ref{serre-wt-extns}.
  \item By Section \ref{sec:transfer}, the statements of \cite[Lem. 6.3.7, Prop. 6.3.10, Thm. 6.3.11]{BHHMS} remain valid in our setup.  Note in particular that Lemma \ref{resext} and Proposition \ref{serre-wt-extns} imply that the submodule structures of the representations of \textit{loc. ~cit.}~are preserved under transferring to $\sfK_\nU$.
  \item The references to Proposition 4.2.1 and Corollary 4.2.6 in \emph{loc.~cit}.~are replaced by Theorem \ref{thm:height+monodromy} and Corollary \ref{cor:single-type-cor}.
\end{itemize}
Hence, the proof of \cite[Prop. 8.3.3]{BHHMS} applies in our setting.

Finally, the proof of \cite[Thm. 8.3.4]{BHHMS} is modified by replacing references to \cite[Lem. 8.3.2, Prop. 8.3.3]{BHHMS} with the previous three paragraphs, and replacing the reference to Proposition 4.2.1 by Theorem \ref{thm:height+monodromy}.  This finishes the proof.
\end{proof}

We now inductively define a sequence of locally algebraic representations of $\sfK_\nU$ over $\cO = W(\F)$ which we will use to prove the main cyclicity statement.  Set $L_{-1} := \tP_\sigma$ and define inductively
$$L_j := L_{j-1}\times_{P_\sigma}R'_{2,j}.$$
Thus, we have 
$$L_{f - 1} = \tP_\sigma \times_{P_\sigma} R_{2,0}' \times_{P_\sigma} R_{2,1}' \times_{P_\sigma} \dots \times_{P_\sigma} R_{2,{f-1}}',$$
which transfers to the lattice denoted $R$ in \cite[\S 7.3]{BHHMS}.  By construction $L_j$ is a $\sfK_\nU$-stable lattice inside
$$\tP_\sigma[1/p] \oplus \left(\bigoplus_{j'=0}^{j}V(\alpha_{j'})\otimes_{\cO}\tP_\sigma\right)[1/p].$$

Further, suppose we are given a tame inertial type $\tau'$ satisfying $(\tau')^{\varphi^{f}}\cong \tau^{\prime\vee}$ and choose $\sigma \in \JH(\overline{\sigma(\tau')})$.  Then $\sigma(\tau')$ is a quotient of $\tP_\sigma[1/p]$, and we let $\sigma(\tau')^\circ$ denote the $\cO$-lattice inside $\sigma(\tau')$ defined as the image of $\tP_\sigma$.  We let $T'_{2,j}$ denote the image of the composite map
$$R'_{2,j} \longhookrightarrow V(\alpha_j) \otimes_{\cO}\tP_\sigma \longtwoheadrightarrow V(\alpha_j) \otimes_{\cO}\sigma(\tau')^\circ.$$
We then define
$$N_j := L_{j-1}\times_{Y_j}T'_{2,j}$$
for $0\leq j\leq f-1$, where $Y_j := T'_{2,j}/p(V(\alpha_j)\otimes_\cO \sigma(\tau')^\circ)$

\begin{rmk}
\label{rmk:trnsf:latt}
By Section \ref{sec:transfer:char0} and the transfer strategy of Subsection \ref{subsub:transfer:K}, the $\sfK_\nG$-stable lattices $(N_j)_{\nG}, (L_j)_\nG, (T'_{2,j})_\nG$ are exactly the lattices denoted as $N_j, L_j, T'_{2,j}$ in \cite[\S 8.3]{BHHMS}, and the $\sfK_\nG$-representation $(Y_j)_\nG$ is the representation denoted $Y_j$ in \emph{op.~cit.}.  Thus, by Lemmas \ref{resprops} and \ref{integralres} we check that the statements of \cite[Lems. 8.3.5, 8.3.6]{BHHMS} remain valid in our setting of $\sfK_\nU$-representations.  Moreover, using \eqref{eq:BP:comb}, we see that the statements of \cite[Prop. 7.4.3, Lem. 8.3.1]{BHHMS} hold true in our setting, which implies that the proof of \cite[Prop. 8.3.7]{BHHMS} also holds in our setup (note that in order to imitate the proof of \cite[Prop. 8.3.7]{BHHMS} in our context we also need to invoke Proposition \ref{prop:freeness:2} above).
\end{rmk}

We can finally use the results from Section \ref{sec:multi-type} to prove the following.

\begin{thm}
The $R_\infty^v$-module $M_\infty^v(L_j)$ is free of rank $1$ over $R_\infty^v/\bigcap_{\tau'}\fp_{\lambda,\tau'}$, where the intersection runs over the tame inertial types $\tau'$ and Hodge--Tate weights $\lambda = (\lambda_{j'})_{0 \leq j' \leq f - 1}$ such that:
\begin{itemize}
  \item $(\tau')^{\varphi^{f}}\cong \tau^{\prime\vee}$;
  \item $\sigma \in \JH(\overline{\sigma(\tau')})$;
  \item $\lambda_{j'} \in \{(1,0),~ (2,-1)\}$ if $j'\leq j$; and 
  \item $\lambda_{j'} = (1,0)$ if $j'> j$.
\end{itemize}
Further, $\fp_{\lambda,\tau'}$ denotes the prime ideal given by the kernel of the map $R_\infty^v \longtwoheadrightarrow R_{\lambda,\tau',\infty}$.
\end{thm}

\begin{proof}
The proof is a direct adaptation of \cite[Thm. 8.3.9]{BHHMS}, with the following modifications:
\begin{itemize}
  \item References to \cite[Props. 4.2.1, 4.3.1, 4.3.3, Lem. 4.3.2]{BHHMS} are replaced by Theorems \ref{thm:height+monodromy}, \ref{prop:multitype-def-ring} and Proposition \ref{prop:p:in:inter}.
  \item References to \cite[Props. 8.2.6, 8.3.3, Thm. 8.3.4, Lem. 8.3.5, Prop. 8.3.7]{BHHMS} are replaced by Proposition \ref{prop:freeness:2}, the proof of Proposition \ref{prop:cyclic:lattice:1}, and Remark \ref{rmk:trnsf:latt}.
\end{itemize}
Thus, the desired conclusion follows analogously to \cite{BHHMS}.
\end{proof}

Using the transfer procedure of Subsection \ref{subsec:transfertoGL2}, the mod $p$ reduction of $L_{f - 1}$ corresponds to the mod $p$ reduction of the lattice denoted $R$ in \cite[\S 7.3]{BHHMS} (hence to $(\textnormal{Proj}_{\sfK_\nG/Z(\sfK_\nG)_{1}}(\sigma_{\nG}))/\fm^2_{\nG}$ by Corollary 7.3.4 of \emph{loc.~cit}.).  Thus, we obtain:
\begin{cor}
\label{cor:mult:1}
Let $\sigma\in \nW^?(\rhobar)$.
The surjection
$$(\textnormal{Proj}_{\sfK_\nU/Z(\sfK_\nU)_{1}}(\sigma))/\fm^2_{\nU} \longtwoheadrightarrow \sigma$$
induces an isomorphism 
$$M_\infty^v\left((\textnormal{Proj}_{\sfK_\nU/Z(\sfK_\nU)_{1}}(\sigma))/\fm^2_{\nU}\right)/\fm_\infty \stackrel{\sim}{\longrightarrow} M_\infty^v(\sigma)/\fm_\infty.$$
\end{cor}

\begin{rmk}
The statements above hold true in the non-minimal setting, by replacing ``free of rank $1$'' (resp., ``generated by one element'') by ``free of rank $r$'' (resp., ``generated by $r$ elements''), and the representations $D_0(\rhobar)$, $D_1(\rhobar)$ in the statement of Proposition \ref{prop:freeness:2} by  $D_0(\rhobar)^{\oplus r}$, $D_1(\rhobar)^{\oplus r}$ respectively.
\end{rmk}

\subsection{Gelfand--Kirillov dimension}
We may now prove our main result on Gelfand--Kirillov dimensions.  For the relevant notions and definitions, we refer to \cite[\S 5.1]{BHHMS}.  We keep the global setup of Section \ref{global:setup}, and in particular, recall that $\pi$ is defined as 
$$\pi := \varinjlim_{\sfK_v'}\Hom_{\sfK_0^v/\sfK_1^v}\left(\bigotimes_{v'\in \Sigma_p^+\smallsetminus\{v\}} \sigma_{v'},~ S_{\bbG}(\sfK^v_1\sfK'_v,\bbF)[\fm'_\rbar]\right).$$

\begin{thm}
In the global setting of Sections \ref{global:setup} and \ref{subsec:PF} assume moreover that $\rbar|_{\Gamma_{F^+_v}}$ is tamely ramified and 12-generic for all $v\in \Sigma_p^+$.
Then
$$\dim_{\bbG(F_v^+)}(\pi) = [F^+_v:\qp].$$
\end{thm}

\begin{proof}
The proof closely follows that of \cite[Thm. 8.4.1]{BHHMS}, with the following modifications:
\begin{itemize}
  \item References to \cite[Thm. 6.4.7]{BHHMS} are replaced by Theorem \ref{thm:abstract:GKdim}, whose hypotheses hold by \cite[Cor. 7.5]{koziolmorra}, Corollary \ref{cor:mult:1} and Proposition \ref{prop:freeness:2}.
  \item References to \cite{dottole} are replaced by \cite[\S 7C]{koziolmorra}.
\end{itemize}
These substitutions suffice to prove the result.  
\end{proof}

It is likely that the analogues of \cite[Thm. 8.4.3, Cors. 8.4.4, 8.5.1]{BHHMS} (faithful flatness of ``big'' patched modules, existence of nonzero admissible unitary Banach space representations of $\bbG(F_v^+)=\nU_{1,1}(F_v^+)$, faithful flatness of dual completed cohomology) hold true in our situation, following arguments analogous to those appearing in \cite[\S\S 8.4, 8.5]{BHHMS}.  However, we have decided to not pursue this here.

\appendix

\section{A local lifting lemma}
\label{appA}

In this section, we show that the $L$-parameters which we consider in the body of the paper admit crystalline lifts with Hodge--Tate weights perscribed by the set of Serre weights.

\begin{lemma}
\label{cryslift}
Let $\rhobar: \Gamma_K \longrightarrow {}^C\nU_{1,1}(\bbF)$ be a $3$-generic tamely ramified $L$-parameter which satisfies $\widehat{\imath}\circ \rhobar = \overline{\varepsilon}$.  If $F(\mu)\in \nW^?(\rhobar)$ with $\mu \in X_1(\un{T}_\nU)$, then $\rhobar$ has a semisimple crystalline lift $\rho: \Gamma_K \longrightarrow {}^C\nU_{1,1}(E)$ which satisfies
$$\HT_j(\BC(\rho)) = \mu_j + (1,0),\qquad \HT_{j + f}(\BC(\rho)) = - \fw(\mu_j) + (1,0)$$
for all $0 \leq j \leq f - 1$, and $\widehat{\imath}\circ \rho = \varepsilon$.  
\end{lemma}

\begin{proof}
Write $\rhobar|_{I_K}\cong \overline{\tau}(s,\nu + \eta)$, where $\nu$ is $3$-deep in the fundamental $p$-alcove and let $F(\mu)\in \nW^?(\rhobar)$ be as in the statement of the lemma.  Then $\BC(F(\mu))\in \nW^?(\BC(\rhobar))$ by \cite[Thm. 4.9]{koziolmorra}, where $\BC(\rhobar)$ is defined in Subsection \ref{subsub:Lparamdefs} and the Serre weight $\BC(F(\mu))$ is defined in Subsection \ref{twist-by-epsilon}.  In particular, by equation \eqref{BC-of-inertial-type} and the definition of $\BC(F(\mu))$, we can write $\BC(\rhobar)|_{I_{K_2}}\cong \overline{\tau}'((s,s),\BC(\nu)+\eta')$ and $\BC(F(\mu)) = F(\mu,-{\un{\fw}}(\mu))$.

We now use the results on the extention graph for $\textnormal{GL}_2(k_{K_2})$ from Section \ref{appendix:EGC}.
By Proposition \ref{prop:SW:extgr:U2} and Lemma \ref{lem:comp:ext:gr}, we can write 
$$\BC(F(\mu))=F'\Big(\mathfrak{t}_{\BC(\nu)}\big((s,s)(\overline{\BC}(\omega))\big)\Big)$$ 
for some $\omega\in \Sigma$.  Thus $\BC(\mu)\equiv\mathfrak{t}_{\BC(\nu)}((s,s)(\overline{\BC}(\omega)))$ modulo $(p-\pi')X^0(\underline{T}')$. Using \cite[Lem. 2.4.4]{BHHMS}, we conclude that
$$\left\{F'\left(\mathfrak{t}_{\BC(\nu)}((s,s)(\Sigma'))\right)\right\} = \left\{F'\left(\mathfrak{t}_{\BC(\mu)}((v,v)(\Sigma'))\right)\right\},$$
where $(v,v) := w_{(s,s)\overline{\BC}(\omega)}(s,s)(s_\omega,s_\omega)$, $w_{(s,s)\overline{\BC}(\omega)} \in \un{W}'$ is defined in \cite[Rem. 2.4.5(i)]{BHHMS}, and $s_{\omega}\in \un{W}$ is defined by $s_{\omega,j} = 1$ if $\omega_j = 0$ and $s_{\omega,j} = \fw$ otherwise.  By Proposition \ref{prop:SW:extgr:GL2}, we get that 
$$\nW^?(\BC(\rhobar)) = \nW^?\big(\overline{\tau}'((v,v),\BC(\mu) + \eta')\big)$$ 
which implies, by the $\nG\nL_2$ analog of Corollary \ref{cor:types-cont-weight:3}, that 
$$\BC(\rhobar)|_{I_{K_2}}\cong \overline{\tau}'((v,v),\BC(\mu) + \eta').$$
In particular, by \cite[\S 2.3]{BHHMS} we have $\BC(\rhobar)|_{I_{K_2}}\cong \omega_{2f}^{\mathbf{a}^{(0)}_1}\oplus \omega_{2f}^{\mathbf{a}^{(0)}_2}$ where 
$(\mathbf{a}^{(0)}_1,\mathbf{a}^{(0)}_2) := \sum_{j'=0}^{2f-1}p^{j'}\boldsymbol{\alpha}_{j'}$, with 
$$\boldsymbol{\alpha}_{j'} :=  v_1^{-1} \cdots v_{j'}^{-1}((-{\fw})^{\delta_{j'>f-1}}(\mu_{j'}) + \eta'_{j'})$$
for $j' = 0, \cdots, 2f - 1$, and the indices of $v_{j'}$, $\mu_{j'}$ are taken modulo $f$.  (Recall that $\omega_{2f}:I_{K_2} \longrightarrow \bbF^\times$ denotes the character given as the composition $I_{K_2} \longtwoheadrightarrow \cO_{K_2}^\times \longtwoheadrightarrow k_{K_2}^\times \stackrel{\sigma'_{0}}{\longrightarrow}\bbF^\times$.  We use the same notation to denote the extension of this character to $\Gamma_{K_2}$ which sends $\varphi^{2f}$ to $1$.)

For each $j'$, consider the Lubin--Tate character $\chi_{K_2,j'}:\Gamma_{K_2}\longrightarrow \cO_{K_2}^\times\stackrel{\sigma'_{j'}}{\longrightarrow}\cO^\times$. 
It is a crystalline character which lifts $\omega_{2f}^{p^{j'}}$ and has Hodge--Tate weight $1$ at embedding $j'$ and 0 otherwise.
Writing $\BC(\rhobar)(\varphi^{2f}) = (\xi,\xi^{-1})\in {T_\nG}(\F)$ and letting $\widetilde{\xi} \in \cO^\times$ denote the Teichm\"uller lift of $\xi$, we thus conclude that 
$$\rho'_{\mu} := \left(\prod_{j' = 0}^{2f - 1}\chi_{K_2,j'}^{\boldsymbol{\alpha}_{j',1}}\right)\textnormal{nr}_{2f,\widetilde{\xi}} ~\oplus~ \left(\prod_{j' = 0}^{2f - 1}\chi_{K_2,j'}^{\boldsymbol{\alpha}_{j',2}}\right) \textnormal{nr}_{2f,\widetilde{\xi}^{-1}}$$ 
is a semisimple crystalline representation lifting $\BC(\rhobar)$ whose Hodge--Tate weights satisfy the conditions of the lemma.  Since $\boldsymbol{\alpha}_{j'} = -v_\tau^{-1}\fw(\boldsymbol{\alpha}_{j'-f})+(1,1)$ for all $j' = 0,\dots,2f-1$ (where $v_\tau := v_0v_{f - 1}\cdots v_2v_1$), we see that $(\chi_{K_2,j'}^{\boldsymbol{\alpha}_{j',1}}\oplus \chi_{K_2,j'}^{\boldsymbol{\alpha}_{j',2}})^{\varphi^{f}}\cong \chi_{K_2,{j'-f}}^{-\boldsymbol{\alpha}_{j'-f,1}+ 1}\oplus \chi_{K_2,j'-f}^{-\boldsymbol{\alpha}_{j'-f,2}+1}$ for all $j' = 0,\dots,2f-1$ and hence, taking into account the the choice of the unramified part of $\rho'_{\mu}$, we see that $(\rho'_{\mu})^{\varphi^{f}} \cong {\rho'_{\mu}}^{\vee}\otimes \varepsilon$.  This implies that $\rho'_\mu:\Gamma_{K_2}\longrightarrow \textnormal{GL}_2(\cO)$ descends to a semisimple crystalline representation $\rho_\mu:\Gamma_{K}\longrightarrow {}^C\nU_{1,1}(\cO)$ lifting $\rhobar$ and satisfying the desired conditions.
\end{proof}

\section{A global lifting lemma}
\label{appB}

In this section, we closely follow the arguments of \cite[App. A]{emertongee} in order to show that our local $L$-parameters can be globalized to automorphic global $L$-parameters.

The setup is the following.  Let $p \geq 7$ and let $K$ denote the unramified extension of $\bbQ_p$ of degree $f$.  We fix throughout this appendix a continuous $L$-parameter $\varrhobar: \Gamma_K \longrightarrow \cG_2(\bbF)$ satisfying $\varrhobar^{-1}(\nG\nL_2(\bbF) \times \bbF^\times) = \Gamma_{K_2}$ and $\nu\circ\varrhobar = \overline{\varepsilon}^{-1}$.

We first globalize $\varrhobar$.  

\begin{lemma}
There exists a CM field $L$ with maximal totally real subfield $L^+$, and a continuous $L$-parameter $\rbar: \Gamma_{L^+} \longrightarrow \cG_2(\bbF)$ satisfying the following:
\begin{itemize}
  \item $p$ is unramified in $L$, and every place $v$ of $L^+$ above $p$ is \textbf{inert} in $L$;
  \item if $v$ is a place of $L^+$ above $p$, then $L^+_v \cong K$, $L_v \cong K_2$, and $\rbar|_{\Gamma_{L^+_v}} \cong \varrhobar$;
  \item $\nu\circ \rbar = \overline{\varepsilon}^{-1}$;
  \item $\rbar^{-1}(\nG\nL_2(\bbF) \times \bbF^\times) = \Gamma_L$;
  \item $\rbar(\Gamma_{L(\zeta_p)}) = \nG\nL_2(\bbF) \times \{1\}$ (in particular, $\rbar(\Gamma_{L(\zeta_p)})$ is adequate, cf. \cite[Prop. 6.5]{BLGG});
  \item $\overline{L}^{\ker(\textnormal{ad}^0(\rbar))}$ does not contain $L(\zeta_p)$;
  \item if $v\nmid p$ is a finite place of $L^+$, then $\rbar|_{\Gamma_{L^+_v}}$ is unramified;
  \item $L/L^+$ is unramified at all finite places.
\end{itemize}
\end{lemma}

\begin{proof}
  Everything but the last bullet point follows exactly as in the proof of \cite[Prop. A.2]{emertongee}.   Indeed, the proof of \cite[Prop. A.2]{emertongee}, as written, produces fields $L/L^+$ satisfying bullet points 1,2,3,4,5,7 above, and one may check that the sixth bullet point ``$\overline{L}^{\ker(\textnormal{ad}^0(\rbar))}$ does not contain $L(\zeta_p)$'' holds using an argument analogous to \cite[pf. of Prop. 5.10, p. 117]{gee:liftingthms}.  (We remark that the statement of \cite[Prop. A.2]{emertongee} seems to contain an incompatibility when $n = 2$.  Namely, supposing that the third point of \emph{op.~cit}.~is satisfied, we then obtain in the notation of that article that $\overline{\varepsilon}(\gamma)^{-1} = (\nu \circ \rhobar)(\gamma) = 1$ for every $\gamma \in \ker(\rhobar)$.  Thus $\overline{L}^{\ker(\rhobar)} \supset \overline{L}^{\ker(\overline{\varepsilon}|_{\Gamma_L})} = L(\zeta_p)$, contradicting the sixth point of \emph{op.~cit.}.  We believe that the correct statement should be ``$\overline{L}^{\ker(\textnormal{ad}(\rhobar|_{\Gamma_L}))}$ does not contain $L(\zeta_p)$,'' as in \cite[\S 5.1]{emertongee}.)

  To see that the last condition may also be achieved, suppose $L_0/L_0^+$ satisfies all but the last bullet point.  We then apply \cite[Lem. 4.1.2]{CHT}: their $F$ is our $L_0^+$, their $D$ is our $\overline{L_0}^{\ker(\rbar)}(\zeta_p)$, their $S$ is our places of $L_0^+$ above $p$ and $\infty$ and places of $L_0^+$ which ramify in $L_0$.  For $v$ above $p$ or $\infty$, we take $E_v' = (L_0^+)_v$, and if $v$ ramifies in $L_0$, then we take $E_v' = (L_0)_v$.  The cited lemma furnishes the desired (totally real) $L^+$, and we take $L := L^+L_0$.
\end{proof}

Next, we produce a lift of $\rbar$ to characteristic 0.  Let us define a global deformation problem $\cS$ as follows.  Let $\Sigma_p^+$ denote the places of $L^+$ above $p$, and $\Sigma_p$ the places of $L$ above $p$ (so that restriction to $L^+$ induces a bijection between $\Sigma_p$ and $\Sigma_p^+$).  Further, fix $\mu = (\mu_{j'})_{j' \in \cJ' = \Hom(K_2,E)} \in (\bbZ^{2})^{2f}$ satisfying
\begin{equation}
  \label{FLineqs}
  \mu_{j',2} + p > \mu_{j',1} > \mu_{j',2}\qquad \textnormal{and} \qquad \mu_{j' + f,i} = 1 - \mu_{j',3 - i}
\end{equation}
for all $j'$.  We define $R^{\Box, \mu, \mathbf{1}}_{\varrhobar}$ to be the unique $\cO$-flat quotient of $R^{\Box}_{\varrhobar}$ with the property that if $B$ is a finite local $E$-algebra, then $x : R^{\Box}_{\varrhobar} \longrightarrow B$ factors through $R^{\Box, \mu, \mathbf{1}}_{\varrhobar}$ if and only if the corresponding representation $\varrho_x : \Gamma_{K}\longrightarrow  \cG_2(B)$ is crystalline, satisfying $\nu \circ \varrho_x = \varepsilon^{-1}$ and $\textnormal{HT}_{j'}(\BC'(\varrho_x)) = \mu_{j'}$.  (Note that in this appendix we are using a different notation for $\cG_2$-valued deformation rings than in the body of the paper.)  We also \textbf{assume} that $\varrhobar$ admits such a crystalline lift which is moreover semisimple, so that in particular we have $R^{\Box,\mu,\mathbf{1}}_{\varrhobar} \neq 0$.  We note that this implies $\BC'(\rbar)|_{\Gamma_{L_v}}$ has a potentially diagonalizable lift with distinct Hodge--Tate weights for every $v \in \Sigma_p$.

We define $\cS$ as
$$\cS := \left(L/L^+,~ \Sigma_p^+,~ \Sigma_p,~ \cO,~ \rbar,~ \varepsilon^{-1},~ \{R^{\Box, \mu, \mathbf{1}}_{\varrhobar}\}_{v \in \Sigma_p^+}\right).$$
This is a slight generalization of the global deformation problems considered in \cite[\S 2.3]{CHT}, since we are allowing places in $\Sigma_p^+$ to be inert in $L$.  Thus, a deformation is of type $\cS$ if it lifts $\rbar$, and locally at $p$ it is a crystalline lift of $\varrhobar$ with Hodge--Tate weights determined by $\mu$.

\begin{lemma}
  After possibly enlarging $E$, the deformation ring $R_{\cS}^{\textnormal{univ}}$ admits an $E$-valued point.  Thus, there exists a lift $r:\Gamma_{L^+} \longrightarrow \cG_2(\cO)$ of $\rbar$ satisfying:
  \begin{itemize}
    \item $\nu\circ r = \varepsilon^{-1}$;
    \item $r$ is unramified outside of places dividing $p$; and
    \item if $v$ is a place of $L^+$ above $p$, then $r|_{\Gamma_{L^+_v}}$ is crystalline with Hodge--Tate weights given by $\mu$. 
  \end{itemize}
\end{lemma}

\begin{proof}
To construct the desired lift, we define a second deformation problem.  Let $M^+$ denote a finite, solvable, totally real Galois extension of $L^+$ satisfying the following properties:
\begin{itemize}
  \item $p$ is unramified in $M := M^+L$ and every place $w$ of $M^+$ above $p$ is \textbf{split} in $M$;
  \item if $w$ is a place of $M^+$ above $p$, then $M^+_w \cong K_2$;
  \item $\rbar(\Gamma_{M^+}) = \rbar(\Gamma_{L^+})$;
  \item $\overline{M}^{\ker(\textnormal{ad}^0(\rbar|_{\Gamma_{M^+}}))}$ does not contain $M(\zeta_p)$.
\end{itemize}
(To see that such an extension $M^+$ exists, we may apply \cite[Lem. 4.1.2]{CHT}: their $F$ is our $L^+$, their $D$ is our $\overline{L}^{\ker(\rbar)}(\zeta_p)$, and their $S$ is our $\Sigma_p^+ \cup \{v|\infty\}$.  For $v$ an infinite place of $L^+$, we take $E_v' = \bbR$, and for $v$ a place of $L^+$ above $p$, we take $E_v'$ to be the unramified extension of $L^+_v \cong K$ of degree 2.)

Let $\Sigma_{M,p}^+$ denote the set of places of $M^+$ above $p$, and for each $w \in \Sigma_{M,p}^+$, we choose a place $\widetilde{w}$ of $M$ lying above $w$, and let $\widetilde{\Sigma_{M,p}^+}$ denote the set of these places.  For each $w \in \Sigma_{M,p}^+$, we let $\BC_w(\mu)$ denote the element of $(\bbZ_+^{2})^{\Hom(M_{\widetilde{w}},E)}$ defined by $\BC_w(\mu)_{j'} = \mu_{j'|_{K_2}}$ for $j' \in \Hom(M_{\widetilde{w}},E)$.  We then consider the local deformation ring $R_{\rbar|_{\Gamma_{M_{\widetilde{w}}}}}^{\Box,\BC_w(\mu),\mathbf{1}}$ parametrizing lifts of $\rbar|_{\Gamma_{M_{\widetilde{w}}}}$ which are crystalline with Hodge--Tate weights given by $\BC_w(\mu)$.

We define $\cS_M$ as 
$$\cS_M := \left(M/M^+,~ \Sigma_{M,p}^+,~ \widetilde{\Sigma_{M,p}^+},~ \cO,~ \rbar|_{\Gamma_M^+},~ \varepsilon^{-1},~ \{R_{\BC'(\rbar)|_{\Gamma_{M_{\widetilde{w}}}}}^{\Box,\BC_w(\mu),\mathbf{1}}\}_{w \in \Sigma_{M,p}^+}\right).$$
The map $r \longmapsto r|_{\Gamma_{M^+}}$ induces a morphism between deformations of type $\cS$ and deformations of type $\cS_M$.  This in turn gives a morphism between universal framed deformation rings
$$R_{\cS_M}^\textnormal{univ} \longrightarrow R_{\cS}^\textnormal{univ}.$$
By \cite[Lem. 1.2.3(1)]{BLGGT}, this morphism makes $R_{\cS}^\textnormal{univ}$ into a finite $R_{\cS_M}^\textnormal{univ}$-module.

We now apply the proof of \cite[Thm. 4.3.1]{BLGGT} to the representation $\rbar|_{\Gamma_{M^+}}$.  Namely, for each $v \in \Sigma_v^+$, let $\varrho_v:\Gamma_{L^+_v} \longrightarrow \cG_2(\cO)$ denote a semisimple crystalline lift of $\rbar|_{\Gamma_{L^+_v}}$ with Hodge--Tate weights given by $\mu$.  Then $\BC'(\varrho_v)|_{\Gamma_{M_{\widetilde{w}}}}$ is a semisimple crystalline lift of $\BC'(\rbar|_{\Gamma_{M^+}})|_{\Gamma_{M_{\widetilde{w}}}}$.  Since the Hodge--Tate weights given by $\mu$ are in the Fontaine--Laffaille range, \cite[Prop. 3.6]{KW-II} implies that $R_{\BC'(\rbar)|_{\Gamma_{M_{\widetilde{w}}}}}^{\Box,\BC_w(\mu),\mathbf{1}}$ is an integral domain.  Therefore, the proof of \cite[Thm. 4.3.1]{BLGGT} shows that $R_{\cS_M}^{\textnormal{univ}}$ is a finite $\cO$-module (taking $\cC_v$ to be all of $\textnormal{Spec}(R_{\BC'(\rbar)|_{\Gamma_{M_{\widetilde{w}}}}}^{\Box,\BC_w(\mu),\mathbf{1}})$).  Consequently, $R_\cS^{\textnormal{univ}}$ is a finite $\cO$-module, and \cite[Cor. 2.3.5]{CHT} along with \cite[Thm. 3.3.7]{bellovingee} implies that $R_{\cS}^{\textnormal{univ}}$ has Krull dimension at least 1.  We therefore obtain a morphism $R_{\cS}^{\textnormal{univ}} \longrightarrow E$ (after possibly extending $E$), which gives a lift $r:\Gamma_{L^+} \longrightarrow \cG_2(\cO)$ of $\rbar$ of type $\cS$.  
\end{proof}

Our next step is to find an extension of $L$ over which $\BC'(r)$ becomes automorphic.

\begin{lemma}
  There exists an finite Galois extension of totally real fields $F^+/L^+$ such that such that 
  \begin{itemize}
    \item $F:= F^+L$ is linearly disjoint from $\overline{L}^{\ker(\rbar)}(\zeta_p)$;
    \item every place of $L^+$ above $p$ splits in $F^+$; and
    \item $(\BC'(r)|_{\Gamma_F},\varepsilon^{-1})$ is automorphic (in the notation of \cite[\S 2.1]{BLGGT}).
  \end{itemize}
\end{lemma}
\noindent (We note that the splitness assumption implies that every place $v$ of $F^+$ above $p$ is \textbf{inert} in $F$.)

\begin{proof}
  We may apply the proof of \cite[Prop. A.6]{emertongee} verbatim; one can check that the relevant constructions do not depend on the splitting behavior of $L/L^+$ above $p$.  We remark that the reference to Lemma 2.1.1 of \cite{BLGGT} which appears in the proof of \cite[Prop. A.6]{emertongee} refers to the first arXiv version of \cite{BLGGT}, and corresponds to Lemma 2.2.4 in the published version of \cite{BLGGT}.  
\end{proof}

\begin{cor}
\label{appendix:maincor}
Let $p \geq 7$ and let $K$ denote the unramified extension of $\bbQ_p$ of degree $f$.  Let $\varrhobar: \Gamma_K \longrightarrow \cG_2(\bbF)$ denote a continuous $L$-parameter satisfying $\varrhobar^{-1}(\nG\nL_2(\bbF) \times \bbF^\times) = \Gamma_{K_2}$ and $\nu\circ\varrhobar = \overline{\varepsilon}^{-1}$.  Assume further that $\varrhobar$ admits a crystalline lift, with Hodge--Tate weights satisfying \eqref{FLineqs}.

 Then there exists a CM field $F$ with maximal totally real subfield $F^+$ and a continuous $L$-parameter $\rbar:\Gamma_{F^+} \longrightarrow \cG_2(\bbF)$ satisfying the following:
  \begin{itemize}
    \item $p$ is unramified in $F$, and every place $v$ of $F^+$ above $p$ is \textbf{inert} in $F$;
    \item if $v$ is a place of $F^+$ above $p$, then $F^+_v \cong K$, $F_v \cong K_2$ and $\rbar|_{\Gamma_{F^+_v}} \cong \varrhobar$;
    \item $\nu \circ \rbar = \overline{\varepsilon}^{-1}$;
    \item $\rbar^{-1}(\nG\nL_2(\bbF) \times \bbF^{\times}) = \Gamma_F$;
    \item $\rbar(\Gamma_{F(\zeta_p)}) = \nG\nL_2(\bbF) \times \{1\}$ (in particular, $\rbar(\Gamma_{F(\zeta_p)})$ is adequate);
    \item $\overline{F}^{\ker(\textnormal{ad}^0(\rbar))}$ does not contain $F(\zeta_p)$;
    \item if $v \nmid p$ is a finite place of $F^+$, then $\rbar|_{\Gamma_{F^+_v}}$ is unramified;
    \item $F/F^+$ is unramified at all finite places;
    \item $(\BC'(\rbar),\overline{\varepsilon}^{-1})$ is automorphic.
  \end{itemize}
\end{cor}

\begin{cor}
\label{appendix:cor2}
  Let $\bbG_{/\cO_{F^+}}$ denote the rank 2 unitary group over $\cO_{F^+}$ constructed in Section \ref{global:setup} (see also \cite[\S 6A]{koziolmorra}).  Then there exists a sufficiently small compact open subgroup $\sfK = \prod_v \sfK_v \subset \bbG(\bbA_{F^+}^\infty)$ and a finite set $T$ of places (including all inert places such that $\sfK_v$ is not hyperspecial and all split places such that $\sfK_v \neq \bbG(\cO_{F^+_v})$) for which
  $$S_{\bbG}\left(\sfK,V(\mu + (\underline{0},\underline{1}))^{\textnormal{d}}\right)_{\fm_{\rbar}}\neq 0,$$
  where $\fm_{\rbar} \subset \bbT^T$ is the maximal ideal associated to $\rbar$ (as in \cite[Def. 6.5]{koziolmorra}), and where $\mu$ denotes the Hodge--Tate weights of a crystalline lift of $\varrhobar$ satisfying \eqref{FLineqs}.
\end{cor}

\begin{proof}
We put ourselves in the situation of Corollary \ref{appendix:maincor}.  In particular, we apply \cite[Thm. 11.5.1]{rogawski} and the Jacquet--Langlands correspondence (as in the last paragraph in the proof of \cite[Prop. 7.2]{koziolmorra}) to any automorphic representation $\Pi$ of $\nG\nL_2(\bbA_F)$ for which $\overline{r_{\imath}(\Pi)} \cong \BC'(\rbar)$.  This procedure gives a compact open subgroup $\sfK = \prod_v \sfK_v$ of $\bbG(\bbA_{F^+}^\infty)$ (which we may shrink if necessary) and an automorphic representation $\pi$ of $\bbG(\bbA_{F^+})$ which, by local/global compatibility, contributes to the space 
$$S_{\bbG}\left(\sfK, V(\mu + (\underline{0},\underline{1}))^{\textnormal{d}}\right)_{\fm_{\rbar}}.$$
In particular, this implies that 
$$S_{\bbG}\left(\sfK, V(\mu + (\underline{0},\underline{1}))^{\textnormal{d}}\right)_{\fm_{\rbar}} \neq 0.$$
\end{proof}

We shall apply the previous results in the body of the paper using the following lemma.

\begin{cor}
  \label{main-global-lift-cor}
Suppose $\rhobar:\Gamma_K \longrightarrow {}^C\nU_{1,1}(\bbF)$ is a $3$-generic $L$-parameter valued in the $C$-group, and satisfying $\widehat{\imath}\circ \rhobar = \omega$.  Assume further that we have chosen $\lambda = (\lambda_{j'}) \in \{(2,-1),~(1,0)\}^{2f}$ which satisfies
$$\lambda_{j' + f} = \lambda_{j'} \qquad \textnormal{for all $j'$},$$
and $\tau'$ is a tame inertial type of $I_{K_2}$ satisfying $(\tau')^{\varphi^f} \cong \tau'^{\vee}$.  If
$$\JH\left(\overline{\sigma(\tau')\otimes_E \bigotimes_{E,0 \leq j' \leq f - 1} V_E(\lambda_{j'}-(1,0))^{(j')}}\right) \cap \nW^?(\rhobar) \neq \emptyset,$$
then $\rhobar$ has a potentially crystalline lift $\rho$ with inertial type $\tau'$, which satisfies $\HT_{j'}(\BC(\rho)) = \lambda_{j'}$ for all $j' \in \cJ'$, and $\widehat{\imath}\circ \rho = \varepsilon$.
\end{cor}

\begin{proof}
Suppose $F(\mu)$ is a Serre weight contained in the intersection
$$\JH\left(\overline{\sigma(\tau')\otimes_E \bigotimes_{E,0 \leq j' \leq f - 1} V_E(\lambda_{j'} - (1,0))^{(j')}}\right) \cap \nW^?(\rhobar),$$
where $\mu \in X^*(\underline{T}_\nU)$.  By Lemma \ref{cryslift}, $\rhobar$ admits a crystalline lift $\rho$ with Hodge--Tate weights given by $\mu + (\underline{1}, \underline{0})$.  In particular, since $\rhobar$ is 3-generic, Proposition \ref{prop:SW:extgr:U2} implies that $\mu$ lies $2$-deep in the fundamental $p$-alcove, and therefore $\mu + (\underline{1}, \underline{0})$ satisfies the conditions \eqref{FLineqs}.

By composing $\rhobar$ and $\rho$ with the isomorphism ${}^C\nU_{1,1} \cong \cG_2$, we obtain $L$-parameters $\varrhobar$ and $\varrho$ valued in $\cG_2$, and we may put ourselves in the setting of Corollary \ref{appendix:maincor}.  Note that $\HT_{j'}(\BC(\rho)) = \mu_{j'} + (1,0)$ while $\HT_{j'}(\BC'(\varrho)) = \mu_{j'} + (1,0) - (1,1) = \mu_{j'} - (0,1)$ (so that $\mu - (\underline{0},\underline{1})$ also satisfies the inequalities \eqref{FLineqs}).  By Corollary \ref{appendix:cor2}, we have a sufficiently small compact open subgroup $\sfK = \prod_v \sfK_v \subset \bbG(\bbA_{F^+}^\infty)$ for which
  $$S_{\bbG}\left(\sfK,V(\mu)^{\textnormal{d}}\right)_{\fm_{\rbar}}\neq 0.$$
By \cite[Lem. 6.3]{koziolmorra}, this implies
  $$S_{\bbG}\left(\sfK,V(\mu)^{\textnormal{d}}\otimes_{\cO}\bbF\right)_{\fm_{\rbar}} \cong S_{\bbG}\left(\sfK,F(\mu)^\vee\right)_{\fm_{\rbar}}\neq 0.$$
Since $F(\mu)^\vee$ is a Jordan--H\"older factor of $\left(\overline{\sigma(\tau')\otimes_E \bigotimes_{E,0 \leq j' \leq f - 1} V_E(\lambda_{j'}-(1,0))^{(j')}}\right)^\vee$, and the functor of algebraic automorphic forms is exact, we obtain
$$S_{\bbG}\left(\sfK,\left(\overline{\sigma(\tau')\otimes_E \bigotimes_{E,0 \leq j' \leq f - 1} V_E(\lambda_{j'}-(1,0))^{(j')}}\right)^\vee\right)_{\fm_{\rbar}} \neq 0.$$
Another application of \cite[Lem. 6.3]{koziolmorra} implies
$$S_{\bbG}\left(\sfK,\left(\sigma(\tau')\otimes_E \bigotimes_{E,0 \leq j' \leq f - 1} V_E(\lambda_{j'}-(1,0))^{(j')}\right)^\vee\right)_{\fm_{\rbar}} \neq 0.$$
Using \cite[Thm. 6.1]{koziolmorra}, this gives rise to a global Galois representation $r_{\imath}(\pi):\Gamma_F \longrightarrow \nG\nL_2(\overline{E})$.  We then extend $r_{\imath}(\pi)$ to an $L$-parameter $\Gamma_{F^+} \longrightarrow \cG_2(\overline{E})$, apply the isomorphism $\cG_2 \cong {}^C\nU_{1,1}$, and restrict to $\Gamma_{F^+_v}$ for any place $v$ of $F^+$ above $p$.  By local/global compatibility, the resulting $L$-parameter has the desired properties.
\end{proof}

\bibliographystyle{amsalpha}
\bibliography{refs-new}

\end{document}